\documentclass{amsart}
\usepackage{enumitem}
\usepackage{amssymb}
\usepackage{amsthm}
\usepackage{amsmath}
\usepackage{caption}
\usepackage{graphicx}
\usepackage[all]{xy}

\usepackage[cal=boondoxo]{mathalfa}

\usepackage[usenames]{xcolor}
\usepackage{tikz}
\usepackage{svg}

\usepackage[colorlinks=true, pdfstartview=FitV, linkcolor=cyan, citecolor=magenta, urlcolor=blue]{hyperref}

\theoremstyle{plain}
\newtheorem{proposition}{Proposition}[section]
\newtheorem{theorem}[proposition]{Theorem}
\newtheorem*{theorem*}{Theorem}
\newtheorem*{conjecture*}{Conjecture}
\newtheorem{lemma}[proposition]{Lemma}
\newtheorem{corollary}[proposition]{Corollary}
\theoremstyle{definition}

\newtheorem{definition}[proposition]{Definition}

\newtheorem{introthm}{Theorem}
\newtheorem{introcor}[introthm]{Corollary}

\theoremstyle{remark}
\newtheorem{remark}[proposition]{Remark}

\numberwithin{equation}{section}

\let\oldtocsection=\tocsection

\let\oldtocsubsection=\tocsubsection

\let\oldtocsubsubsection=\tocsubsubsection

\renewcommand{\tocsection}[2]{\hspace{0em}\oldtocsection{#1}{#2}}
\renewcommand{\tocsubsection}[2]{\hspace{1em}\oldtocsubsection{#1}{#2}}
\renewcommand{\tocsubsubsection}[2]{\hspace{2em}\oldtocsubsubsection{#1}{#2}}

\DeclareMathOperator{\SL}{\mathsf{SL}}
\DeclareMathOperator{\GL}{\mathsf{GL}}

\DeclareMathOperator{\PGL}{\mathsf{PGL}}
\DeclareMathOperator{\Hom}{Hom}

\DeclareMathOperator{\End}{End}

\DeclareMathOperator{\id}{id}

\def\co{\colon\thinspace} 

\DeclareMathOperator{\Ac}{\mathcal{A}}
\DeclareMathOperator{\Bc}{\mathcal{B}}

\DeclareMathOperator{\Fc}{\mathcal{F}}

\DeclareMathOperator{\Rc}{\mathcal{R}}
\DeclareMathOperator{\Tc}{\mathcal{T}}

\DeclareMathOperator{\Yc}{\mathcal{Y}}

\DeclareMathOperator{\Cb}{\mathbb{C}}

\DeclareMathOperator{\Rb}{\mathbb{R}}

\DeclareMathOperator{\Zb}{\mathbb{Z}}

\DeclareMathOperator{\bsf}{\mathsf{b}}
\DeclareMathOperator{\csf}{\mathsf{c}}

\DeclareMathOperator{\gsf}{\mathsf{g}}

\DeclareMathOperator{\ksf}{\mathsf{k}}

\DeclareMathOperator{\tsf}{\mathsf{t}}

\newcommand{\abs}[1]{\left|#1\right|}

\newcommand{\wt}[1]{\widetilde{#1}}
\newcommand{\wh}[1]{\widehat{#1}}

\newcommand{\vertiii}[1]{{\left\vert\kern-0.25ex\left\vert\kern-0.25ex\left\vert #1 
    \right\vert\kern-0.25ex\right\vert\kern-0.25ex\right\vert}}
	
	\newcounter{notes}

\DeclareSymbolFont{extraup}{U}{zavm}{m}{n}
\DeclareMathSymbol{\varheart}{\mathalpha}{extraup}{86}
\DeclareMathSymbol{\vardiamond}{\mathalpha}{extraup}{87}

\begin{document}

\title[Topology of the space of $d$-pleated surfaces]{Topology of the space of $d$-pleated surfaces}

\author[Maloni S.]{Sara Maloni}
\address{Department of Mathematics, University of Virginia}
\email{sm4cw@virginia.edu}
\urladdr{sites.google.com/view/sara-maloni/}

\author[Martone G.]{Giuseppe Martone}
\address{Department of Mathematics and Statistics, Sam Houston State University}
\email{gxm120@shsu.edu}
\urladdr{sites.google.com/view/giuseppemartone}

\author[Mazzoli F.]{Filippo Mazzoli}
\address{Department of Mathematics, University of California Riverside}
\email{filippo.mazzoli@ucr.edu}
\urladdr{filippomazzoli.github.io/index.html}

\author[Zhang T.]{Tengren Zhang}
\address{Department of Mathematics, National University of Singapore}
\email{matzt@nus.edu.sg}
\urladdr{sites.google.com/site/tengren85/}

\thanks{S.M. and F.M. were partially supported by U.S. National Science Foundation grant DMS-1848346 (NSF CAREER). F.M. acknowledges support from the European Research Council (ERC) under the European Union’s Horizon 2020 research and innovation programme (grant agreement No 101018839). G.M. acknowledges partial support by an American Mathematical Society (AMS)-Simons Research Enhancement Grant for PUI Faculty. T.Z. was partially supported by the NUS-MOE grant A-8001950-00-00 and A-8000458-00-00. The authors also acknowledge support from the GEAR Network, funded by the National Science Foundation under grant numbers DMS 1107452, 1107263, and 1107367 (``RNMS: GEometric structures And Representation varieties") and from the Institut Henri Poincar\'e (UAR 839 CNRS-Sorbonne Universit\'e) and LabEx CARMIN (ANR-10-LABX-59-01).} 

\date{\today}

\begin{abstract} Given a maximal geodesic lamination $\lambda$ on a closed oriented surface $S$ of genus $g$, the space of $d$--pleated surfaces with pleating locus $\lambda$ is an open subset of $\Hom(\pi_1(S),\PGL_d(\Cb))$ obtained by applying generalized bending along $\lambda$ to Hitchin representations. When $d=2$, one recovers abstract pleated surfaces in $\mathbb H^3$. In this paper, we study the topology of the space $\mathfrak R(\lambda,d)$ of conjugacy classes of $d$--pleated surfaces with pleating locus $\lambda$. Firstly, we prove that $\mathfrak R(\lambda,d)$ is real-analytically diffeomorphic to $\Rb^{(d^2-1)(2g-2)}\times(\Rb/2\pi\Zb)^{(d^2-1)(2g-2)}\times \Zb_d$, where $\Zb_d$ denotes the finite cyclic group of order $d$. Furthermore, we show that each connected component of the space of conjugacy classes in $\Hom(\pi_1(S),\PGL_d(\Cb))$ contains exactly one component of $\mathfrak R(\lambda,d)$.
\end{abstract}

\maketitle

\tableofcontents

\section{Introduction}\label{introduction}

Let $S$ be a closed, connected, oriented hyperbolic surface, and let $\Gamma$ denote its fundamental group. Given a maximal geodesic lamination $\lambda$ on $S$ and any integer $d\geq 2$, in \cite{MMMZ1} we introduced the notion of a \emph{$d$--pleated surface with pleating locus $\lambda$}. These are representations in $\Hom(\Gamma,\PGL_d(\Cb))$ which satisfy an Anosov contraction property along the leaves of $\lambda$ and a hyperconvexity property along the plaques of $\lambda$. When $d=2$, the notion of $2$--pleated surfaces coincides with that of (abstract) pleated surfaces described by Thurston \cite{thurston-notes} and Bonahon \cite{bonahon-toulouse} as bending deformations of hyperbolic surfaces. For general $d$, a similar interpretation holds: in \cite{MMMZ1} we defined a family of explicit deformations of $\PGL_d(\Rb)$--Hitchin representations into $\PGL_d(\Cb)$, called \emph{generalized bending along $\lambda$}, and showed that in this way we obtain exactly the set of $d$--pleated surfaces with pleating locus $\lambda$.

In this article, we focus on the space $\mathfrak R(\lambda,d)$ of conjugacy classes of $d$--pleated surfaces with pleating locus $\lambda$. Our two main results are as follows. First, we prove that the space $\mathfrak R(\lambda,d)$ has $d$ connected components, each of which is real-analytically diffeomorphic to the product 
\[\Rb^{(d^2-1)(2g-2)}\times(\Rb/2\pi\Zb)^{(d^2-1)(2g-2)},\] 
where $g$ is the genus of $S$. Second, we show that each connected component of the space $\mathfrak X(\Gamma,\PGL_d(\Cb))$ of conjugacy classes in $\Hom(\Gamma,\PGL_d(\Cb))$ contains exactly one component of $\mathfrak R(\lambda,d)$. 

\medskip

We will now describe our main results and the objects therein in greater detail. Given a maximal geodesic lamination $\lambda$ on $S$, let $T^1\lambda\subset T^1S$ denote the set of vectors that are tangent to the leaves of $\lambda$, let $\widetilde\lambda$ denote the lift of $\lambda$ to the universal cover $\widetilde S$ of $S$, and let $\partial\widetilde\lambda\subset\partial\widetilde S$ denote the set of endpoints of the leaves of $\widetilde\lambda$.

If $\rho:\Gamma\to\PGL_d(\Cb)$ is a representation that admits a $\rho$--equivariant, $\lambda$--continuous, $\lambda$--transverse map $\xi:\partial\widetilde\lambda\to\mathcal F(\Cb^d)$ (see Section \ref{d-pleated} for precise definitions), one can construct for each integer in $\{1,\dots,d-1\}$, a certain line bundle over $T^1\lambda$ that admits a flow which covers the geodesic flow on $T^1\lambda$, and is linear when restricted to the fibers. We say that $\rho$ is \emph{$\lambda$--Borel Anosov} if it admits a $\rho$--equivariant, $\lambda$--transverse, $\lambda$--continuous map $\xi:\partial\widetilde\lambda\to\mathcal F(\Cb^d)$, and the induced line bundles over $T^1\lambda$ are uniformly contracted by the flow, see \cite[Section 3.1]{MMMZ1} for more details. If $\rho$ is a $\lambda$--Borel Anosov representation, then the map $\xi$ is uniquely determined by $\rho$, and so we refer to it as the \emph{$\lambda$--limit map} of $\rho$. 

A pair $(\rho,\xi)$ is a \textit{$d$--pleated surface with pleating locus $\lambda$} if $\rho:\Gamma\to\PGL_d(\Cb)$ is a $\lambda$--Borel Anosov representation with $\lambda$--limit map $\xi:\partial\wt\lambda\to\Fc(\Cb^d)$, and $\xi$ is \emph{$\lambda$--hyperconvex}, i.e. for every plaque of $\widetilde\lambda$, the triple of flags assigned by $\xi$ to the vertices of the plaque are in general position. By the uniqueness of $\lambda$--limit maps, we may embed the set of $d$--pleated surfaces with pleating locus $\lambda$ onto an open subset of $\Hom(\Gamma,\PGL_d(\Cb))$ that is invariant under conjugation, and avoids the singular locus \cite[Section 3.2]{MMMZ1}.  As such, the set $\mathfrak R(\lambda,d)$ of conjugacy classes of $d$--pleated surfaces with pleating locus $\lambda$ is naturally an open subset of $\mathfrak X(\Gamma,\PGL_d(\Cb))$. It also follows from \cite[Theorem A]{MMMZ1} that $\mathfrak R(\lambda,d)$ is naturally a complex manifold.

Examples of $d$--pleated surfaces include the \emph{$\PGL_d(\Rb)$--Hitchin representations}. These are representations from $\Gamma$ to $\PGL_d(\Rb)$ that can be continuously deformed to a representation of the form $\iota\circ j$, where $j:\Gamma\to\PGL_2(\Rb)$ is some Fuchsian (i.e. discrete and faithful) representation, and $\iota:\PGL_2(\Rb)\to\PGL_d(\Rb)$ is some irreducible representation. They are central objects studied in Higher Teichm\"uller theory. It follows from the seminal results of Labourie \cite{labourie-anosov} that the $\PGL_d(\Rb)$--Hitchin representations are $d$--pleated surfaces with any pleating locus. Furthermore, Hitchin \cite{hit_lie} (also see Bonahon-Dreyer \cite{BoD}) proved that the set ${\rm Hit}_d(S)$ of conjugacy classes of $\PGL_d(\Rb)$--Hitchin representations is homeomorphic to $\Rb^{(d^2-1)(2g-2)}$. As a consequence of \cite{MMMZ1}, the Bonahon-Dreyer parametrization of ${\rm Hit}_d(S)$ is a real analytic map.

The first main result of this article describes $\mathfrak R(\lambda,d)$ as a real analytic manifold.
\begin{introthm} \label{ThmA} 
	As real analytic manifolds, 
	\[\mathfrak R(\lambda,d)\cong\Rb^{(d^2-1)(2g-2)}\times (\Rb/2\pi \Zb)^{(d^2-1)(2g-2)}\times\Zb_d,\]
	where $\Zb_d$ is the cyclic group of order $d$. In particular, $\mathfrak R(\lambda,d)$ has $d$ connected components.
\end{introthm}
When $d=2$, Theorem \ref{ThmA} is due to Bonahon \cite{bonahon-toulouse}. 

It is a classical result of Li \cite{Li} that $\mathfrak X(\Gamma,\sf{PGL}_d(\Cb))$ also has $d$ connected components. Our second main result is the following:
\begin{introthm} \label{ThmB}
	Each connected component of $\mathfrak X(\Gamma,\sf{PGL}_d(\Cb))$ contains exactly one connected component of $\mathfrak R(\lambda,d)$.
\end{introthm}

Theorem \ref{ThmB} implies that for fixed $\lambda$, one can reach every connected component of $\mathfrak X(\Gamma,\sf{PGL}_d(\Cb))$ by applying generalized bending deformations along $\lambda$ to the Hitchin representations. One can interpret this as a first step towards the following conjecture.

\begin{conjecture*}
	For every maximal geodesic lamination $\lambda$, the space $\mathfrak R(\lambda,d)$ of conjugacy classes of $d$--pleated surfaces with pleating locus $\lambda$ is dense in $\mathfrak X(\Gamma,\PGL_d(\Cb))$.
\end{conjecture*}

In other words, we conjecture that for every $\lambda$, one can approximate every representation from $\Gamma$ to $\PGL_d(\Cb)$ arbitrarily well by generalized bending deformations along $\lambda$ of Hitchin representations. 
\medskip

\noindent{\bf A parameterization of $\mathfrak R(\lambda,d)$.}\,
The proofs of both Theorem \ref{ThmA} and Theorem \ref{ThmB} rely heavily on the main parameterization theorem we proved in \cite{MMMZ1}, which we will now briefly describe. 

Let $\Ac$ denote the set of pairs of positive integers that sum to $d$ and let $\Bc$ denote the set of triples of positive integers that sum to $d$. Denote by $\wt\Delta^{2*}$ the set of distinct pairs of plaques of $\wt\lambda$ and by $\wt\Delta^o$ the set of plaques of $\wt\lambda$ equipped with an ordering of their vertices. For any Abelian group $G$, a $\lambda$--\textit{cocyclic pair of dimension $d$ with values in $G$} is a pair $(\alpha, \theta)$ of $\Gamma$--invariant maps
$$\alpha \co \wt\Delta^{2*} \times \mathcal{A} \to G, \;\;\text{ and }\;\; \theta \co \wt\Delta^o  \times \mathcal{B} \to G$$
that satisfy certain symmetries and a cocycle boundary condition, see Definition~\ref{def_cocycle} for a precise definition. From the definitions, it follows that the set $\mathcal Y(\lambda,d;G)$ is an Abelian group. Furthermore, if $G$ is a Lie group, then so is $\mathcal Y(\lambda,d;G)$. 

\begin{theorem}\label{thm:main-old}\cite[Theorem A]{MMMZ1}
	There is a real analytic diffeomorphism 
	\[\Phi:\mathfrak{R}(\lambda,d)\to{\rm Hit}_d(S)\times\Yc(\lambda,d;\Rb/2\pi\Zb).\]
\end{theorem}

Geometrically, one should think of a $\lambda$--cocyclic pair $(\alpha,\theta)\in \Yc(\lambda,d;\Rb/2\pi\Zb)$ as the ``bending data" with which one can bend a conjugacy class of Hitchin representations along the plaques and leaves of $\lambda$, see \cite[Section 10]{MMMZ1} for more details. Thus, the bijectivity of $\Phi$ in Theorem \ref{thm:main-old} says that every conjugacy class of $d$--pleated surfaces in $\mathfrak{R}(\lambda,d)$ is uniquely realized as a conjugacy class of Hitchin representations in ${\rm Hit}_d(S)$ that is bent along the leaves and plaques of $\lambda$ according to a $\lambda$--cocyclic pair in $\Yc(\lambda,d;\Rb/2\pi\Zb)$.
\medskip

\noindent{\bf The proof of Theorem \ref{ThmA}.}\, Given Theorem \ref{thm:main-old}, in order to prove Theorem~\ref{ThmA}, it suffices to study $\Yc(\lambda,d;\Rb/2\pi\Zb)$. 

\begin{introthm} \label{Thm-general} 
For any Abelian Lie group $G$, there is an isomorphism 
\[
I:\mathcal Y(\lambda,d;G)\to G^{(d^2-1)(2g-2)}\times G_d,
\]
where $G_d=\{g\in G\colon d\cdot g=e\}$.
\end{introthm}
We refer to $G_d$ as the \emph{$d$--torsion} of $G$. For our purposes, we only require Theorem~\ref{Thm-general} for $G=\Rb/2\pi \Zb$ and $G=\Cb/2\pi i\Zb$. However, since its proof is completely algebraic, it goes through for arbitrary $G$. In the case when $G=\Rb$, Theorem \ref{Thm-general} was also proven previously by Bonahon and Dreyer \cite{BoD} (where the $d$--torsion of $\Rb$ is the trivial group).  

We prove Theorem \ref{Thm-general} by explicitly constructing the isomorphism $I$ in two steps. Choose a (trivalent) train track neighborhood $N$ of $\lambda$ and a maximal tree $M\subset N$, see Section \ref{sec:surjectivity} for definitions. For the first step, we use the homological realization of $\mathcal Y(\lambda,d;G)$ to explicitly construct, given our choices, an embedding 
\[\mathcal Y(\lambda,d;G)\to(G^{\mathcal A})^{6g-5}\times(G^\mathcal B)^{12g-12},\] 
and explicitly describe its image by specifying its defining equations, see Theorem~\ref{thm: hom}.  (Note that $6g-5$ is the number of rectangles in $N$ but not in $M$, while $12g-12$ is the number of vertical boundary components of $N$.) 
Then, we specify a procedure to rearrange these defining equations to see that this image is abstractly isomorphic to $G^{(d^2-1)(2g-2)}\times G_d$, see Theorem \ref{thm: topology}. \\

\noindent{\bf The proof of Theorem \ref{ThmB}.}\,
Li \cite{Li} proved that $\mathfrak X(\Gamma,\PGL_d(\Cb))$ has $d$ connected components by defining a surjective \emph{obstruction map}
\[{\rm ob}_d:\mathfrak X(\Gamma,\PGL_d(\Cb))\to\Zb_d\]
whose fibers are the connected components of $\mathfrak X(\Gamma,\PGL_d(\Cb))$. At the same time, we can define the \emph{torsion map}
\[{\rm tor}_d:\mathfrak R(\lambda,d)\to\Zb_d\]
that assigns to every $[\rho]\in\mathfrak R(\lambda,d)$ the $d$--torsion factor of the bending data $(\alpha,\theta)$ of $[\rho]$, i.e. if $\Psi([\rho])=([\rho_0],(\alpha,\theta))$,
\[\pi_{\rm tor}: (\Rb/2\pi\Zb)^{(d^2-1)(2g-2)} \times \Zb_d\to \Zb_d\] 
is the obvious projection, and if $I$ is the isomorphism given in Theorem \ref{Thm-general} (specialized to the case when $G=\Rb/2\pi\Zb$), then ${\rm tor}_d([\rho])=\pi_{\rm tor}\circ I(\alpha,\theta)$. Observe that the fibers of ${\rm tor}_d$ are the connected components of $\mathfrak R(\lambda,d)$, so to prove Theorem \ref{ThmB}, it suffices to prove the following theorem.

\begin{introthm}\label{ThmD}
For any $d$--pleated surface $[\rho]\in\mathfrak R(\lambda,d)\subset\mathfrak X(\Gamma,\PGL_d(\Cb))$, we have 
\[{\rm ob}_d([\rho])={\rm tor}_d([\rho]).\]
\end{introthm}

Since the map $I$ in Theorem \ref{Thm-general} is given explicitly, as a consequence of Theorem~\ref{ThmD}, one obtains a formula for ${\rm ob}_d([\rho])$ in terms of the bending data of $[\rho]$. We only give the formula here in the case when $d$ is odd as it is simpler. See Remark \ref{tor formula} (combined with Theorem \ref{ThmD}) for the general formulae.

\begin{introcor}\label{ThmE}
If $d$ is odd, $[\rho]\in\mathfrak{R}(\lambda,d)$, and $\Phi([\rho])=([\rho_0],(\alpha,\theta))$, then
\[{\rm ob}_d([\rho])=-\sum_{T\in\Delta}\sum_{{\bf j}\in\mathcal B^*}\theta({\bf x}_T,{\bf j}),\]
where $\Delta$ is the set of plaques of $\lambda$, ${\bf x}_T$ is some (any) clockwise labelling of the vertices of some (any) lift of $T$ in $\widetilde S$, and $\Bc^*:=\left\{{\bf j}=(j_1,j_2,j_3)\in\Bc:j_1,j_2,j_3\le \frac{d-1}{2}\right\}$.
\end{introcor}

The proof of Theorem \ref{ThmD} requires several ingredients. The first is the fact proven in \cite{MMMZ1} that the $\lambda$--limit maps of $\lambda$--Borel-Anosov representations admit unique compatible slithering maps. Theorem~\ref{thm: slithering map} gives the precise statement for what this means, but one should think of the slithering map $\Sigma$ compatible with $\xi$ as a canonical way to assign to every pair of leaves $(g_1,g_2)$ of $\widetilde\lambda$ a linear map $\Sigma(g_1,g_2)$ that sends the pair of flags assigned by $\xi$ to the endpoints of $g_2$ to the pair of flags assigned to the endpoints of $g_1$. In the case when $d=2$, this slithering map is induced by the horocyclic foliation, as described by Bonahon \cite[Section 2]{bonahon-toulouse}. 

Another important ingredient is an alternative description of Li's obstruction map that is adapted to a choice of a maximal tree $\mathcal G'$ in a graph $\mathcal G\subset S$ such that $S - \mathcal G$ is homeomorphic to an open ball, see Section \ref{sec:connected}.

Given these ingredients, a detailed outline of the strategy of the proof of Theorem~\ref{ThmD} is given in Section \ref{tor=ob}. Briefly, the boundary of the maximal tree $M$ in $N$, when viewed as a based loop $\mathsf{c}$ in $S$ oriented counterclockwise about $M$, lifts to a loop $\widetilde{\mathsf c}$ in $\widetilde S$, which determines a \emph{cutting sequence}
\[g_0,g_1,\dots,g_\ell=g_0,\]
where each $g_j$ is a leaf of $\widetilde\lambda$. The precise construction of the cutting sequence is given in Section~\ref{sec: cutting}, but one should think of this as a particular sequence of leaves that $\widetilde{\mathsf c}$ intersects. Using this, we define
\[\Sigma_\rho(\widetilde{\mathsf c}):=\Sigma(g_\ell,g_{\ell-1})\dots\Sigma(g_2,g_1)\Sigma(g_1,g_0)\in\SL(d,\Cb),\]
and prove using an inductive argument that $\Sigma_\rho(\widetilde{\csf})=\id$, see Proposition \ref{prop: slithering is trivial}.

At the same time, using the maximal tree $M\subset N$, we may also construct a graph $\mathcal G\subset S$ that has one vertex in every plaque of $\lambda$ and one edge for every rectangle in $N$ but not in $M$, see Section \ref{sec: family}. Then by choosing a maximal tree $\mathcal G'\subset\mathcal G$, and applying the alternative description of the obstruction map, we may construct, for each leaf $g_j$ along the cutting sequence, a certain basis ${\bf v}(j)=(v_1(j),\dots,v_d(j))$ of $\Cb^d$ that has the following properties:
\begin{enumerate}
\item[(I)] $\exp({\rm ob}_d([\rho]))\,{\bf v}(0)={\bf v}(\ell)$.
\item[(II)] For each $j\in\{1,\dots,\ell\}$ and $m\in\{1,\dots,d\}$, $\Sigma(g_j,g_{j-1})$ sends the vector $v_m(j-1)$ of the basis ${\bf v}(j-1)=(v_1(j-1),\dots,v_d(j-1))$ to a multiple $a_m(j)\in\Cb\setminus\{0\}$ of the vector $v_m(j)$ in the basis ${\bf v}(j)=(v_1(j),\dots,v_d(j))$, i.e.
\begin{align*}
	\Sigma(g_j,g_{j-1})\, v_m(j-1)=a_m(j)\,v_m(j).
\end{align*}
\end{enumerate}
We refer to the $a_m(j)$'s as the \emph{slithering coefficients} of ${\bf v}(0),\dots,{\bf v}(\ell)$. From the triviality of $\Sigma_\rho(\widetilde{\csf})$, together with properties (I) and (II) above, one deduces that for all $m\in\{1,\dots,d\}$, we have
\begin{equation*}
{\rm ob}_d([\rho])=-\log \left(\prod_{j=1}^\ell a_m(j)\right).
\end{equation*} 

To finish the proof, we compute
\[\log\left(\prod_{j=1}^\ell a_{\lfloor\frac{d+1}{2}\rfloor}(j)\right)\] 
explicitly in terms of the \emph{shear-bend data of $[\rho]$}, which is a $\lambda$--cocyclic pair in $\mathcal Y(\lambda,d;\Cb/2\pi i\Zb)$ whose imaginary part is the bending data of $[\rho]$, see Proposition~\ref{cor final}. Then, we observe that the expression we obtain is identical to the explicit expression for $-{\rm tor}_d(\rho)$ in terms of the bending data of $[\rho]$ given in Remark~\ref{tor formula}.

\section{Preliminaries of $d$--pleated surfaces}\label{background}

In this section, we recall the definition of $d$--pleated surfaces and the results from \cite{MMMZ1} that will be needed in our exposition. To this purpose, we start by briefly reviewing the collection of projective invariants for triples and quadruples of flags introduced by Fock and Goncharov \cite{fock-goncharov-1} (see Section \ref{sec: flags}). Section \ref{d-pleated} summarizes the background and the necessary terminology on maximal geodesic laminations that will be used throughout the paper. We conclude the section with the definition of $d$--pleated surfaces, their associated shear-bend $\lambda$--cocyclic pairs and the parametrization result established in \cite{MMMZ1} (see Theorem \ref{thm: parameterization}).

Throughout our exposition, $S$ will always denote a closed, connected, oriented surface of genus larger than $1$, endowed with some fixed hyperbolic metric, and $\Gamma$ will denote the deck group of the universal cover $\pi_S\colon \wt S\to S$.

\subsection{Flags and their invariants} \label{sec: flags}

Let $\Fc(\Cb^d)$ denote the space of complete flags in $\Cb^d$. For any $m\geq 2$, an $m$--tuple of flags $F_1,\dots, F_m$ in $\mathcal F(\Cb^d)$ is in {\em general position} if for any non-negative integers $k_1,\dots,k_m$ such that $\sum_{i=1}^mk_i=d$, we have
\[
F_1^{k_1}+\dots+F^{k_m}_m=\Cb^d,
\] 
where for any flag $F\in \Fc(\Cb^d)$ and any integer $k\in\{0,1,\dots, d\}$, $F^k$ denotes the $k$--dimensional subspace determined by $F$.
In particular, a pair of flags in $\Fc(\Cb^d)$ is {\em transverse} if they are in general position. 

We now describe the projective invariants for triples and quadruples of flags in $\Fc(\Cb^d)$ that were introduced by Fock and Goncharov \cite{fock-goncharov-1}. Let $\mathcal{A}$ (respectively, $\mathcal {B}$) be the set of pairs (respectively, triples) of positive integers that sum to $d$. Let also $\mathcal F(\Cb^d)^{[4]}$ denote the set of quadruples $({\bf G},{\bf H})=(G_1, G_2, H_1,H_2)$ of flags in $\Fc(\Cb^d)$ such that $G_1$, $G_2$, and $H_k$ are in general position for both $k=1,2$. Fix a $\Cb$--linear isomorphism $\bigwedge^{d}(\Cb^d) \cong \Cb$. For any flag $F\in\Fc(\Cb^d)$ and any $k\in\{1,\dots,d-1\}$, choose a non-zero element
\[f^{k}\in \bigwedge^kF^{k}.\] 
With this, we define the following collection of projective invariants:
\begin{itemize}
	\item For any triple ${\bf j}=(j_1,j_2,j_3) \in \mathcal{B}$ and any triple ${\bf F}=(F_1,F_2,F_3)$ in general position, the ${\bf j}$--\textit{triple ratio} of ${\bf F}$ is defined by
	$$T^{{\bf j}}({\bf F}):= \frac{f_1^{j_1+1}\wedge f_2^{j_2}\wedge f_3^{j_3-1}}{f_1^{j_1-1}\wedge f_2^{j_2}\wedge f_3^{j_3+1}} \frac{f_1^{j_1}\wedge f_2^{j_2-1}\wedge f_3^{j_3+1}}{f_1^{j_1}\wedge f_2^{j_2+1}\wedge f_3^{j_3-1}}  \frac{f_1^{j_1-1}\wedge f_2^{j_2+1}\wedge f_3^{j_3}}{f_1^{j_1+1}\wedge f_2^{j_2-1}\wedge f_3^{j_3}}.$$
	\item For any pair ${\bf i}=(i_1,i_2) \in \mathcal{A}$ and any $({\bf G},{\bf H})\in\Fc(\Cb^d)^{[4]}$, the ${\bf i}$--\textit{double ratio} of $({\bf G},{\bf H})$ is defined by
	$$D^{\bf i}({\bf G},{\bf H}):= -\frac{g_1^{i_1}\wedge g_2^{i_2-1}\wedge h_1^{1}}{g_1^{i_1}\wedge g_2^{i_2-1}\wedge h_2^{1}} \frac{g_1^{i_1-1}\wedge g_2^{i_2}\wedge h_2^{1}}{g_1^{i_1-1}\wedge g_2^{i_2}\wedge h_1^{1}}.$$
\end{itemize}

Observe that these invariants are well-defined (they do not depend on the choice of $f^{k}\in \bigwedge^kF^{k}$) and take values in $\Cb - \{0\}$. Thus we may define
\begin{align*}
	\tau^{\bf j}({\bf F}):=\log T^{\bf j}({\bf F})\in\Cb/2\pi i\Zb
\end{align*}
and 
\begin{align*}
	\label{eqn: sigmadef}\sigma^{\bf i}({\bf G},{\bf H}):=\log D^{\bf i}({\bf G},{\bf H})\in\Cb/2\pi i\Zb.
\end{align*}

Given two $m$--tuples of flags in general position $(F_1, \dots, F_m)$ and $(G_1, \dots, G_m)$, we say that $(F_1, \dots, F_m)$ and $(G_1, \dots, G_m)$ are projectively equivalent if there exists a transformation $A \in \PGL_d(\Cb)$ such that $A(F_i) = G_i$ for every $i \in \{1, \dots, m\}$. Double and triple ratios provide a complete set of invariants on the space of triples of flags in general position and on the set of quadruples $\mathcal{F}(\Cb^d)^{[4]}$, as described by the following statement:

\begin{proposition}\label{prop: flag invariants}
	The following properties hold:
	\begin{enumerate}
		\item Two triples of flags in general position ${\bf F} = (F_1,F_2,F_3)$ and ${\bf G} = (G_1, G_2, G_3)$ are projectively equivalent if and only if
		\[
		T^{\bf j}({\bf F}) = T^{\bf j}({\bf G}) \qquad \forall {\bf j} \in \Bc .
		\]
		\item Two quadruples $(F_1,F_2,G_1,G_2),  (H_1,H_2,K_1,K_2) \in \mathcal{F}(\Cb^d)^{[4]}$ are projectively equivalent if and only if
		\begin{gather*}
			T^{\bf j}(F_1,F_2,G_1) = T^{\bf j}(H_1,H_2,K_1), \quad T^{\bf j}(F_1,F_2,G_2) = T^{\bf j}(H_1,H_2,K_2) \qquad \forall {\bf j} \in \Bc , \\
			D^{\bf i}(F_1,F_2,G_1,G_2) = D^{\bf i}(H_1,H_2,K_1,K_2) \qquad \forall {\bf i} \in \Ac .
		\end{gather*}
	\end{enumerate}
\end{proposition}

(For a proof of Proposition \ref{prop: flag invariants}, see e.g. \cite[Lemmas 2.2.8, 2.3.7]{tengren_thesis}.)
	
	\subsection{Maximal geodesic laminations} \label{sec: max geod}
	
	A {\em geodesic lamination} $\lambda$ is a closed subset of $S$ which can be decomposed into a disjoint union of simple (complete) geodesics which are called {\em leaves} of the lamination $\lambda$. We say that $\lambda$ is {\em maximal} if it is maximal with respect to inclusion among all geodesic laminations in $S$. A lamination $\lambda$ is maximal exactly when the connected components of $S-\lambda$, called {\em plaques} of $\lambda$, are hyperbolic ideal triangles. We refer to \cite[Chapter~I.4]{MR2235710} (see also \cite[Section~8.5]{thurston-notes}) for a description of the structure of geodesic laminations of finite type hyperbolic surfaces.
	
	Fix, once and for all, a maximal geodesic lamination $\lambda$ in $S$. Denote by $\Lambda$ the set of leaves of $\lambda$, and by $\Delta$ the set of plaques of $\lambda$. Let $\wt \lambda$ be the preimage of $\lambda$ in the universal cover $\wt S\cong \mathbb H^2$, and let $\wt \Lambda$ and $\wt \Delta$ be the set of leaves and the set of plaques of $\wt \lambda$, respectively. We denote by $\partial \wt \lambda$ the subset of the Gromov boundary of $\wt S$ that consists of all endpoints of leaves of the lifted lamination $\wt \lambda$.
	
	The set
	\[
	\wt \Lambda^o:=\{(x,y)\in(\partial\wt \lambda)^2\mid\text{ $x$ and $y$ are the endpoints of a leaf in }\wt \Lambda\}
	\]
	is naturally in bijection with the set of leaves of $\wt\lambda$ equipped with an orientation. Similarly, the set
	\[
	\wt \Delta^o:=\{(x,y,z)\in(\partial\wt \lambda)^3\mid\text{ $x$, $y$, and $z$ are the vertices of a plaque in }\wt \Delta\}
	\]
	is naturally in bijection with the the set of plaques in $\wt \Delta$ equipped with a labelling of their vertices. Recall that $\Gamma$ is the group of deck transformations of the universal cover $\pi_S\colon\wt S\to S$. The quotients $\Lambda^o:=\wt\Lambda^o/\Gamma$ and $\Delta^o:=\wt\Delta^o/\Gamma$ identify with the set of oriented leaves of $\lambda$ and the set of plaques of $\lambda$ endowed with a labeling of their vertices, respectively. 
	
	We denote the natural forgetful projections (that forget either the orientation of the leaves or the labeling of the vertices) as follows:
	\begin{gather*}
		\pi_{\wt\Lambda}\colon\wt\Lambda^o\to\wt\Lambda\qquad \pi_{\Lambda}\colon\Lambda^o\to\Lambda\qquad
		\pi_{\wt\Delta}\colon\wt\Delta^o\to\wt\Delta\qquad \pi_{\Delta}\colon\Delta^o\to\Delta.
	\end{gather*}
	
	Given a plaque $T \in \widetilde{\Delta}$ and a leaf $g \in \widetilde{\Lambda}$, we say that $g$ is an \emph{edge} of $T$ if both endpoints of $g$ are vertices of $T$. Moreover, given $g_1,g_2\in\widetilde{\Lambda}$, we say that a leaf $h\in\widetilde{\Lambda}$ (respectively, a plaque $T\in\widetilde\Delta$) \emph{separates} $g_1$ and $g_2$ if $h$ (respectively, $T$) intersects some (equivalently, any) closed geodesic segment in $\wt S$ with one endpoint in $g_1$ and the other endpoint in $g_2$.
	Notice that given a pair of distinct plaques $T_1,T_2\in\wt\Delta$, there are unique edges $g_1$ of $T_1$ and $g_2$ of $T_2$ such that both $T_1$ and $T_2$ do not separate $g_1$ and $g_2$. We refer to $(g_1,g_2)$ as the \emph{separating pair of edges} for $(T_1,T_2)$. Then, we say that a leaf $h\in\widetilde{\Lambda}$ or a plaque $T\in\widetilde\Delta$ \emph{separates} $T_1$ and $T_2$ if it separates $g_1$ and $g_2$.
	
	For any oriented leaf $\gsf\in\wt\Lambda^o$, we denote by $\gsf^+$ and $\gsf^-$ its forward and backward endpoints respectively. Then we say that $\gsf_1,\gsf_2\in\wt\Lambda^o$ are \emph{coherently oriented} if 
	\[\gsf_2^+ \preceq \gsf_1^+ \prec \gsf_1^-\preceq \gsf_2^- \prec \gsf_2^+\quad\text{ or }\quad \gsf_1^+ \preceq \gsf_2^+ \prec \gsf_2^-\preceq \gsf_1^- \prec \gsf_1^+,\]
	where $\prec$ denotes the counterclockwise order of $\partial\widetilde{S}$ induced by the orientation on $S$. More informally, this is requiring $\gsf_1$ and $\gsf_2$ to be oriented ``in the same direction".
	
	\subsection{$\lambda$--Borel Anosov representations and $d$--pleated surfaces}\label{d-pleated}
	
	Let $\lambda$ be a maximal geodesic lamination and let $\xi$ be a function from the set of endpoints $\partial \widetilde{\lambda}$ to the space $\Fc(\Cb^d)$ of complete flags of $\Cb^d$. Notice that $\xi$ has an associated map $\xi \times \xi : \widetilde{\Lambda}^o\to\Fc(\Cb^d)^2$ on the space of leaves of $\widetilde{\lambda}$ given by $\gsf\mapsto(\xi(\gsf^+),\xi(\gsf^-))$. We say that $\xi$ is:
	\begin{enumerate}
		\item \emph{$\lambda$--continuous} if the associated map $\xi \times \xi : \widetilde{\Lambda}^o\to\Fc(\Cb^d)^2$ is continuous.
		\item \emph{$\lambda$--transverse} if the image of $\xi \times \xi$ lies in the set of transverse pairs of flags.
		\item \emph{$\lambda$--hyperconvex} if for every plaque of $\wt \lambda$ the image of its vertices under $\xi$ is in general position.
	\end{enumerate}
	
	If $\rho:\Gamma\to\PGL_d(\Cb)$ is a representation that admits a $\rho$--equivariant, $\lambda$--continuous, $\lambda$--transverse map $\xi:\partial\widetilde\lambda\to\mathcal F(\Cb^d)$, one can construct for each ${\bf i}\in\Ac$, a certain line bundle $\widehat{H}_{\bf i}$ over the subset $T^1 \lambda$ of the unit tangent bundle $T^1 S$ given by
	\[T^1\lambda:=\{v\in T^1S \mid v\text{ is tangent to a leaf of }\lambda\} . \]
	The geodesic flow of $S$ preserves the subset $T^1 \lambda$ and it lifts naturally to a flow on the total space $\widehat{H}_{\bf i}$, see \cite[Section 3.1]{MMMZ1} for more details. We then say that $\rho$ is \emph{$\lambda$--Borel Anosov} if for all ${\bf i}\in\Ac$, $\widehat{H}_{\bf i}$ is uniformly contracted by the flow. We omit the technical characterization of the uniform contraction of the flow on the line bundles $\wh H_{\bf i} \to T^1 \lambda$ as it will not be needed in the present paper. We refer the interested reader to \cite[Section 3.1]{MMMZ1}. The consequences of this dynamical condition will be recalled whenever needed.

	If $\rho$ is a $\lambda$--Borel Anosov representation, then the map $\xi$ is unique to $\rho$, and so we refer to it as the \emph{$\lambda$--limit map} of $\rho$. A pair $(\rho,\xi)$ is a \textit{$d$--pleated surface with pleating locus $\lambda$} if $\rho:\Gamma\to\PGL_d(\Cb)$ is a $\lambda$--Borel Anosov representation with $\lambda$--limit map $\xi:\partial\wt\lambda\to\Fc(\Cb^d)$, and $\xi$ is \emph{$\lambda$--hyperconvex}. Let $\Rc(\lambda,d)$ denote the set of $d$--pleated surfaces by with pleating locus $\lambda$.
		
	Since the $\lambda$--limit maps of $\lambda$--Borel Anosov representations are unique (see \cite[Theorem 1.1 and Remark~1.3]{wang2021anosov}), the map
	\[\Rc(\lambda,d)\to{\rm Hom}(\Gamma,\PGL_d(\Cb))\]
	given by $(\rho,\xi)\mapsto\rho$ is injective, so we may regard $\Rc(\lambda,d)$ as a subset of $\Hom(\Gamma,\PGL_d(\Cb))$. Also, since $\lambda$--Borel Anosov representations are open in $\Hom(\Gamma,\PGL_d(\Cb))$ (see \cite[Theorem 1.2]{wang2021anosov}), $\mathcal R(\lambda,d)\subset {\rm Hom}(\Gamma,\PGL_d(\Cb))$ is open. Furthermore, as a subset of ${\rm Hom}(\Gamma,\PGL_d(\Cb))$, $\mathcal R(\lambda,d)$ is invariant under conjugation and avoids the singular locus of ${\rm Hom}(\Gamma,\PGL_d(\Cb))$ (see \cite[Proposition 3.9]{MMMZ1}). As such, the complex structure on the set of smooth points on ${\rm Hom}(\Gamma,\PGL_d(\Cb))$ restricts to a complex structure on $\mathcal R(\lambda,d)$, which in turn descends to a complex structure on $\mathfrak R(\lambda,d):=\mathcal R(\lambda,d)/\PGL_d(\Cb)$, and the natural inclusion of the Hitchin component ${\rm Hit}_d(S)\to\mathfrak R(\lambda,d)$ is a real-analytic embedding (see \cite[Section 7.1]{MMMZ1}).
		
	The main result of \cite{MMMZ1} is a parameterization theorem for $\mathfrak R(\lambda,d)$ via $\Cb/2\pi i\Zb$--valued $\lambda$--cocyclic pairs, which we now define. Let $\wt\Delta^{2*}$ denote the set of distinct pairs of plaques of $\wt\lambda$. Recall also that $\widetilde\Delta^o$ is the set of plaques of $\wt \lambda$ equipped with a labelling of their vertices, $\Ac$ is the set of pairs of positive integers that sum up to $d$, and $\Bc$ is the set of triples of positive integers that sum up to $d$. It will be convenient to introduce the following notation. For any triple ${\bf j}=(j_1,j_2,j_3) \in \mathcal{B}$ and any labeled plaque ${\bf x}=(x_1,x_2,x_3) \in \widetilde{\Delta}^o$, denote
		\[\wh{\bf j}:=(j_2,j_1,j_3),\quad{\bf j}_+:=(j_2,j_3,j_1),\quad{\bf j}_-:=(j_3,j_1,j_2),\]
		\[\quad\wh{\bf x}:=(x_2,x_1,x_3),\quad{\bf x}_+:=(x_2,x_3,x_1),\quad{\bf x}_-:=(x_3,x_1,x_2).\] 
		Similarly, for any pair ${\bf i} =(i_1,i_2) \in \Ac$ and any pair of distinct plaques ${\bf T}=(T_1,T_2) \in \widetilde{\Delta}^{2*}$, denote
		\[\wh{\bf i}:=(i_2,i_1)\quad\text{ and }\quad\wh{\bf T}:=(T_2,T_1).\]
		
		\begin{definition}\label{def_cocycle}
			For any Abelian group $G$, a $\lambda$--\textit{cocyclic pair of dimension $d$ with values in $G$} is a pair $(\alpha, \theta)$, where $$\alpha \co \wt\Delta^{2*} \times \mathcal{A} \to G, \;\;\text{ and }\;\; \theta \co \wt\Delta^o  \times \mathcal{B} \to G$$
			satisfy the following properties:
			\begin{enumerate}
				\item(Symmetry of $\alpha$) For all pairs of distinct plaques ${\bf T}\in \wt\Delta^{2*}$ and for all ${\bf i}\in \mathcal{A}$ $$\alpha\left({\bf T}, {\bf i}\right) = \alpha(\widehat{\bf T}, \widehat{\bf i});$$ 
				\item(Symmetry of $\theta$) For every labeled plaque ${\bf x}:=(x_1, x_2, x_3)\in \wt\Delta^o$ and ${\bf j}:=(j_1, j_2, j_3) \in \mathcal{B}$ $$\theta\left({\bf x}, {\bf j}\right) = \theta\left({\bf x}_+, {\bf j}_+\right)= \theta\left({\bf x}_-, {\bf j}_-\right)=-\theta(\widehat{\bf x},\widehat{\bf j});$$
				\item($\Gamma$--invariance of $\alpha$) For all ${\bf T} \in \wt\Delta^{2*}$, ${\bf i}\in \mathcal{A}$, and $\gamma\in\Gamma$, 
				$$\alpha\left(\gamma\cdot{\bf T}, {\bf i}\right) = \alpha({\bf T}, {\bf i});$$
				\item($\Gamma$--invariance of $\theta$) For every pair of distinct plaques ${\bf x}\in \wt\Delta^o$, triple of integers ${\bf j} \in \mathcal{B}$, and $\gamma\in\Gamma$, 
				$$\theta\left(\gamma\cdot {\bf x}, {\bf j}\right) = \theta({\bf x}, {\bf j});$$
				\item(Cocycle boundary condition) Let $T_1$, $T_2$, and $T$ be pairwise distinct plaques of $\wt\lambda$ such that $T$ separates $T_1$ and $T_2$, and let ${\bf x}_T = (x_{T,1}, x_{T,2}, x_{T,3})$ be the labelling of the vertices of $T$ such that the geodesic with endpoints $x_{T,1}$ and $x_{T,2}$ separates $T_1$ and $T$, while the geodesic with endpoints $x_{T,2}$ and $x_{T,3}$ separates $T$ and $T_2$. Then for all ${\bf i}\in \mathcal{A}$, we have that 
				\[
				\alpha((T_1,T_2), {\bf i}) = \alpha((T_1,T), {\bf i}) + \alpha((T,T_2), {\bf i}) + \sum_{{\bf j} \in \Bc : j_2 = i_1} \theta({\bf x}_T,{\bf j})
				\]	
				if $x_{T,3}\prec x_{T,2}\prec x_{T,1}$, and
				\[
				\alpha((T_1,T_2), {\bf i}) = \alpha((T_1,T), {\bf i}) + \alpha((T,T_2), {\bf i}) - \sum_{{\bf j} \in \Bc : j_2 = i_2} \theta({\bf x}_T,{\bf j})
				\]
				if $x_{T,1}\prec x_{T,2}\prec x_{T,3}$. Here, recall that $\prec$ denotes the counterclockwise order of $\partial\widetilde{S}$ induced by the orientation on $S$.
			\end{enumerate}	
		\end{definition} 
		
		Let $\Yc(\lambda,d;G)$ denote the set of $\lambda$--cocyclic pairs of dimension $d$ with values in $G$. We will often write $\alpha^{\bf i}({\bf T}):=\alpha({\bf T},{\bf i})$ and $\theta^{\bf j}({\bf x}):=\theta({\bf x},{\bf j})$ when convenient. We will also denote
		\[\alpha({\bf T}):=(\alpha^{\bf i}({\bf T}))_{{\bf i}\in\Ac}\in G^{\Ac}\quad\text{and}\quad\theta({\bf x}):=(\theta^{\bf j}({\bf x}))_{{\bf j}\in\Bc}\in G^{\Bc}.\]
		Observe that $\Yc(\lambda,d;G)$ is naturally an Abelian group. Furthermore, $\Yc(\lambda,d;\Rb)$ is a real vector space, $\Yc(\lambda,d;\Rb/2\pi\Zb)$ is a compact, Abelian Lie group, and 
		\[\Yc(\lambda,d;\Cb/2\pi i\Zb)=\Yc(\lambda,d;\Rb)+i\Yc(\lambda,d;\Rb/2\pi\Zb) \]
		is a complex, Abelian Lie group.
		
		To associate to every $(\rho,\xi)\in\mathcal R(\lambda,d)$ a point in $\Yc(\lambda,d;\Cb/2\pi i\Zb)$, we will use the fact that the $\lambda$--limit maps of $\lambda$--Borel Anosov representations admit unique slithering maps. More precisely, given a pair of leaves $g_1$ and $g_2$ of $\widetilde\lambda$, let $Q(g_1,g_2)$ denote the set of leaves of $\wt\lambda$ that separate $g_1$ and $g_2$. Following Bonahon and Dreyer \cite{BoD}, we proved the following theorem. 
		
		\begin{theorem}\cite[Theorem 6.2]{MMMZ1}\label{thm: slithering map} 
			For any $\lambda$--Borel Anosov representation $\rho\colon\Gamma\to\sf{PGL}_d(\Cb)$ with $\lambda$--limit map $\xi:\partial\wt\lambda\to\Fc(\Cb^d)$, there exists a unique $\rho$--equivariant map 
			\[
			\Sigma\co\wt\Lambda^2\to\SL_d(\Cb)
			\] 
			satisfying the following properties:
			\begin{enumerate}
				\item $\Sigma(g,g)=\id$ for all $g\in\wt\Lambda$, $\Sigma(g_1,g_2)=\Sigma(g_2,g_1)^{-1}$ for all $g_1,g_2\in\wt\Lambda$, and $\Sigma(g_1,g_2)\Sigma(g_2,g_3)=\Sigma(g_1,g_3)$ for all $g_1,g_2,g_3\in \wt\Lambda$ such that $g_2$ separates $g_1$ and $g_3$.
				\item For all leaves $g_1$ and $g_2$ of $\wt\lambda$, $\Sigma|_{Q(g_1,g_2)}$ is H\"older continuous with respect to the metric on $Q(g_1,g_2)$ induced by a Riemannian metric on $\partial\wt S\cong\mathbb{S}^1$ and the metric on $\SL_d(\Cb)$ induced by the operator norm on $\End(\Cb^d)$.
				\item If $g_1,g_2\in\wt\Lambda$ share an endpoint, then $\Sigma(g_1,g_2)$ is unipotent. 
				\item If $\gsf_1, \gsf_2\in\wt\Lambda^o$ are coherently oriented and $\pi_{\wt\Lambda}(\gsf_j) = g_j$ for $j = 1, 2$, then $\Sigma(g_1,g_2)$ sends $\xi(\gsf_2)$ to $\xi(\gsf_1)$.
			\end{enumerate}
		\end{theorem}
		We refer to the map $\Sigma$ from Theorem \ref{thm: slithering map} as the \emph{slithering map compatible with $\xi$}. We sometimes write $\Sigma_{g_1,g_2}:=\Sigma(g_1,g_2)$ when convenient.
		
		Now, let $(\rho,\xi)$ be a $d$--pleated surface. Using the slithering map $\Sigma$ compatible with $\xi$, we may define a $\lambda$--cocyclic pair $(\alpha_\rho,\theta_\rho)\in\Yc(\lambda,d;\Cb/2\pi i\Zb)$ as follows:
		\begin{itemize}
			\item for all pairs of distinct plaques ${\bf T}=(T_1,T_2)\in\wt\Delta^{2*}$ and all ${\bf i}\in\Ac$, set
			$$\alpha_\rho ({\bf T}, {\bf i})=\alpha_\rho^{\bf i} ({\bf T}) := \sigma^{{\bf i}}\left(\xi(y_2), \xi(y_1), \xi(y_3), \Sigma(g_1,g_2)\cdot\xi(z)\right),$$
			where $(g_1,g_2)$ is the separating pair of edges for ${\bf T}=(T_1,T_2)$, $z$ is the vertex of $T_2$ that is not an endpoint of $g_2$, and $(y_1,y_2,y_3)$ are the vertices $T_1$, enumerated so that $y_1$ and $y_2$ are the endpoints of $g_1$ and $y_1\prec y_2\prec y_3$. 
			\item for all labeled plaque ${\bf x}\in\wt\Delta^o$ and all triple of indices ${\bf j}\in\Bc$, set
			$$\theta_\rho ({\bf x}, {\bf j})=\theta_\rho^{\bf j}({\bf x}) := \tau^{\bf j}\left(\xi({\bf x})\right).$$
		\end{itemize}
		We refer to $(\alpha_\rho,\theta_\rho)$ as the \emph{shear-bend $\lambda$--cocyclic pair of $\rho$}. The fact that $(\alpha_\rho,\theta_\rho)\in \Yc(\lambda,d;\Cb/2\pi i\Zb)$ follows from the symmetries of the triple and double ratios, see for example, in \cite[Proposition~2.4]{MMMZ1}. 
		
		Since the triple ratios and double ratios are projective invariants, we may define the shear-bend parameterization map
		\[\mathfrak{sb}_d:\mathfrak{R}(\lambda,d)\to\Yc(\lambda,d;\Cb/2\pi i\Zb)\]
		by setting $\mathfrak{sb}_d([\rho])=(\alpha_\rho,\theta_\rho)$. We then show the following.
		\begin{theorem}\cite[Theorem A]{MMMZ1}\label{thm: parameterization}
			The map $\mathfrak{sb}_d$
			is a biholomorphism onto $\mathcal{C}(\lambda, d) + i\, \mathcal{Y}(\lambda,d;\Rb/2\pi \Zb)$, where $\mathcal{C}(\lambda, d)\subset\Yc(\lambda,d;\Rb)$ is a convex, open polyhedral cone. Furthermore, $\mathfrak{sb}_d|_{{\rm Hit}_d(S)}:{\rm Hit}_d(S)\to\mathcal{C}(\lambda, d)$ is the Bonahon-Dreyer parameterization, which is a real analytic diffeomorphism.
		\end{theorem}

\section{Train track neighborhoods and their maximal trees}\label{sec:surjectivity}

The proofs of both Theorem \ref{Thm-general} and Theorem \ref{ThmD} rely heavily on the choice of a train track neighborhood for the maximal lamination $\lambda$. In this section, we recall the definition of train track neighborhood $N$ and of its orientation cover $N^o \to N$ (see Section \ref{ssub:tt_neighborhood}), we introduce the notion of a tree $L$ inside $N$ (see Section \ref{ssub:max tree}) and we describe the structure of the orientation cover $N^o$ of $N$ via the choice of a maximal tree (see Section \ref{sub:graph_mathcal_g_on_s}).

\subsection{Train track neighborhoods} \label{ssub:tt_neighborhood}

We now recall the notion of a train track neighborhood for the fixed maximal geodesic lamination $\lambda$. We work with the definition of train track neighborhood and the related terminology from the work of Bonahon and Dreyer \cite[Section~4.2]{BoD}. Classical references on the subject (which use slightly different definitions, better suited to the study of hyperbolic surfaces) are the work of Thurston \cite[Section~8.9]{thurston-notes} and of Penner and Harer \cite{penner-harer}.

First, let $r\co[0,1]\times[0,1]\to S$ (respectively, $r\co[0,1]\times[0,1]\to \wt S$) be the restriction of a smooth embedding from a neighborhood of $[0,1]\times [0,1]$ in $\Rb^2$ to $S$. We refer to the image $R$ of $r$ as a \emph{rectangle} in $S$ (respectively, $\wt S$). The boundary $\partial R$ of $R$ can be divided into the \emph{horizontal boundary} $\partial_hR:=r([0,1]\times\{0,1\})$ and the {\em vertical boundary} $\partial_vR:=r(\{0,1\}\times[0,1])$. A \emph{tie} of the rectangle $R$ is a subset of the form $r(\{x\}\times[0,1])$ for some $x\in [0,1]$. The points
\[r(0,0),\, r(0,1),\, r(1,0),\text{ and }r(1,1)\]
are called the \emph{vertices} of $R$.

A {\em(trivalent) train track neighborhood} for $\lambda$ is a closed neighborhood $N\subset S$ of $\lambda$ which can be written as a union of finitely many rectangles $\{R_1,\dots,R_n\}$, such that the following conditions are satisfied:
\begin{enumerate}
	\item If two rectangles $R_i$ and $R_j$ intersect, then every component of $R_i\cap R_j$ is, up to switching the roles of $i$ and $j$, a vertical boundary component of $R_i$, lies in a vertical boundary component of $R_j$, and contains exactly one vertex of $R_j$. Every vertical component of a rectangle $R_j$ that satisfies the properties above (with respect to some rectangle $R_i$) is called a \emph{switch} of $N$. 
	\item For each rectangle $R_i$, every vertex of $R_i$ is contained in some $R_j$ different from $R_i$.
	\item For each rectangle $R_i$, every tie of $R_i$ intersects some leaf of $\lambda$, and any such intersection is transverse.
	\item Each component of $S-N$ is a topological cell, and its boundary is the union of six smooth curves, three of which each lie in some switch, and the other three are each a union of horizontal boundary components of rectangles.
\end{enumerate}

Let $N$ be a train track neighborhood of $\lambda$, and let $g$ be the genus of $S$. Observe that there is a natural bijection between the set of switches of $N$ and the set of vertices of the plaques of $\lambda$. Since $S-\lambda$ is the union of $4g-4$ plaques, it follows that  $N$ has $12g-12$ switches. Furthermore, every switch of $N$ contains exactly three vertical boundary components of rectangles in $N$. Since every rectangle of $N$ has two vertical boundary components, it follows that $N$ has $18g-18$ rectangles.

The boundary of $N$ can be naturally decomposed into a union of connected segments, each of which either lies in some switch, or is the union of horizontal boundary components of some rectangles. The union of all (closed) segments in $\partial N$ that satisfy the former (respectively, latter) is the \emph{vertical boundary} (respectively, \emph{horizontal boundary}) of $N$, which we denote by $\partial_h N$ (respectively, $\partial_v N$). 

Notice that the train track neighborhood $N$ lifts to a closed neighborhood $\wt N\subset \wt S$ of $\wt\lambda$, which we refer to as a $\Gamma$--invariant \emph{train track neighborhood} of $\wt\lambda$. Clearly, $\wt N$ can be written as countable union of rectangles, each of which is a lift of one of the rectangles of $\wt N$. Also, the boundary of $\wt N$ can be written as the union of the \emph{vertical boundary} $\partial_v\wt N$ and \emph{horizontal boundary} $\partial_h\wt N$, which descend via the universal cover $\pi_S:\wt S\to S$ to $\partial_v N$ and $\partial_h N$ respectively. 

\subsubsection{Orientation covers}\label{sec: orientation covers}

If $N$ is a train track neighborhood of the geodesic lamination $\lambda$ in $S$, the \emph{orientation cover} $N^o$ of $N$ is the set of pairs $(x,u)$, where $x\in N$ and $u$ is an orientation of the tie of $N$ that contains $x$. We may endow $N^o$ with a topology by embedding $N^o$ into $T^1 S$, in which case the natural covering map $\pi_N:N^o\to N$ is a double cover.  We denote $\partial_h N^o:=\pi_N^{-1}(\partial_h N)$ and $\partial_v N^o:=\pi_N^{-1}(\partial_v N)$, and refer to these as the \emph{horizontal boundary} and \emph{vertical boundary} of $N^o$ respectively.

The preimage $\lambda^o := \pi_N^{-1}(\lambda)$ inside $N^o$, together with the projection map $\pi_N|_{\lambda^o} : \lambda^o \to \lambda$, can be thought as the $2$--fold orientation cover of the maximal lamination $\lambda$. Indeed, any point in $\lambda^o$ corresponds to the data of a point $x \in \lambda$ and of an orientation of the tie of $N$ that passes through $x$. Since the ties of $N$ are everywhere transverse to the leaves of $\lambda$ and since the surface $S$ is oriented, the orientation of the tie containing $x$ corresponds to the choice of an orientation of the leaf of $\lambda$ $g\in\Lambda$ that contains $x$ equipped with the orientation $o$, following the convention that the leaf $g$ is oriented so that the tie through $x$ crosses the leaf from left to right. Via this embedding, $\lambda^o$ is a geodesic lamination in $N^o$, and $N^o$ is a train track neighborhood of $\lambda^o$, and $\pi_N$ restricts to the $2$--fold cover $\lambda^o\to\lambda$. 

Hence, the leaves of $\lambda^o$ are naturally oriented, and can be identified with the set $\Lambda^o$ of oriented leaves of $\lambda$. Also, each vertical boundary component, tie, and switch of $N^o$ has a natural orientation, which can be viewed respectively as a vertical boundary component, tie, and switch of $N$ equipped with the induced orientation. Similarly, each rectangle of $N^o$ can be viewed as a rectangle of $N$ equipped with a continuous orientation on its ties.

\subsection{Trees in a train track neighborhood} \label{ssub:max tree}

Let $N$ be a train track neighborhood of the maximal geodesic lamination $\lambda$. For every rectangle $R$ of $N$, we fix once and for all a decomposition 
\[R=K_1\cup\check R\cup K_2\]
such that $K_1,K_2,\check R$ are closed, connected subsets of $R$ that are unions of ties, and satisfy the following properties:
\begin{enumerate}
	\item $K_1$ and $K_2$ each contain a vertical boundary component of $R$,
	\item $K_1\cap K_2$ is empty,
	\item $K_1\cap \check R$ and $K_2\cap\check R$ are both ties of $R$.
\end{enumerate}
We refer to $K_1$ and $K_2$ as the \emph{stumps} in $R$, and $\check R$ as a \emph{truncated rectangle} in $R$. Notice also that every switch $s$ of $N$ intersects exactly three stumps, exactly one of which contains $s$. We refer to the union of these three stumps as the \emph{stumpy switch} containing $s$.

We say that $L\subset N$ is a \emph{tree} if it is the union of stumpy switches and truncated rectangles, such that
\begin{enumerate}
	\item if a truncated rectangle $\check R$ lies in $L$, then the two stumpy switches that intersect $\check R$ also lie in $L$, and
	\item $L$ is connected and simply connected.
\end{enumerate}
One should think of the stumpy switches as vertices of $L$ and the truncated rectangles as edges of $L$. We also say that a tree is \emph{maximal} if it is maximal with respect to inclusion, see Figure \ref{fig:orientable.rectangles}. 

\begin{figure}[h!]
	\centering
	\includegraphics[width=\textwidth]{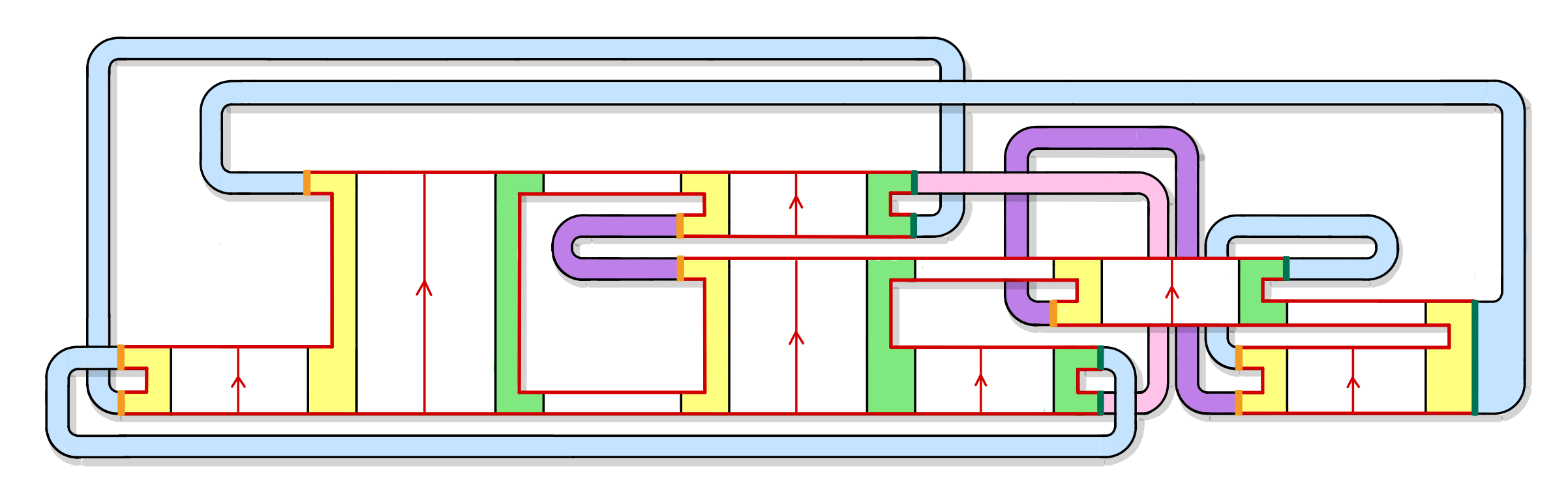}
	\caption{\label{fig:orientable.rectangles}\small A train track neighborhood $N$ of some lamination in a genus $2$ surface. Rectangles of a fixed maximal tree $M\subset N$ are in red, with arrows indicating the chosen orientation on its ties. Stumpy switches containing a left (resp. right) vertical boundary component of $N$ are in yellow (resp. green). Left (resp. right) exits of $M$ are in orange (resp. dark green). Orientable rectangles are in blue, left (resp. right) unorientable rectangles are in purple (resp. pink).}
\end{figure}

Suppose that $L\subset N$ is a tree endowed with a continuous orientation $o$ on its ties. Being the surface $S$ orientable and $L$ contractible, there exist exactly two choices of orientations. An \emph{exit} of $L$ is a tie of $N$ that lies in the boundary $\partial L$ of $L$. Moreover, we say that an exit of $L$ is \emph{left} (respectively, \emph{right}) if $L$ lies to the right (respectively, left) of the exit with respect to the chosen orientation. Similarly, we say that a vertical boundary component of $N$ that lies in $L$ is \emph{left} (respectively, \emph{right}) if $L$ lies to the right (respectively, left) of the vertical boundary component, see Figure \ref{fig:orientable.rectangles}. Let $\mathcal E(L)$, $\mathcal E^\ell(L)$, and $\mathcal E^r(L)$ denote the set of exits of $L$, left exists of $L$ and right exits of $L$. Similarly, we denote by $\mathcal S(L)$, $\mathcal S^\ell(L)$, and $\mathcal S^r(L)$ the set of vertical boundary components of $N$ that lie in $L$, the set of left vertical boundary components in $L$ and the set of right vertical boundary components of $N$. Clearly
\[
\mathcal E(L)=\mathcal E^\ell(L) \sqcup \mathcal E^r(L)\quad\text{and}\quad\mathcal S(L)=\mathcal S^\ell(L) \sqcup \mathcal S^r(L).
\] 

Next, suppose that $M\subset N$ is a maximal tree, together with the choice of a continuous orientation of its ties. If $R$ is a rectangle of $N$ that does not lie in $M$, we say that $R$ is \emph{orientable} (respectively, \emph{unorientable}) for $M$ if the orientation of the ties of $M$ extends (respectively, does not extend) to a continuous orientation of the ties of $M\cup R$, see Figure \ref{fig:orientable.rectangles}.
By definition, every rectangle of $N$ is either a rectangle of $M$, an orientable rectangle for $M$, or an unorientable rectangle for $M$. We denote the sets of orientable rectangles and the set of unorientable rectangles for $M$ by $\mathcal{O}(M)$ and $\mathcal{U}(M)$, respectively. Observe that $\mathcal U(M)$ is non-empty; indeed, for any plaque $T$ of $\lambda$, there is an unorientable rectangle that intersects an edge of $T$.

Every orientable rectangle $R\in\mathcal O(M)$ lies to the left of one of its vertical boundary components, and to the right of the other. On the other hand, an unorientable rectangle $R\in\mathcal U(M)$ has to lie either to the left of both of its vertical boundary components or to the right of both of its vertical boundary components. We call the former a \emph{left unorientable rectangle}, and the latter a \emph{right unorientable rectangle}. Denote by $\mathcal U^\ell(M)$ (respectively, $\mathcal U^r(M)$) the set of left (respectively, right) unorientable rectangles. 

Henceforth, we will fix once and for all a maximal tree $M\subset N$, together with a continuous orientation $o$ on its ties. Notice that, by maximality, $M$ must contain all vertical boundary components of $N$. In the remainder of the exposition, we will simplify notation and drop the dependence of the maximal tree $M$. In other words, we will write
\[\mathcal E:=\mathcal E(M),\quad\mathcal O:=\mathcal O(M),\quad\mathcal U:=\mathcal U(M), \]
\[\mathcal E^\ell:=\mathcal E^\ell(M),\quad \mathcal E^r:=\mathcal E^r(M),\quad\mathcal S^\ell:=\mathcal S^\ell(M), \quad\mathcal S^r:=\mathcal S^r(M),\]
\[\quad\mathcal U^\ell:=\mathcal U^\ell(M), \quad\text{and}\quad\mathcal U^r:=\mathcal U^r(M).
\]

\subsection{Lifts of $M$ to $N^o$} \label{sub:graph_mathcal_g_on_s}

Recall that $\pi_N : N^o \to N$ is the orientation double cover and $\iota : N^o \to N^o$ is the covering involution of $\pi_N$. Since the maximal tree $M$ is contractible, the preimage $\pi_N^{-1}(M)$ is the disjoint union of two path components, and the restriction of $\pi_N$ to each of them is a homeomorphism onto $M$. For exactly one of these path components, denoted $M^o$, the orientation on its ties inherited from the orientation on the ties of $N^o$ is mapped via $\pi_N$ to the chosen continuous orientation on the ties of $M$. Then note that the other connected component of $\pi_N^{-1}(M)$ is $\iota(M^o)$.

Observe that if $R$ is an orientable rectangle, then the two vertical boundary components of each connected component of $\pi_N^{-1}(R)$ either both lie in $M^o$ or both lie in $\iota(M^o)$. On the other hand, if $R$ is unorientable, then each connected component of $\pi_N^{-1}(R)$ has one vertical boundary component in $M^o$ and one in $\iota(M^o)$. Thus, for every rectangle $R$ of $N$, we may choose once and for all a connected component $R^o$ of $\pi_N^{-1}(R)$ as follows:
\begin{itemize}
	\item If $R$ is a rectangle of $M$, set $R^o$ to be the connected component of $\pi_N^{-1}(R)$ that lies in $M^o$.
	\item If $R$ is an orientable rectangle for $M$, set $R^o$ to be the connected component of $\pi_N^{-1}(R)$ whose boundary ties both lie in $M^o$.
	\item If $R$ is an unorientable rectangle for $M$, choose once and for all a connected component of $\pi_N^{-1}(R)$, and set it to be $R^o$.
\end{itemize}
We refer to $R^o$ as the \emph{orientation on $R$ chosen by $M^o$}. Notice that the set of rectangles of $N^o$ is equal to
\[
\{R^o: \text{$R$ rectangle of $N$} \}\cup\{\iota(R^o):\text{$R$ rectangle of $N$}\}.
\]

On the other hand, for every vertical boundary component $t\in\mathcal S$, let $P_t$ be the connected component of $S - N$ whose boundary $\partial P_t$ contains $t$. Then let 
\begin{itemize}
	\item $t^o$ be the connected component of $\pi_N^{-1}(t)$ that lies in $M^o$, and 
	\item $t^{cw}$ be the connected component of $\pi_N^{-1}(t)$ that, when viewed as $t$ equipped with an orientation, is an oriented subsegment of $\partial P_t$ equipped with the clockwise orientation about $P_t$.
\end{itemize}
We refer to $t^o$ as the \emph{orientation on $t$ chosen by $M^o$} and $t^{cw}$ as the \emph{orientation on $t$ induced by the clockwise orientation on $P_t$}.
Note that for all $t\in\mathcal S$, we have
\begin{align}\label{eqn: hs}
	t^o=\left\{ \begin{array}{ll}
		t^{cw}&\text{if }t\in\mathcal S^r,\\
		\iota(t^{cw})&\text{if }t\in\mathcal S^\ell .
	\end{array} \right.
\end{align}
Moreover, the set of vertical boundary components of $N^o$ can be expressed as
\[
\{t^o:t\in\mathcal S\}\cup\{\iota(t^o):t\in\mathcal S\}=\{t^{cw}:t\in\mathcal S\}\cup\{\iota(t^{cw}):t\in\mathcal S\}.
\]

\section{Describing $\Yc(\lambda,d;G)$ as a subgroup of $(G^{\Ac})^{\mathcal O  \sqcup \mathcal U } \times (G^{\Bc})^{\mathcal S }$}\label{sec:isomorphism}

Let $G$ be any Abelian Lie group. The main theorem of this section realizes the Lie group $\mathcal Y(\lambda,d;G)$ as a subgroup of a finite product of copies of $G$, cut out by an explicit set of equations. This is an intermediate step towards understanding the global topology of $\Yc(\lambda,d;G)$. 

\subsubsection{Notation and main statement}\label{subsubsec:notation for hom}

To state the aforementioned result formally, recall that we fixed once and for all:
\begin{itemize}
	\item a maximal geodesic lamination $\lambda$ of $S$, 
	\item a train track neighborhood $N$ of $\lambda$, 
	\item a maximal tree $M\subset N$, and 
	\item a connected component $M^o$ of $\pi_N^{-1}(M)$, where $\pi_N:N^o\to N$ is the orientation double cover. 
\end{itemize} 
We also introduced the following notations:
\begin{itemize}
	\item $\mathcal O$, $\mathcal U$, $\mathcal U^\ell$, and $\mathcal U^r$ denote the sets of orientable rectangles, unorientable rectangles, left unorientable rectangles, and right unorientable rectangles for $M$, respectively (see Section \ref{ssub:max tree}).
	\item $\mathcal S$, $\mathcal S^\ell$, $\mathcal S^r$ respectively denote the sets of vertical boundary components, left vertical boundary components, and right vertical boundary components of $N$ (see Sections \ref{ssub:tt_neighborhood} and \ref{ssub:max tree}).
	\item $\mathcal A$ denotes the set of pairs of positive integers that sum to $d$, and $\mathcal B$ denotes the set of triples of positive integers that sum to $d$, (see Section \ref{sec: flags}).
	\item For every triple of indices ${\bf j}=(j_1,j_2,j_3)\in\mathcal B$, we denote by ${\bf j}_+:=(j_2,j_3,j_1)$, ${\bf j}_-:=(j_3,j_1,j_2)$ its cyclic permutations, and we set $\widehat{\bf j}:=(j_3,j_2,j_1)$. Furthermore, for all ${\bf i}=(i_1,i_2)\in\mathcal A$, we set $\widehat{\bf i}:=(i_2,i_1)$, (see Section \ref{d-pleated}).
	\item For every rectangle $R$ of $N$, $R^o$ denotes the orientation on $R$ chosen by $M^o$, and for every vertical boundary component $t\in\mathcal S$, $t^{cw}$ denotes the orientation on $t$ induced by the clockwise orientation on $\partial P_t$ about $P_t$, which is the connected component of $S\setminus N$ that contains $t$ in its boundary (see Section \ref{sub:graph_mathcal_g_on_s}).
	\item $\pi_S:\widetilde S\to S$ denotes the universal cover of $S$.
\end{itemize}

We will also need to make the following technical choices:

\begin{itemize}
	\item For each vertical boundary component $\tsf$ of $N^o$, we select once and for all a lift $T \in \widetilde{\Delta}$ of the plaque that contains $\pi_N(\tsf)$ and a labeling of its vertices ${\bf x}_{\tsf}=(x_{\tsf,1},x_{\tsf,2},x_{\tsf,3})\in\widetilde\Delta^o$ so that the following property holds. If $\mathsf s$ denotes the tie of $N^o$ that contains $\tsf$ and $g_-$, $g_+$ are the edges of $T$ such that $x_{\tsf,1}$ is an endpoint of $g_-$, $x_{\tsf,3}$ is an endpoint of $g_+$, and $x_{\tsf,2}$ is the common endpoint of $g_-$ and $g_+$, then the backward and forward endpoints of $\mathsf s\cap \pi_S(T)$ (here we view $\mathsf s$ as a tie of $N$ equipped with an orientation) lie in $\pi_S(g_-)$ and $\pi_S(g_+)$ respectively. We refer to ${\bf x}_{\tsf}\in\widetilde\Delta^o$ as the \emph{triple of vertices chosen for $\tsf$}. Notice that any two such choices differ by an element of the deck group. 
	\item For each rectangle $\mathsf{R}$ of $N^o$, choose once and for all a distinct pair of plaques ${\bf T}_{\mathsf{R}}:=(T_{\mathsf{R},1},T_{\mathsf{R},2})\in\widetilde\Delta^2$ such that $\pi_S(T_{\mathsf{R},1})$ and $\pi_S(T_{\mathsf{R},2})$ respectively contain the backward and forward endpoints of some (every) tie of $\mathsf{R}$ (here we view ${\mathsf{R}}$ as a rectangle of $N$ equipped with a continuous orientation on its ties). We refer to ${\bf T}_{\mathsf{R}}$ as the \emph{pair of plaques chosen for $\mathsf{R}$}.
\end{itemize}

For any vertical boundary component $t$ of $N$, we will also denote by $t_+, t_-\in\mathcal S$ the two other vertical boundary components of $N$ that are contained in $\partial P_t$, so that $t<t_+<t_-<t$ in the clockwise orientation of $\partial P_t$ about $P_t$.

We can now state the main result of the current section:

\begin{theorem}\label{thm: hom}
	The map $I_1:\mathcal Y(\lambda,d;G)\to (G^{\Ac})^{\mathcal O  \sqcup \mathcal U } \times (G^{\Bc})^{\mathcal S }$ given by
	\begin{equation}\label{eq: I_1}
		I_1:\left(\alpha,\theta\right)\mapsto\Big((\alpha({\bf T}_{R^o}))_{R\in\mathcal O \sqcup\mathcal U},(\theta({\bf x}_{t^{cw}}))_{t\in\mathcal S}\Big)
	\end{equation}
	is an embedding of Lie groups, and its image is \[Y=Y(G):= \left\{({\bf v}, {\bf z}) \in (G^{\Ac})^{\mathcal O  \sqcup \mathcal U } \times (G^{\Bc})^{\mathcal S } \left| \,\,\, \parbox{5cm}{$\blacklozenge(t, {\bf j})$ for all ${\bf j} \in \Bc$ and $t \in \mathcal S $, and $\clubsuit({\bf i})$ for all ${\bf i} \in \Ac$.} \right. \right\} ,\]
	where
	\begin{gather*}
		\tag*{$\blacklozenge(t, {\bf j})$} {z}_{t}^{\bf j} = {z}_{t_+}^{{\bf j}_+} = {z}_{t_-}^{{\bf j}_-}, \\
		\tag*{$\clubsuit({\bf i})$} \sum_{R \in \mathcal U ^r} ( v_R^{\mathbf{i}} + v_R^{\widehat{\mathbf{i}}}) - \sum_{R \in \mathcal U ^\ell} (v_R^{\mathbf{i}} + v_R^{\widehat{\mathbf{i}}} ) = \sum_{t\in\mathcal S^\ell} \sum_{{\bf j} \in \Bc : j_2 = i_2} z^{\bf j}_{t} - \sum_{t\in\mathcal S^r} \sum_{{\bf j} \in \Bc : j_2 = i_1} z^{\bf j}_{t}.
	\end{gather*}
	Here, ${\bf v}=(v_R)_{R\in\mathcal O \sqcup\mathcal U}=((v_R^{\bf i})_{{\bf i}\in\Ac})_{R\in\mathcal O \sqcup\mathcal U}$ and ${\bf z}=(z_t)_{t\in\mathcal S}=((z_t^{\bf j})_{{\bf j}\in\Bc})_{t\in\mathcal S}$.
\end{theorem}

The proof of this statement is outlined in Section \ref{proof of subgroup theorem}. Our argument heavily relies on the homological interpretation of the space of $\lambda$--cocyclic pairs of dimension $d$ with values in $G$ associated with $N$ (see Proposition \ref{prop: Homological}). For this reason, we start by recalling the necessary homological framework.

\subsection{Homology}

Let $N$ be a train track neighborhood of $\lambda$ and let $\pi_N : N^o \to N$ denote its orientation cover. We will discuss the homology groups of $N^o$, and their relationship with $\lambda$--cocyclic pairs. We refer to \cite[Section 4]{MMMZ1} and references within for more details.

\subsubsection{Relative homology of $N^o$}\label{subsubsec:relative homology of No}

Let $\mathcal S$ and $\mathcal S^o$ denote the set of vertical boundary components of the train track neighborhoods $N$ and $N^o$, respectively. Since the vertical boundary components are contractible, the preimage of any $t \in \mathcal{S}$ has two connected components that are vertical boundary components of $N^o$. In other words, the orientation cover projection determines a two-to-one, surjective map $\mathcal{S}^o \to \mathcal{S}$ that sends each $\mathsf{t} \in \mathcal{S}^o$ into $\pi_N(\mathsf{t}) \in \mathcal{S}$. Choose a point in each vertical boundary component of $N$ and let $\{q_{\mathsf{t}} \mid \mathsf{t} \in \mathcal{S}^o \}$ be the preimage inside $\partial_v N^o$ of such finite collection of points. By construction, the involution $\iota : N^o \to N^o$ of the orientation cover $\pi_N : N^o \to N$ satisfies $\iota(q_{\mathsf{t}}) = q_{\iota(\mathsf{t})}$ for every vertical boundary component $\mathsf{t}$ of $N^o$. We may view each $q_{\tsf}$ as a $0$--cycle of $\partial_v N^o$, and so it represents a homology class $[q_{\tsf}]$ in $H_0(\partial_vN^o;\Zb)$.

Similarly, let $\mathcal N$ and $\mathcal N^o$ be the sets of rectangles of $N$ and of $N^o$, respectively. The orientation cover map determines a two-to-one surjective map $\mathcal{N}^o \to \mathcal{N}$, sending $\mathsf{R} \in \mathcal{N}^o$ to $\pi_N(\mathsf{R})$. For each rectangle $R$ of $N$, choose a curve $r_R$ in $R$ that is transverse to every tie of $R$, and has endpoints in each of the two vertical boundary components of $R$. For each endpoint $p$ of $r_R$, let $t_p$ be the vertical boundary component of $N$ that lies in the switch containing $p$, and let $s_p$ be the subsegment of the switch containing $p$ whose endpoints are $p$ and $q_{t_p}$. Then let $k_R$ be the curve $r_R\cup s_{p_R}\cup s_{q_R}$, where $p_R$ and $q_R$ are the endpoints of $r_R$. Consider now the set of all the lifts of the curves $k_R$ inside $N^o$, as $R$ varies among the rectangles of $N$. We can index this set as $\{ k_{\mathsf{R}} \mid \mathsf{R} \in \mathcal{N}^o \}$, where $\mathcal{N}^o$ denotes the set of rectangles of $N^o$. Since the ties of $N^o$ are naturally oriented, $\ksf_{\mathsf{R}}$ can be oriented so that it passes from the right to the left of every tie of $\mathsf{R}$. Thus, we may also view $\ksf_{\mathsf{R}}$ as a $1$--cycle of $N^o$ relative to $\partial_v N^o$, and so it represents a relative homology class $[\ksf_{\mathsf{R}}]$ in $H_1(N^o, \partial_v N^o; \Zb)$. One can check that the cover involution $\iota$ exchanges the unoriented paths $k_{\mathsf{R}}$ and $k_{\iota(\mathsf{R})}$, and the orientation of $\iota(\ksf_{\mathsf R})$ is opposite to the orientation of $\ksf_{\iota(\mathsf R)}$.

Let $G$ be an Abelian group. Then the homology groups $H_1(N^o, \partial_v N^o; G)$ and $H_0(\partial_v N^o; G)$ decompose as
\begin{equation*}\label{eqn: H1}
	H_1(N^o, \partial_v N^o; G) =\bigoplus_{\mathsf{R}\in\mathcal N^o}G\cdot[\ksf_{\mathsf{R}}]
\end{equation*}
and
\begin{equation*}\label{eqn: H0}
	H_0(\partial_v N^o;G)=\bigoplus_{\tsf\in\mathcal S^o}G\cdot[q_{\tsf}].
\end{equation*}
Henceforth, we will refer to $\{[\ksf_{\mathsf{R}}]:\mathsf{R}\in\mathcal N^o\}$ and $\{[q_{\tsf}]:\tsf\in\mathcal S^o\}$ as the \emph{standard generating sets} of $H_1(N^o,\partial_v N^o; G)$ and $H_0(\partial_v N^o,G)$, respectively. 

Since the cover involution $\iota$ leaves $\partial_v N^o$ invariant, it induces involutions 
\begin{align*}
	\iota_*:H_1(N^o, \partial_v N^o; G^\mathcal{A})\to H_1(N^o, \partial_v N^o; G^\mathcal{A})
\end{align*}
and
\begin{align*}
	\iota_*:H_0(\partial_v N^o; G^\mathcal{A})\to H_0(\partial_v N^o; G^\mathcal{A}).
\end{align*}
More explicitly, if $\{[\ksf_{\mathsf{R}}]:\mathsf{R}\in\mathcal N^o\}$ and $\{[q_{\tsf}]:\tsf\in\mathcal S^o\}$ are the standard generating sets of $H_1(N^o,\partial_v N^o; G)$ and $H_0(\partial_v N^o,G)$ respectively, then 
\[\iota_*(g\cdot[\ksf_{\mathsf{R}}])=-g\cdot[\ksf_{\iota(\mathsf{R})}]\quad\text{and}\quad\iota_*(g\cdot[q_{\tsf}])=g\cdot[q_{\iota(\tsf)}]\]
for every rectangle $\mathsf{R}$ and vertical boundary compoennt $\tsf$ of $N^o$, and all $g\in G^{\mathcal A}$.

Given any $g = (g_{\bf i})_{{\bf i} \in \Ac}\in G^{\Ac}$, let $\widehat{g}\in G^{\Ac}$ be the element such that $\widehat{g}_{\bf i}=g_{\widehat{\bf i}}$ for all ${\bf i}\in\Ac$. Then let 
\begin{align*}
	\widehat{\cdot}:H_1(N^o, \partial_v N^o; G^\mathcal{A})\to H_1(N^o, \partial_v N^o; G^\mathcal{A})
\end{align*}
be the involution that sends $g\cdot [\ksf_{\mathsf{R}}]$ to $\widehat{g}\cdot [\ksf_{\mathsf{R}}]$ for all $\mathsf{R}\in\mathcal N^o$ and all $g\in G^\mathcal A$, and similarly let
\[\widehat{\cdot}:H_0(\partial_v N^o; G^\mathcal{A})\to H_0(\partial_v N^o; G^\mathcal{A})\]
be the involution that sends $g\cdot [q_{\tsf}]$ to $\widehat{g}\cdot [q_{\tsf}]$ for all $\tsf\in\mathcal S^o$ and all $g\in G^\mathcal A$.

\subsubsection{Homological interpretation of $\lambda$--cocyclic pairs}\label{subsubsec:homological interpretation of cocyclic pairs} 

We will now describe a homological interpretation of $\lambda$--cocyclic pairs. To do so, we need the notion of a triangle data function.

\begin{definition}\label{def: triangle data function}
	A $\lambda$--\textit{triangle data function of dimension $d$ with values in $G$} is a map
	$$\theta \co \wt\Delta^o  \times \mathcal{B} \to G$$
	satisfying the following conditions:
	\begin{enumerate}
		\item(Symmetry) for all ${\bf x}:=(x_1, x_2, x_3)\in \wt\Delta^o$ and ${\bf j}:=(j_1, j_2, j_3) \in \mathcal{B}$,
		$$\theta\left({\bf x}, {\bf j}\right) = \theta\left({\bf x}_+, {\bf j}_+\right)= \theta\left({\bf x}_-, {\bf j}_-\right)=-\theta(\widehat{\bf x},\widehat{\bf j});$$
		\item($\Gamma$--equivariance) for all ${\bf x}\in \wt\Delta^o$, all ${\bf j} \in \mathcal{B}$, and all $\gamma\in\Gamma$, 
		$$\theta\left(\gamma\cdot {\bf x}, {\bf j}\right) = \theta\left({\bf x}, {\bf j}\right).$$
	\end{enumerate}	
	We denote by $\mathcal{T}(\lambda,d;G)$ the set of $\lambda$--triangle data functions of dimension $d$ with values in $G$.
\end{definition}

We may associate to any $\theta\in \mathcal{T}(\lambda,d;G)$ a homology class $K(\theta)\in H_0(\partial_v N^o; G^{\Ac})$ as follows. Define 
\[\mathfrak{s}_\theta:\widetilde{\Delta}^o\to G^{\Ac}\] 
by setting
\[
\mathfrak{s}_\theta({\bf x})^{\mathbf{i}} = 
\begin{cases}
	\sum_{\mathbf{j} \in \mathcal{B} : j_2 = i_1} \theta(\mathbf{x},\mathbf{j}) &  \text{if $x_3 \prec x_2 \prec x_1$,}  \\
	\sum_{\mathbf{j} \in \mathcal{B} : j_2 = i_2} \theta(\mathbf{x},\mathbf{j}) & \text{if $x_1 \prec x_2 \prec x_3$},\end{cases}
\]
and define
\begin{align*}
	K(\theta) : = - \sum_{\tsf\in\mathcal S^o} \mathfrak{s}_\theta({\bf x}_{\tsf}) \cdot [q_{\tsf}] \in H_0(\partial_v N^o; G^\mathcal{A}),
\end{align*}
where for each $\tsf\in\mathcal S^o$, ${\bf x}_{\tsf}\in\widetilde\Delta^o$ is the triple of vertices chosen for $\tsf$ (see Section \ref{subsubsec:notation for hom}). The $\Gamma$--invariance of $\theta$ implies that $K(\theta)$ does not depend on the choices of the triples of vertices chosen for $\mathsf{t}$. 

Now, observe that, for any cocyclic pair $(\alpha,\theta)\in\Yc(\lambda,d;G)$, the component $\theta$ is a $\lambda$--triangle data function, so we may associate to $\theta$ the homology class $K(\theta)\in H_0(\partial_v N^o;G^{\Ac})$ as described above. At the same time, we may associate to $\alpha$ the homology class in $\varsigma_G(\alpha)\in H_1(N^o, \partial_v N^o; G^{\Ac})$ defined by
\[\varsigma_G(\alpha):=\sum_{\mathsf{R}\in\mathcal N^o}\alpha({\bf T}_{\mathsf{R}})\cdot[\ksf_{\mathsf{R}}]\in H_1(N^o, \partial_v N^o; G^{\Ac}),\]
where for each $\mathsf{R}\in\mathcal N^o$, ${\bf T}_{\mathsf{R}}$ is the pair of plaques chosen for $\mathsf{R}$ (see Section \ref{subsubsec:notation for hom}). Again, the $\Gamma$--invariance of $\alpha$ implies that $\varsigma_G(\alpha)$ does not depend on the choices of the pairs of plaques chosen for $\mathsf{R}$. 

The following proposition \cite[Proposition 4.14]{MMMZ1} gives a parameterization of $\Yc(\lambda,d;G)$ by a Lie subgroup of $H_1(N^o,\partial_v N^o;G^{\Ac})\times \mathcal{T}(\lambda,d;G)$. 

\begin{proposition}\label{prop: Homological}
	Let $N$ be a train track neighborhood of $\lambda$. The map
	\begin{align*}\label{eqn: E}
		E:\Yc(\lambda,d;G)\to H_1(N^o,\partial_v N^o;G^{\Ac})\times \mathcal{T}(\lambda,d;G)
	\end{align*}
	defined by $E:(\alpha,\theta)\mapsto (\varsigma_G(\alpha),\theta)$ is an embedding of Lie groups, whose image is
	\[\{ (c, \theta) \in H_1(N^o,\partial_v N^o;G^{\Ac}) \times \mathcal{T}(\lambda,d;G) \mid \iota_*(c)=-\widehat{c}\text{ and }\partial c = K(\theta) \}.\]
\end{proposition}

\subsection{Proof of Theorem \ref{thm: hom}}\label{proof of subgroup theorem}
Before going into details, we will explain the main difficulty of the proof of Theorem \ref{thm: hom}. Observe that the definition of the map $I_1$ given in equation (\ref{eq: I_1}) does not involve the values of $\alpha$ on ${\bf T}_{R^o}$ for the rectangles $R$ that lie in the maximal tree $M$. On the other hand, we know from Proposition \ref{prop: Homological} that $\alpha$ is determined by its values on ${\bf T}_{R^o}$ for all rectangles $R$ of $N$. As such, the crux of the proof of Theorem \ref{thm: hom} is to show that for any $(\alpha,\theta)\in\Yc(\lambda,d;G)$, one can recover from $I_1(\alpha,\theta)$ the components $\alpha({\bf T}_{R^o})$, as $R$ varies among the rectangles of $M$. We will prove this using the homological interpretation of $\Yc(\lambda,d;G)$. 

To set up the required homology machinery, recall that $\mathcal{N}$ denotes the set of rectangles of $N$, that $\iota:N^o\to N^o$ is the covering involution of $\pi_N$, and for every $t\in\mathcal S$, $t^o$ denotes the orientation on $t$ chosen by $M^o$. Then the standard generating set for $H_1(N^o, \partial_v N^o; G^\mathcal{A})$ decomposes as
\[\{[\ksf_{R^o}]\mid R\in\mathcal N\}\cup\{[\ksf_{\iota(R^o)}]\mid R\in\mathcal N\} , \] 
and, similarly, the standard generating set for $H_0(\partial_v N^o; G^\mathcal{A})$ decomposes as
\[\{[q_{ t^o}]\mid t\in\mathcal S\}\cup\{[q_{\iota( t^o)}]\mid t\in\mathcal S\} . \]

Let $\mathcal{M}$ be the set of rectangles of $M$. Observe that the maps
\begin{align*}
	\beta : (G^{\Ac})^{\mathcal N}=(G^{\Ac})^{\mathcal M}\times(G^{\Ac})^{\mathcal O \sqcup\mathcal U} \longrightarrow\{c\in H_1(N^o,\partial_v N^o;G^{\Ac})\mid \iota_*(c)=-\widehat{c}\}
\end{align*}
given by
\[\beta(\mathbf{u}) := \sum_{R \in \mathcal N} u_{R} \cdot [\mathsf{k}_{R^o}] + \widehat{u_{R}} \cdot [\mathsf{k}_{\iota(R^o)}]\quad\text{for all}\quad{\bf u}=(u_R)_{R\in\mathcal N}=((u_R^{\bf i})_{{\bf i}\in\mathcal A})_{R\in\mathcal N}\]
and 
\begin{align*}
	\delta : (G^{\Ac})^{\mathcal S } \longrightarrow \{e\in H_0(\partial_v N^o;G^{\Ac})\mid\iota_*(e)=-\widehat{e}\}
\end{align*}
given by
\[\delta(\mathbf{w}) := \sum_{t \in \mathcal S } w_t\cdot [q_{t^o}] - \widehat{w_t} \cdot [q_{\iota(t^o)}]\quad\text{for all}\quad{\bf w}=(w_t)_{t\in\mathcal S }=((w_t^{\bf i})_{{\bf i}\in\Ac})_{t\in\mathcal S }\]
are isomorphisms of groups. Indeed, the fact that $\beta$ and $\delta$ are injective group homomorphisms is clear from their definitions, while their surjectivity is a straightforward consequence of expressing the equation $\iota_*(c)=-\widehat{c}$ and $\iota_*(e)=-\widehat{e}$ in the standard generating sets of $H_1(N^o,\partial_v N^o;G^{\Ac})$ and $H_0(\partial_v N^o;G^{\Ac})$ respectively (see Section \ref{subsubsec:relative homology of No}).

The following lemma is the key technical step, which tells us the conditions needed to recover the components $(\alpha({\bf T}_{R^o}))_{R\in\mathcal M}$ from the image $I_1(\alpha,\theta)$.

\begin{lemma}\label{lem: existence}
	If ${\bf w} = ((w^{\bf i}_t)_{{\bf i} \in \Ac})_{t \in \mathcal{S}} \in (G^{\Ac})^{\mathcal S }$ and ${\bf v} = ((v^{\bf i}_R)_{{\bf i} \in \Ac})_{R \in \mathcal O \sqcup\mathcal U} \in (G^\mathcal A)^{\mathcal O \sqcup\mathcal U}$, then there is some ${\bf v}'\in (G^\mathcal A)^{\mathcal M}$ such that $\partial\beta({\bf v}',{\bf v})=\delta({\bf w})$ if and only if
	\[ \sum_{R \in \mathcal U ^r} ( v_R^{\mathbf{i}} + v_R^{\widehat{\mathbf{i}}}) - \sum_{R \in \mathcal U ^\ell} (v_R^{\mathbf{i}} + v_R^{\widehat{\mathbf{i}}} ) = \sum_{t\in\mathcal S}w^{\bf i}_t\]
	for all ${\bf i}\in\Ac$. Furthermore, if such ${\bf v}'$ exists, then it is unique.
\end{lemma}

The proof of Theorem \ref{thm: hom} also requires the following lemma which, for any $\theta\in\Tc(\lambda,d;G)$, gives a formula for $K(\theta)$ in the standard generating set for $H_0(\partial_v N^o,G^{\Ac})$. This lemma will be used to relate the equations in the statement of Lemma \ref{lem: existence} to the equations $\clubsuit({\bf i})$ in the statement of Theorem \ref{thm: hom}.

\begin{lemma}\label{lem: Ktheta}
	For any $\lambda$--triangle data function $\theta\in\Tc(\lambda,d;G)$, the homology class $K(\theta)\in H_0(\partial_v N^o;G^{\Ac})$ is given by
	\begin{align*}
		K(\theta)&=\sum_{t\in\mathcal S^\ell} \left(\left(\sum_{\mathbf{j} \in \mathcal{B} : j_2 = i_2} \theta^{\bf j}(\mathbf{x}_{t^{cw}})\right)_{{\bf i}\in\Ac} \cdot [q_{t^o}]-\left(\sum_{\mathbf{j} \in \mathcal{B} : j_2 = i_1} \theta^{\bf j}(\mathbf{x}_{t^{cw}})\right)_{{\bf i}\in\Ac} \cdot [q_{\iota(t^o)}]\right)\\
		&+\sum_{t\in\mathcal S^r} \left(\left(\sum_{\mathbf{j} \in \mathcal{B} : j_2 = i_2} \theta^{\bf j}(\mathbf{x}_{t^{cw}})\right)_{{\bf i}\in\Ac} \cdot [q_{\iota(t^o)}]-\left(\sum_{\mathbf{j} \in \mathcal{B} : j_2 = i_1} \theta^{\bf j}(\mathbf{x}_{t^{cw}})\right)_{{\bf i}\in\Ac} \cdot [q_{t^o}]\right).
	\end{align*}
	In particular, if we set ${\bf w}=((w_t^{\bf i})_{{\bf i}\in\Ac})_{t\in\mathcal S}\in (G^{\Ac})^{\mathcal S}$ to be the element given by
	\[w_t^{\bf i}:=\left\{\begin{array}{ll}
		\displaystyle\sum_{\mathbf{j} \in \mathcal{B} : j_2 = i_2} \theta^{\bf j}(\mathbf{x}_{t^{cw}})&\text{if } t\in\mathcal S^\ell\\
		\displaystyle-\sum_{\mathbf{j} \in \mathcal{B} : j_2 = i_1} \theta^{\bf j}(\mathbf{x}_{t^{cw}})&\text{if } t\in\mathcal S^r,
	\end{array}\right.\]
	then $K(\theta)=\delta({\bf w})$. 
\end{lemma}

Assuming Lemmas \ref{lem: existence} and \ref{lem: Ktheta}, we now finish the proof of Theorem \ref{thm: hom}.

\begin{proof}[Proof of Theorem \ref{thm: hom}]
	Let
	\[Y':=\{ (c, \theta) \in H_1(N^o,\partial_v N^o;G^{\Ac}) \times \mathcal{T}(\lambda,d;G) \mid\iota_*(c)=-\widehat{c}\text{ and }\partial c = K(\theta) \},\]
	and let $I_1':Y'\to (G^{\Ac})^{\mathcal O  \sqcup \mathcal U } \times (G^{\Bc})^{\mathcal S }$ be the map given by
	\[I_1'\left(\sum_{R\in\mathcal N}\left(u_{R}\cdot[\ksf_{R^o}]+\widehat{u_{R}}\cdot[\ksf_{\iota(R^o)}]\right),\theta\right)=\big((u_{R})_{R\in\mathcal O \sqcup\mathcal U},(\theta({\bf x}_{t^{cw}}))_{t\in\mathcal S}\big).\]
	Proposition \ref{prop: Homological} gives an isomorphism $E:\mathcal Y(\lambda,d;G)\to Y'$. From the definition of $E$ we can see that $I_1=I_1'\circ E$. Since $I_1$ is a morphism of Lie groups, it suffices to show that $I_1'$ is injective, and its image is $Y$.
	
	As a preliminary step, we describe an explicit parameterization of $\Tc(\lambda,d;G)$. From the definition of ${\bf x}_{\tsf}$ in Section \ref{subsubsec:notation for hom}, we see that for every ${\bf x}\in\widetilde\Delta^o$, there is a unique $\tsf\in\mathcal S^o$ and a unique $\gamma\in\Gamma$ such that ${\bf x}=\gamma\,{\bf x}_{\tsf}$. Thus, we may define a map
	\[\phi:\{{\bf z}\in(G^{\Bc})^{\mathcal S }\mid\blacklozenge(t, {\bf j})\text{ for all }{\bf j} \in \Bc\text{ and } t \in \mathcal S \}\to G^{\widetilde\Delta^o\times\Bc}\]
	by setting $\phi({\bf z}):\wt\Delta^o  \times \mathcal{B} \to G$ to be the map given by \[\phi({\bf z})(\gamma\,{\bf x}_{t^{cw}},{\bf j})=z_t^{\bf j}\quad\text{and}\quad\phi({\bf z})(\gamma\,{\bf x}_{\iota(t^{cw})},{\bf j})=-z_t^{\widehat{\bf j}}\] 
	for all $\gamma\in\Gamma$, $t\in\mathcal S$, and ${\bf j}\in\Bc$. It is straightforward to verify that $\phi$ is a group isomorphism onto $\Tc(\lambda,d;G)\subset G^{\widetilde\Delta^o\times\Bc}$, and its inverse 
	\[\phi^{-1}:\Tc(\lambda,d;G)\to\{{\bf z}\in(G^{\Bc})^{\mathcal S }\mid\blacklozenge(t, {\bf j})\text{ for all }{\bf j} \in \Bc\text{ and } t \in \mathcal S \}\]
	is given by $\phi^{-1}(\theta)=(\theta({\bf x}_{t^{cw}}))_{t\in\mathcal S}$.
	
	We will now verify that the image of $I_1'$ lies in $Y$. Pick any $(c,\theta)\in Y'$. Then $\theta\in\Tc(\lambda,d;G)$, so the parameterization $\phi$ of $\Tc(\lambda,d;G)$ given above implies that $I_1'(c,\theta)$ satisfies $\blacklozenge(t, {\bf j})$ for all ${\bf j} \in \Bc\text{ and } t \in \mathcal S $. Also, by the definition of $Y'$ and $\beta$, there is some $({\bf v}',{\bf v})\in (G^\mathcal A)^{\mathcal M}\times(G^\mathcal A)^{\mathcal O \sqcup\mathcal U}$ such that $c=\beta({\bf v}',{\bf v})$. Then by Lemma~\ref{lem: Ktheta}, we have that:
	\[\partial(\beta({\bf v}',{\bf v}))=\partial c=K(\theta)=\delta({\bf w}),\] 
	where
	${\bf w}\in (G^{\Ac})^{\mathcal S}$ is the element given by
	\[w_t^{\bf i}:=\left\{\begin{array}{ll}
		\displaystyle\sum_{\mathbf{j} \in \mathcal{B} : j_2 = i_2} \theta^{\bf j}(\mathbf{x}_{t^{cw}})&\text{if } t\in\mathcal S^\ell\\
		\displaystyle-\sum_{\mathbf{j} \in \mathcal{B} : j_2 = i_1} \theta^{\bf j}(\mathbf{x}_{t^{cw}})&\text{if } t\in\mathcal S^r.
	\end{array}\right.\]
	Thus, we may apply Lemma \ref{lem: existence} to deduce that $I_1'(c,\theta)$ satisfies $\clubsuit({\bf i})$ for all ${\bf i}\in\Ac$. Hence, the image of $I_1'$ lies in $Y$. 
	
	To finish the proof, it suffices to show that $I_1'$ is a bijection onto $Y$. Let $({\bf v},{\bf z})\in Y$. Since ${\bf z}$ satisfies $\blacklozenge(t, {\bf j})$ for all ${\bf j} \in \Bc$ and $t \in \mathcal S $, $\phi({\bf z})\in \Tc(\lambda,d;G)$. Then by Lemma \ref{lem: Ktheta}, $K(\phi({\bf z}))=\delta({\bf w})$, where ${\bf w}\in (G^{\Ac})^{\mathcal S}$ is the group element given by
	\[w_t^{\bf i}:=\left\{\begin{array}{ll}
		\displaystyle\sum_{\mathbf{j} \in \mathcal{B} : j_2 = i_2} z_t^{\bf j}&\text{if } t\in\mathcal S^\ell,\\
		\displaystyle-\sum_{\mathbf{j} \in \mathcal{B} : j_2 = i_1} z_t^{\bf j}&\text{if } t\in\mathcal S^r.
	\end{array}\right.\]
	Finally, since $({\bf v},{\bf z})\in Y$, Lemma \ref{lem: existence} implies that there is a unique ${\bf v}'\in (G^\mathcal A)^{\mathcal M}$ such that $\partial\beta({\bf v}',{\bf v})=K(\phi({\bf z}))$. We may thus define the map
	\[(I_1')^{-1}:Y\to Y'\quad\text{given by}\quad(I_1')^{-1}:({\bf v},{\bf z})\mapsto (\beta({\bf v}',{\bf v}),\phi({\bf z})).\]
	One can then verify that $(I_1')^{-1}$ is indeed the inverse of $I_1'$, hence proving that $I_1$ is a bijection.
	\end{proof}
	
	It now remains to prove Lemmas \ref{lem: existence} and \ref{lem: Ktheta}. We will do so in the subsequent subsections.
	
	\subsection{Proof of Lemma \ref{lem: existence}} 
	
	Let 
	\[
	i : \partial_v N^o \to N^o\quad\text{and}\quad i': \partial_v N^o \to M^o \sqcup \iota(M^o)\]	
	denote the natural inclusion maps, and let
	\[i_* : H_0(\partial_v N^o;\star) \to H_0(N^o;\star)\quad\text{and}\quad i_*' : H_0(\partial_v N^o;\star) \to H_0(M^o \sqcup \iota(M^o);\star)\]
	denote the associated homomorphisms in homology. (We write $\star$ to mean any coefficient group.) 
	
	Notice that both $\partial_v N^o$ and $M^o \sqcup \iota(M^o)$ are disjoint unions of contractible, path-connected components, so one can see that
	\[H_1(\partial_v N^o; \star) = 0=H_1(M^o \sqcup \iota(M^o); \star).\] 
	Moreover, since every connected component of $N^o$ and $M^o \sqcup \iota(M^o)$ intersects $\partial_v N^o$, we have that
	\[H_0(N^o, \partial_v N^o;\star)=0=H_0(M^o \sqcup \iota(M^o),\partial_v N^o;\star)\] as well. Therefore, the long exact sequences of the pairs $(M^o \sqcup \iota(M^o),\partial_v N^o)$ and $(N^o,\partial_v N^o)$ reduce to the following exact sequences:
	\begin{gather}
		0 \to H_1(M^o \sqcup \iota(M^o), \partial_v N^o;\star) \stackrel{\partial'}{\to} H_0(\partial_v N^o;\star) \stackrel{i_*'}{\to} H_0(M^o \sqcup \iota(M^o);\star) \to 0 , \label{eq:exact sequence M} \\
		0 \to H_1(N^o;\star) \to H_1(N^o, \partial_v N^o;\star) \stackrel{\partial}{\to} H_0(\partial_v N^o;\star) \stackrel{i_*}{\to} H_0(N^o;\star) \to 0, \label{eq:exact sequence N}
	\end{gather}
	where $\partial$ and $\partial'$ denote the usual boundary maps.
	Furthermore, the inclusion of pairs 
	\[j : (M^o \sqcup \iota(M^o),\partial_v N^o) \to (N^o,\partial_v N^o)\] determines a morphism between these exact sequences. In other words, the following diagram commutes:
	{\small
		\begin{center}
			\begin{tikzpicture}[scale=1.7]
				\node (h1nv) at (0,0) {$0$};
				\node (h1n) at (0.8,0) {$H_1(N^o;\star)$};
				\node (h1nnv) at (2.3,0) {$H_1(N^o, \partial_v N^o;\star)$};
				\node (h0nv) at (4.3,0) {$H_0(\partial_v N^o;\star)$};
				\node (h0n) at (6,0) {$H_0(N^o;\star)$};
				\node (h0nnv) at (7.2,0) {$0$};
				\node (h1m) at (0.8,1) {$0$};
				\node (h1mmv) at (2.3,1) {$H_1(M^o \sqcup \iota(M^o), \partial_v N^o;\star)$};
				\node (h0mv) at (4.3,1) {$H_0(\partial_v N^o;\star)$};
				\node (h0m) at (6,1) {$H_0(M^o \sqcup \iota(M^o);\star)$};
				\node (h0mmv) at (7.2,1) {$0$};
				\draw[->] (h1nv) to node[above]{} (h1n);
				\draw[->] (h1n) to node[above]{} (h1nnv);
				\draw[->] (h1nnv) to node[above]{$\partial$} (h0nv);
				\draw[->] (h0nv) to node[above]{$i_*$} (h0n);
				\draw[->] (h0n) to node[above]{} (h0nnv);
				\draw[->] (h1m) to node[above]{} (h1mmv);
				\draw[->] (h1mmv) to node[above]{$\partial'$} (h0mv);
				\draw[->] (h0mv) to node[above]{$i_*'$} (h0m);
				\draw[->] (h0m) to node[above]{} (h0mmv);
				\draw[->] (h1mmv) to node[left]{$j_*$} (h1nnv);
				\draw[->] (h0mv) to node[left]{$\mathrm{id}$} (h0nv);
			\end{tikzpicture}
	\end{center}	}
	
	Using this, we will prove the following lemma, which is the key computation needed to prove Lemma \ref{lem: existence}.
	
	\begin{lemma}\label{lem: in coordinates}
		Let ${\bf w}\in (G^{\Ac})^{\mathcal S }$ and ${\bf v}\in (G^\mathcal A)^{\mathcal O \sqcup\mathcal U}$. Then
		\[ \sum_{R \in \mathcal U ^r} ( v_R^{\mathbf{i}} + v_R^{\widehat{\mathbf{i}}}) - \sum_{R \in \mathcal U ^\ell} (v_R^{\mathbf{i}} + v_R^{\widehat{\mathbf{i}}} ) = \sum_{t\in\mathcal S}w^{\bf i}_t\]
		if and only if $\delta({\bf w})-\partial(\beta({\bf 0},{\bf v}))$ lies in ${\rm Image}(\partial')$.
	\end{lemma}
	
	\begin{proof}
		By the exactness of \eqref{eq:exact sequence M}, we need to show that 
		\[ \sum_{R \in \mathcal U ^r} ( v_R^{\mathbf{i}} + v_R^{\widehat{\mathbf{i}}}) - \sum_{R \in \mathcal U ^\ell} (v_R^{\mathbf{i}} + v_R^{\widehat{\mathbf{i}}} ) = \sum_{t\in\mathcal S}w^{\bf i}_t\]
		if and only if $\delta({\bf w})-\partial(\beta({\bf 0},{\bf v}))$ lies in ${\rm Kernel}(i_*')$. 
		
		For each rectangle $R$ of $N$ that is not in $M$, let $ \tsf_{R,+}, \tsf_{R,-}\in\mathcal S^o$ be the vertical boundary components of $N^o$ that contain the forward and backward endpoints of $\ksf_{R^o}$ respectively. By definition,
		\[\beta({\bf 0},{\bf v})=\sum_{R \in \mathcal O \sqcup\mathcal U} v_{R} \cdot [\mathsf{k}_{R^o}] + \widehat{v_{R}} \cdot [\mathsf{k}_{\iota(R^o)}]=\sum_{R \in \mathcal O \sqcup\mathcal U} v_{R} \cdot [\mathsf{k}_{R^o}] - \widehat{v_{R}} \cdot \iota_*([\mathsf{k}_{R^o}]),\]
		so
		\[\partial(\beta({\bf 0},{\bf v}))=\sum_{R \in \mathcal O \sqcup\mathcal U} v_{R} \cdot [q_{\tsf_{R,+}}]-v_{R} \cdot [q_{\tsf_{R,-}}] - \widehat{v_{R}} \cdot \iota_*([q_{\tsf_{R,+}}])+\widehat{v_{R}} \cdot \iota_*([q_{\tsf_{R,-}}]).\]
		
		Since $M^o\sqcup\iota(M^o)$ has two connected components, namely $M^o$ and $\iota(M^o)$, if we choose a base point $q_{M^o}\in M^o$ and set $q_{\iota(M^o)}:=\iota(q_{M^o})$, then
		\[H_0(M^o \sqcup \iota(M^o);G^{\mathcal A})=G^{\mathcal A}\cdot [q_{M^o}]\oplus G^{\mathcal A}\cdot [q_{\iota(M^o)}].\]
		Furthermore, for all vertical boundary components $\tsf\in\mathcal S^o$ of $N^o$ and $v\in G^{\mathcal A}$, we have \[i_*(v\cdot [q_{\tsf}])=\left\{\begin{array}{ll}
			v\cdot [q_{M^o}]&\text{if }\tsf\subset M^o,\\
			v\cdot [q_{\iota(M^o)}]&\text{if }\tsf\subset \iota(M^o).\\
		\end{array}\right.\]
		This implies that:
		\begin{itemize}
			\item If $R$ is an orientable rectangle, then both $\tsf_{R,+}$ and $\tsf_{R,-}$ lie in $M^o$, so
			\[i_*'(v_{R} \cdot [q_{\tsf_{R,+}}]-v_{R} \cdot [q_{\tsf_{R,-}}])=0=i_*'(\widehat{v_{R}} \cdot \iota_*([q_{\tsf_{R,+}}])-\widehat{v_{R}} \cdot \iota_*([q_{\tsf_{R,-}}])).\]
			\item If $R$ is a left unorientable rectangle, then $\tsf_{R,+}$ and $\tsf_{R,-}$ lie in $\iota(M^o)$ and $M^o$ respectively, so
			\[i_*'(v_{R} \cdot [q_{\tsf_{R,+}}]-v_{R} \cdot [q_{\tsf_{R,-}}])=v_R\cdot[q_{\iota(M^o)}]-v_R\cdot[q_{M^o}]\]
			and
			\[i_*'(-\widehat{v_{R}} \cdot \iota_*([q_{\tsf_{R,+}}])+\widehat{v_{R}} \cdot \iota_*([q_{\tsf_{R,-}}]))=\widehat{v_R}\cdot[q_{\iota(M^o)}]-\widehat{v_R}\cdot[q_{M^o}].\]
			\item If $R$ is a right unorientable rectangle, then $\tsf_{R,+}$ and $\tsf_{R,-}$ lie in $M^o$ and $\iota(M^o)$ respectively, so
			\[i_*'(v_{R} \cdot [q_{\tsf_{R,+}}]-v_{R} \cdot [q_{\tsf_{R,-}}])=v_R\cdot[q_{M^o}]-v_R\cdot[q_{\iota(M^o)}]\]
			and
			\[i_*'(-\widehat{v_{R}} \cdot \iota_*([ t^o_{R,+}])+\widehat{v_{R}} \cdot \iota_*([ t^o_{R,-}]))=\widehat{v_R}\cdot[q_{M^o}]-\widehat{v_R}\cdot[q_{\iota(M^o)}].\]
		\end{itemize}
		
		One can then compute that
		\[i_*'(\partial(\beta({\bf 0},{\bf v})))=\left(\sum_{R \in \mathcal U^\ell} (v_{R}+\widehat{v_{R}})-\sum_{R \in \mathcal U^r} (v_{R}+\widehat{v_{R}}) \right)\cdot ([q_{\iota(M^o)}]-[q_{M^o}]).\]
		At the same time,
		\[i_*'(\delta({\bf w}))=\sum_{t \in \mathcal S } w_t\cdot [q_{M^o}] - \sum_{t \in \mathcal S } \widehat{w_t} \cdot [q_{\iota(M^o)}].\]
		Thus, $\delta({\bf w})-\partial(\beta({\bf 0},{\bf v}))$ lies in ${\rm Kernel}(i_*')$ if and only if
		\[\sum_{t \in \mathcal S } w_t+\sum_{R \in \mathcal U^\ell} (v_{R}+\widehat{v_{R}})-\sum_{R \in \mathcal U^r} (v_{R}+\widehat{v_{R}}) =0.\qedhere\]
	\end{proof}
	
	Using Lemma \ref{lem: in coordinates}, we now finish the proof of Lemma \ref{lem: existence}. 
	\begin{proof}[Proof of Lemma \ref{lem: existence}]
		First, suppose that there is some ${\bf v}'\in (G^\mathcal A)^{\mathcal M}$ such that $\partial(\beta({\bf v}',{\bf v}))=\delta({\bf w})$. Notice that the image of 
		\[j_*:H_1(M^o \sqcup \iota(M^o), \partial_v N^o;G^{\mathcal A})\to H_1(N^o, \partial_v N^o;G^{\mathcal A})\] contains $\beta((G^{\mathcal A})^{\mathcal M}\times\{{\bf 0}\})$. Thus, there is some $x'\in H_1(M^o \sqcup \iota(M^o), \partial_v N^o;G^{\mathcal A})$ such that $j_*(x')=\beta({\bf v}',{\bf 0})$. In particular, 
		\[\delta({\bf w})-\partial(\beta({\bf 0},{\bf v}))=\partial(\beta({\bf v}',{\bf 0}))=(\partial\circ j_*)(x')=\partial'(x'),\]
		so Lemma \ref{lem: in coordinates} implies that 
		\[ \sum_{R \in \mathcal U ^r} ( v_R^{\mathbf{i}} + v_R^{\widehat{\mathbf{i}}}) - \sum_{R \in \mathcal U ^\ell} (v_R^{\mathbf{i}} + v_R^{\widehat{\mathbf{i}}} ) = \sum_{t\in\mathcal S}w^{\bf i}_t.\]
		
		Conversely, suppose that 
		\[ \sum_{R \in \mathcal U ^r} ( v_R^{\mathbf{i}} + v_R^{\widehat{\mathbf{i}}}) - \sum_{R \in \mathcal U ^\ell} (v_R^{\mathbf{i}} + v_R^{\widehat{\mathbf{i}}} ) = \sum_{t\in\mathcal S}w^{\bf i}_t.\]
		Then Lemma \ref{lem: in coordinates} implies that there is some $x'\in H_1(M^o \sqcup \iota(M^o), \partial_v N^o;G^{\mathcal A})$ such that $\partial'(x')=\delta({\bf w})-\partial(\beta({\bf 0},{\bf v}))$. To prove the first claim of the lemma, it now suffices to show that there is some ${\bf v'}\in (G^\mathcal A)^{\mathcal M}$ such that $j_*(x')=\beta({\bf v'},{\bf 0})$; indeed, if this were the case, then 
		\[\partial(\beta({\bf v}',{\bf v}))=\partial(\beta({\bf 0},{\bf v}))+\partial(\beta({\bf v}',{\bf 0}))
		=\partial(\beta({\bf 0},{\bf v}))+\partial'(x')=\delta({\bf w}).\]
		From the definition of $j_*$, it suffices to check that $j_*(x')$ lies in the image of $\beta$, i.e. that $\iota_*(j_*(x')) = - \widehat{j_*(x')}$. Notice that $j_*$ commutes with $\iota_*$ and $\widehat{\cdot}$, and $\partial'$ is injective, so it suffices to check that $\partial'(\iota_*(x') + \widehat{x'})=0$. We compute:
		\begin{align*}
			\partial'(\iota_*(x') + \widehat{x'}) 
			& = \iota_*(\partial'(x')) + \widehat{\partial'(x')} \\
			& = \iota_*(\delta(\mathbf{w}) - \partial(\beta({\bf 0},\mathbf{v}))) + \widehat{\delta(\mathbf{w})} - \partial(\widehat{\beta({\bf 0},\mathbf{v}))} \\
			& = \iota_*(\delta(\mathbf{w})) + \widehat{\delta(\mathbf{w})} - \partial\left(\iota_*(\beta({\bf 0},\mathbf{v}))+\widehat{\beta({\bf 0},\mathbf{v})}\right) = 0.
		\end{align*}
		
		Finally, we prove the uniqueness claim of the lemma. Let ${\bf v}',{\bf v}''\in (G^\mathcal A)^{\mathcal M}$ such that
		$\partial\beta({\bf v}',{\bf v})=\partial\beta({\bf v}'',{\bf v})$. We previously observed that the image of $j_*$ contains $\beta((G^{\mathcal A})^{\mathcal M}\times\{{\bf 0}\})$, so there are $x',x''\in H_1(M^o \sqcup \iota(M^o), \partial_v N^o;G^{\mathcal A})$ such that $j_*(x')=\beta({\bf v}',{\bf 0})$ and $j_*(x'')=\beta({\bf v}'',{\bf 0})$. Then
		\[\partial'(x')=(\partial\circ j_*)(x')=\partial(\beta({\bf v}',{\bf 0}))=\partial(\beta({\bf v}',{\bf v}))-\partial(\beta({\bf 0},{\bf v}))\]
		Similarly, $\partial'(x'')=\partial(\beta({\bf v}'',{\bf v}))-\partial(\beta({\bf 0},{\bf v}))$, so $\partial'(x')=\partial'(x'')$. By the exactness of equation \eqref{eq:exact sequence M}, we see that $\partial'$ is injective. Thus, $x'=x''$, and hence ${\bf v}'={\bf v}''$.
	\end{proof}
	
	\subsection{Computation of $K(\theta)$} Finally, we prove Lemma \ref{lem: Ktheta}.
	
	\begin{proof}[Proof of Lemma \ref{lem: Ktheta}]
		By definition (see Section \ref{subsubsec:homological interpretation of cocyclic pairs}), we have
		\[\mathfrak s _{\theta}({\bf x}_{ t^{cw}})=\left(\sum_{\mathbf{j} \in \mathcal{B} : j_2 = i_1} \theta^{\bf j}(\mathbf{x}_{t^{cw}})\right)_{{\bf i}\in\Ac}\quad\text{and}\quad\mathfrak s _{\theta}({\bf x}_{\iota(t^{cw})})=\left(\sum_{\mathbf{j} \in \mathcal{B} : j_2 = i_2} \theta^{\bf j}(\mathbf{x}_{\iota(t^{cw})})\right)_{{\bf i}\in\Ac}.\]
		So we see that
		\begin{align*}
			K(\theta)&=- \sum_{t\in\mathcal S} \left(\mathfrak s_{\theta}({\bf x}_{ t^{cw}}) \cdot [q_{ t^{cw}}]+\mathfrak s_{\theta}({\bf x}_{\iota( t^{cw})}) \cdot [q_{\iota( t^{cw})}]\right)\\
			&=-\sum_{t\in\mathcal S} \left(\left(\sum_{\mathbf{j} \in \mathcal{B} : j_2 = i_1} \theta^{\bf j}(\mathbf{x}_{t^{cw}})\right)_{{\bf i}\in\Ac} \cdot [q_{t^{cw}}]-\left(\sum_{\mathbf{j} \in \mathcal{B} : j_2 = i_2} \theta^{\bf j}(\mathbf{x}_{t^{cw}})\right)_{{\bf i}\in\Ac} \cdot [q_{\iota(t^{cw})}]\right)\\
			&=-\sum_{t\in\mathcal S^\ell} \left(\left(\sum_{\mathbf{j} \in \mathcal{B} : j_2 = i_1} \theta^{\bf j}(\mathbf{x}_{t^{cw}})\right)_{{\bf i}\in\Ac} \cdot [q_{\iota(t^o)}]-\left(\sum_{\mathbf{j} \in \mathcal{B} : j_2 = i_2} \theta^{\bf j}(\mathbf{x}_{t^{cw}})\right)_{{\bf i}\in\Ac} \cdot [q_{t^o}]\right)\\
			&-\sum_{t\in\mathcal S^r} \left(\left(\sum_{\mathbf{j} \in \mathcal{B} : j_2 = i_1} \theta^{\bf j}(\mathbf{x}_{t^{cw}})\right)_{{\bf i}\in\Ac} \cdot [q_{t^o}]-\left(\sum_{\mathbf{j} \in \mathcal{B} : j_2 = i_2} \theta^{\bf j}(\mathbf{x}_{t^{cw}})\right)_{{\bf i}\in\Ac} \cdot [q_{\iota(t^o)}]\right),
		\end{align*}
		where the first equality is by the definition of $K(\theta)$, the second equality is a consequence of the symmetry properties of $\theta$ and the third equality follows from equation \eqref{eqn: hs}. 
	\end{proof}

\section{Global topology of $\mathcal Y(\lambda,d;\Cb/2\pi i\Zb)$} \label{sec:topology_y}

In Theorem \ref{thm: hom}, we proved that the space $\Yc(\lambda,d;G)$ of $\lambda$--cocyclic pairs of dimension $d$ with values in an Abelian Lie group $G$ is isomorphic to the Lie group
\[
Y:= \left\{({\bf v}, {\bf z}) \in (G^{\Ac})^{\mathcal{O} \sqcup \mathcal{U}} \times (G^{\Bc})^{\mathcal{S}} \left| \,\,\, \parbox{5cm}{$\blacklozenge(t, {\bf j})$ for all ${\bf j} \in \Bc$ and $t \in \mathcal S$, and $\clubsuit({\bf i})$ for all ${\bf i} \in \Ac$.} \right. \right\} ,
\]
where 
\begin{gather*}
	\tag*{$\blacklozenge(t, {\bf j})$} {z}_t^{\bf j} = {z}_{t_+}^{{\bf j}_+} = {z}_{t_-}^{{\bf j}_-} , \\
	\tag*{$\clubsuit({\bf i})$} \sum_{R \in \mathcal{U}^r} ( v_R^{\mathbf{i}} + v_R^{\widehat{\mathbf{i}}}) - \sum_{R \in \mathcal{U}^\ell} (v_R^{\mathbf{i}} + v_R^{\widehat{\mathbf{i}}} ) = \sum_{t\in \mathcal{S}^\ell} \sum_{{\bf j} \in \Bc : j_2 = i_2} z^{\bf j}_t - \sum_{t\in \mathcal{S}^r} \sum_{{\bf j} \in \Bc : j_2 = i_1} z^{\bf j}_t.
\end{gather*}
Here, ${\bf v}=(v_R)_{R\in\mathcal O\sqcup\mathcal U}=((v_R^{\bf i})_{{\bf i}\in\Ac})_{R\in\mathcal O\sqcup\mathcal U}$ and ${\bf z}=(z_t)_{t\in\mathcal S}=((z_t^{\bf j})_{{\bf j}\in\Bc})_{t\in\mathcal S}$. (See Section \ref{subsubsec:notation for hom} for notation.) The purpose of this section is to determine the isomorphism class of $Y$ (see Theorem \ref{Thm-general} from the introduction). To describe the precise statement established here, let us start by introducing some notation.

For each plaque $T$ of $\lambda$, choose once and for all some vertical boundary component $t(T)\in\mathcal S$ that lies in the interior of $T$. Recall that $\Bc$ denotes the set of triples of positive integers that sum up to $d$. Let $\Bc^* \subset \Bc$ be
\begin{align*}
	\Bc^*&:=\left\{{\bf j}=(j_1,j_2,j_3)\in\Bc\mid j_1,j_2,j_3\le \left\lfloor \frac{d-1}{2}\right\rfloor\right\},
\end{align*}
see Figure \ref{fig:boundary-rectangles}, and define the map ${\rm tor}'_d:Y\to G$ as follows:
\begin{itemize}
	\item If $d$ is odd, then
	\begin{align*}
		{\rm tor}'_d({\bf v},{\bf z}): = - \sum_{T\in\Delta} \sum_{{\bf j} \in \Bc^*} z^{\bf j}_{t(T)}.
	\end{align*}
	\item If $d$ is even, then
	\begin{align*}
		{\rm tor}'_d({\bf v},{\bf z}) &:= - \sum_{T\in\Delta} \sum_{{\bf j} \in \Bc^*} z^{\bf j}_{t(T)} + \sum_{R \in \mathcal{U}^r} v_{R}^{{\bf i}^0} - \sum_{R \in \mathcal{U}^\ell} v_{R}^{{\bf i}^0} - \sum_{t\in\mathcal S^\ell}\sum_{{\bf j}\in\Bc^0}z_t^{\bf j},\\
		& = - \sum_{T\in\Delta} \sum_{{\bf j} \in \Bc^*} z^{\bf j}_{t(T)} - \sum_{R \in \mathcal{U}^r} v_{R}^{{\bf i}^0} + \sum_{R \in \mathcal{U}^\ell} v_{R}^{{\bf i}^0} - \sum_{t\in\mathcal S^r}\sum_{{\bf j}\in\Bc^0}z_t^{\bf j},
	\end{align*}
	where ${\bf i}^0:=(\frac{d}{2},\frac{d}{2})$ and $\Bc^0:=\{{\bf j}=(j_1,j_2,j_3)\in\Bc\mid j_2=\frac{d}{2}\}$. 
\end{itemize}
Note that ${\rm tor}_d'$ does not depend on the choice of the vertical boundary components $\{t(T) \mid T \in \Delta\}$, and that the second equality for $d$ even follows from equation $\clubsuit({\bf i}_0)$. The main result of this section states:

\begin{theorem}\label{thm: topology}
	If $G$ is an Abelian Lie group, then there is an isomorphism of groups
	\[I_2:Y \to G^{(d^2-1)(2g-2)} \times G_d,\]
	where $G_d:=\{g\in G \mid d\cdot g=e\}$ is the \emph{$d$--torsion} of $G$.
	Furthermore, if $$\pi_{\rm tor}:G^{(d^2-1)(2g-2)}\times G_d\to G_d$$ denotes the natural projection, then 
	\[\pi_{\rm tor}\circ I_2({\bf v},{\bf z})={\rm tor}_d'({\bf v},{\bf z}).\]
\end{theorem}

For $G=\Rb$, Theorem \ref{thm: topology} was proven in Bonahon-Dreyer \cite[Proposition 8.2]{{BoD}}.

To establish the topology of the space of $d$--pleated surfaces, we specialize Theorem~\ref{thm: topology} to the case where $G=\Cb/2\pi i\Zb$, in which case the $d$--torsion subroup $G_d$ is generated by the element $\frac{2 \pi i}{d} \in \Cb/2\pi i \Zb$ and it is isomorphic to the finite cyclic group $\Zb_d$. Define the map
\[{\rm tor}_d:={\rm tor}_d'\circ I_1\circ \mathfrak{sb}_d:\mathfrak R(\lambda,d)\to \Zb_d,\]
where $I_1$ is the map from Theorem \ref{thm: hom} that sends the space of $\lambda$--cocyclic pair of dimension $d$ with values in $\Cb/2\pi i\Zb$ to a distinguished subgroup of $$((\Cb/2\pi i\Zb)^{\Ac})^{\mathcal O  \sqcup \mathcal U } \times ((\Cb/2\pi i\Zb)^{\Bc})^{\mathcal S }$$, and $\mathfrak{sb}_d$ is the shear-bend parameterization  given by Theorem~\ref{thm: parameterization}. Then, for $\rho\in\mathcal R(\lambda,d)$, set ${\rm tor}_d(\rho):={\rm tor}_d([\rho])$. 

The following is a consequence of Theorem~\ref{thm: topology}.

\begin{corollary}
	The space $\mathfrak R(\lambda,d)$ has $d$ connected components, each of which is real analytically diffeomorphic to
	\[\Rb^{(d^2-1)(2g-2)}\times (\Rb/2\pi \Zb)^{(d^2-1)(2g-2)}.\] 
	Furthermore, the map ${\rm tor}_d:\mathfrak R(\lambda,d)\to\Zb_d$ descends to a bijection between the set of connected components of $\mathfrak R(\lambda,d)$ and $\Zb_d$.
\end{corollary}

\begin{proof}
	By Theorem \ref{thm: parameterization}, the map 
	\[\mathfrak{sb}_d:\mathfrak{R}(\lambda,d)\to\Yc(\lambda,d;\Cb/2\pi i\Zb)=\mathcal{Y}(\lambda,d;\Rb) +i\mathcal{Y}(\lambda,d;\Rb/2\pi \Zb)\]
	is a biholomorphism onto $\mathcal{C}(\lambda, d) +i \mathcal{Y}(\lambda,d;\Rb/2\pi \Zb)$, where $\mathcal C(\lambda,d)$ is an open convex polyhedral cone in the real vector space $\mathcal{Y}(\lambda,d;\Rb)$. By specializing Theorem~\ref{thm: hom} and Theorem~\ref{thm: topology} to the case where $G=\Cb/2\pi i\Zb$, we see that the map 
	\[I_2\circ I_1:\mathcal Y(\lambda,d;\Cb/2\pi i\Zb)\to \Rb^{(d^2-1)(2g-2)}\times (\Rb/2\pi \Zb)^{(d^2-1)(2g-2)}\times\Zb_d\]
	is an isomorphism of Lie groups. Thus,
	\[I_2\circ I_1\circ\mathfrak{sb}_d:\mathfrak{R}(\lambda,d)\to \mathcal{C}(\lambda, d)\times (\Rb/2\pi \Zb)^{(d^2-1)(2g-2)}\times\Zb_d\]
	is a real analytic diffeomorphism, and ${\rm tor}_d=\pi_{\rm tor}\circ I_2\circ I_1\circ\mathfrak{sb}_d$. The corollary follows immediately.
\end{proof}

Before proving Theorem \ref{thm: topology}, we record the expression for ${\rm tor}_d$ in terms of the shear-bend $\lambda$--cocyclic pair of $\rho$.

\begin{remark}\label{tor formula}
	Let $\rho\in\mathcal R(\lambda,d)$, and let $(\alpha,\theta)=\mathfrak{sb}_d([\rho])$. In this case, we have:
	\begin{itemize}
		\item If $d$ is odd, then
		\[{\rm tor}_d(\rho):=-\sum_{T\in\Delta}\sum_{{\bf j}\in\mathcal B^*}\theta({\bf x}_T,{\bf j}),\]
		where ${\bf x}_T\in\widetilde\Delta^o$ is some (any) clockwise ordering of the vertices of some (any) lift to $\widetilde S$ of $T$. 
		\item If $d$ is even, then
		\begin{align*}
			{\rm tor}_d(\rho) &:= -\sum_{T\in\Delta} \sum_{{\bf j} \in \Bc^*} \theta({\bf x}_T,{\bf j}) + \sum_{R \in \mathfrak{U}^r} \alpha({\bf T}_{R},{\bf i}^0) - \sum_{R \in \mathfrak{U}^\ell} \alpha({\bf T}_{R},{\bf i}^0)-\sum_{t\in\mathcal S^\ell}\sum_{{\bf j}\in\Bc^0}\theta({\bf x}_{t},{\bf j})\\
			& = -\sum_{T\in\Delta} \sum_{{\bf j} \in \Bc^*} \theta({\bf x}_T,{\bf j}) - \sum_{R \in \mathfrak{U}^r} \alpha({\bf T}_{R},{\bf i}^0) + \sum_{R \in \mathfrak{U}^\ell} \alpha({\bf T}_{R},{\bf i}^0)-\sum_{t\in\mathcal S^r}\sum_{{\bf j}\in\Bc^0}\theta({\bf x}_{t},{\bf j}),
		\end{align*}
		where:
		\begin{itemize}
			\item[$\circ$] ${\bf x}_T\in\widetilde\Delta^o$ is as above.
			\item[$\circ$] ${\bf T}_R\in \wt\Delta^{2*}$ is some (any) ordering of the pair of plaques of $\widetilde\lambda$ that contain the horizontal boundary components of some (any) lift of $R$ to $\widetilde S$.
			\item[$\circ$] ${\bf x}_t=(x_{t,1},x_{t,2},x_{t,3})\in\widetilde\Delta^o$ is the clockwise ordering of the vertices of the plaque of $\widetilde\lambda$ that contains some (any) lift $\widetilde t$ to $\widetilde S$ of $t$, such that the geodesic with $x_{t,1}$ and $x_{t,3}$ as its endpoints does not intersect the switch of $\widetilde N$ that contains $\widetilde t$.
		\end{itemize}
	\end{itemize}
\end{remark}

\subsection{Proof of Theorem \ref{thm: topology}}

In order to describe the isomorphism class of the group $Y$, we will need to examine the equations $\blacklozenge(t, {\bf j})$ and $\clubsuit({\bf i})$ in great detail. As in the Gauss reduction algorithm for systems of linear equations, we aim to determine a minimal set of components in $(G^{\Ac})^{\mathcal{O} \sqcup \mathcal{U}} \times (G^{\Bc})^{\mathcal{S}}$ that uniquely characterizes every element $({\bf v}, {\bf z})$ in $Y$. The description of such minimal set of components is quite technical and it requires us to introduce suitable subsets of the set of indices $\Ac$ and $\Bc$. In particular, we set:
\begin{align*}
	\Ac'&:=\left\{{\bf i}=(i_1,i_2)\in\Ac\mid i_1\le \left\lfloor\frac{d-1}{2}\right\rfloor\right\},\\
	\Ac''&:=\left\{{\bf i}=(i_1,i_2)\in\Ac\mid i_1\le \left\lceil\frac{d-1}{2}\right\rceil\right\},\\
	\Bc'&:=\left\{{\bf j}=(j_1,j_2,j_3)\in\Bc\mid j_2=1\text{ and }j_3\le \left\lfloor\frac{d-3}{2}\right\rfloor\right\},\\
	\Bc''&:=\left\{{\bf j}=(j_1,j_2,j_3)\in\Bc\mid j_2=1\text{ and }j_3\le \left\lceil\frac{d-3}{2}\right\rceil\right\}.
\end{align*}
Observe that when $d$ is odd, $\Ac'=\Ac''$ and $\Bc'=\Bc''$. Instead, when $d$ is even, $\Ac''=\Ac'\cup\{{\bf i}^0\}$ and $\Bc''=\Bc'\cup\{{\bf j}^0\}$, where ${\bf i}^0=(\frac{d}{2},\frac{d}{2})$ and ${\bf j}^0=(\frac{d}{2},1,\frac{d-2}{2})$. See Figure \ref{fig:boundary-rectangles}.

Moreover, for each ${\bf i}\in\Ac$, we denote by $\spadesuit({\bf i})$ the equation on $(G^{\Ac})^{\mathcal{O} \sqcup \mathcal{U}} \times (G^{\Bc})^{\mathcal{S}}$ given by 
\[\tag*{$\spadesuit({\bf i})$}\label{eq:simple} \sum_{t \in \mathcal S} \sum_{{\bf j} \in \Bc : j_2 = i_2} z_t^{\bf j} = \sum_{t \in \mathcal S} \sum_{{\bf j} \in \Bc : j_2 = i_1} z_t^{\bf j},\]
where as usual, for any $({\bf v},{\bf z})\in (G^{\Ac})^{\mathcal{O} \sqcup \mathcal{U}} \times (G^{\Bc})^{\mathcal{S}}$, we denote ${\bf v}=(v_R)_{R\in\mathcal O\sqcup\mathcal U}=((v_R^{\bf i})_{{\bf i}\in\Ac})_{R\in\mathcal O\sqcup\mathcal U}$ and ${\bf z}=(z_t)_{t\in\mathcal S}=((z_t^{\bf j})_{{\bf j}\in\Bc})_{t\in\mathcal S}$.

\begin{figure}[h!]
	\centering
	\def\svgwidth{5.5cm}
	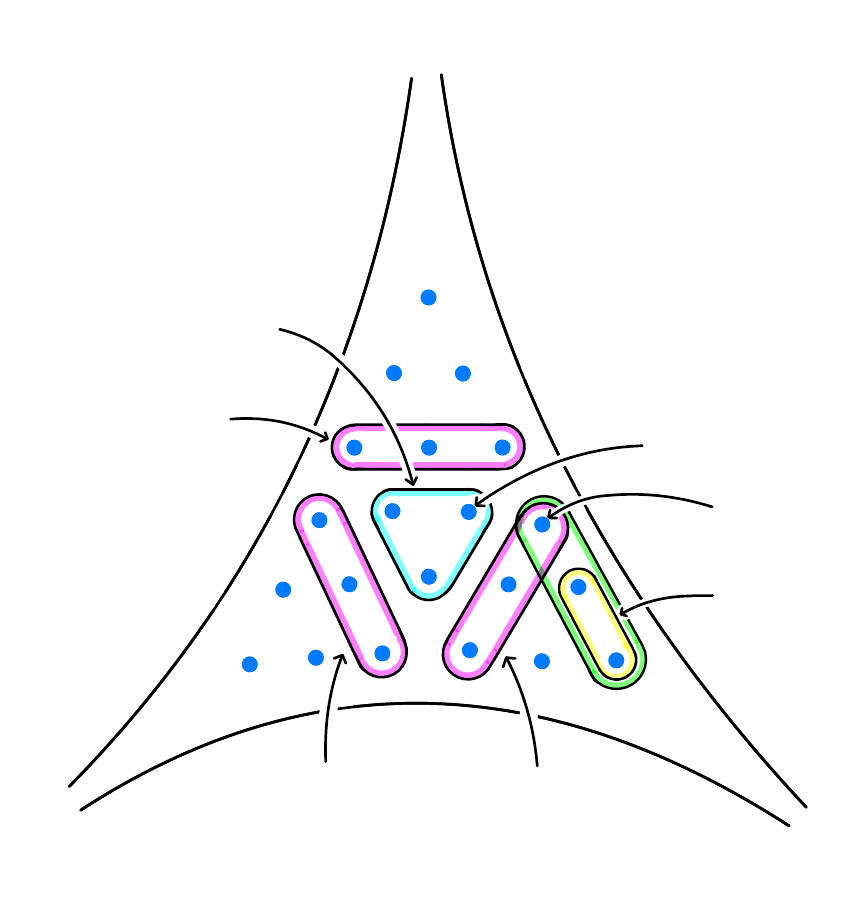
	\def\svgwidth{5.5cm}
	%% Creator: Inkscape 1.2.2 (b0a84865, 2022-12-01), www.inkscape.org
%% Accompanies image file 'figure2b.pdf' (pdf, eps, ps)
\begingroup%
  \makeatletter%
  \providecommand\color[2][]{%
    \errmessage{(Inkscape) Color is used for the text in Inkscape, but the package 'color.sty' is not loaded}%
    \renewcommand\color[2][]{}%
  }%
  \providecommand\transparent[1]{%
    \errmessage{(Inkscape) Transparency is used (non-zero) for the text in Inkscape, but the package 'transparent.sty' is not loaded}%
    \renewcommand\transparent[1]{}%
  }%
  \providecommand\rotatebox[2]{#2}%
  \newcommand*\fsize{\dimexpr\f@size pt\relax}%
  \newcommand*\lineheight[1]{\fontsize{\fsize}{#1\fsize}\selectfont}%
  \ifx\svgwidth\undefined%
    \setlength{\unitlength}{409.43943787bp}%
    \ifx\svgscale\undefined%
      \relax%
    \else%
      \setlength{\unitlength}{\unitlength * \real{\svgscale}}%
    \fi%
  \else%
    \setlength{\unitlength}{\svgwidth}%
  \fi%
  \global\let\svgwidth\undefined%
  \global\let\svgscale\undefined%
  \makeatother%
  \begin{picture}(1,1.00576487)%
    \lineheight{1}%
    \setlength\tabcolsep{0pt}%
    \put(0,0){\includegraphics[width=\unitlength,page=1]{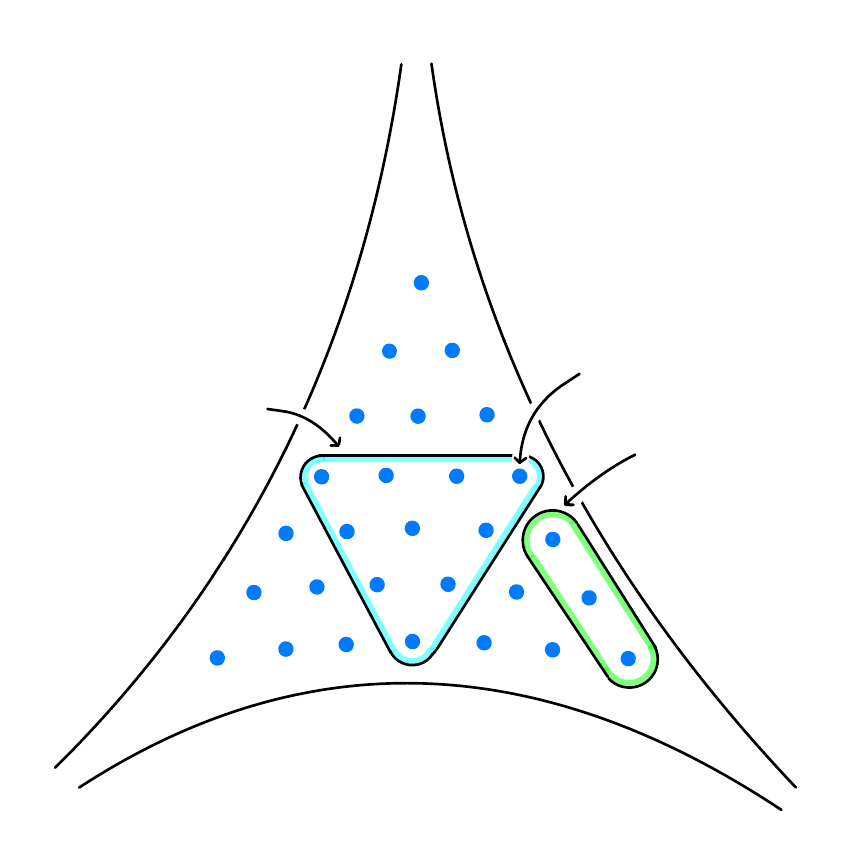}}%
    \put(0.22685609,0.52163042){\color[rgb]{0,0,0}\makebox(0,0)[lt]{\lineheight{1.25}\smash{\begin{tabular}[t]{l}{\small $\mathcal{B}^*$}\end{tabular}}}}%
    \put(0.70212976,0.5581845){\color[rgb]{0,0,0}\makebox(0,0)[lt]{\lineheight{1.25}\smash{\begin{tabular}[t]{l}{\small $\mathbf{j}'$}\end{tabular}}}}%
    \put(0.766492,0.46026999){\color[rgb]{0,0,0}\makebox(0,0)[lt]{\lineheight{1.25}\smash{\begin{tabular}[t]{l}{\small $\mathcal{B}''$}\end{tabular}}}}%
    \put(0.01894726,0.05520355){\color[rgb]{0,0,0}\makebox(0,0)[lt]{\lineheight{1.25}\smash{\begin{tabular}[t]{l}{\small $x_2$}\end{tabular}}}}%
    \put(0.46561914,0.96639956){\color[rgb]{0,0,0}\makebox(0,0)[lt]{\lineheight{1.25}\smash{\begin{tabular}[t]{l}{\small $x_3$}\end{tabular}}}}%
    \put(0.9427928,0.03524864){\color[rgb]{0,0,0}\makebox(0,0)[lt]{\lineheight{1.25}\smash{\begin{tabular}[t]{l}{\small $x_1$}\end{tabular}}}}%
  \end{picture}%
\endgroup%

	\caption{\small Dots are elements of $\mathcal B$. $\mathcal B^*$ is the cyan triangle, and $\mathcal B''$ is the green rectangle. Left (resp. right) figure represents the even (resp. odd) case with $d = 8$ (resp. $d = 9$.) On the left $\mathcal B^0$, and $\mathcal B^0_\pm$ are in purple and $\mathcal B'$ is in yellow. For $d$ even, $\mathcal B''=\mathcal B'\cup \{{\bf j}^0\}$.}
	\label{fig:boundary-rectangles}
\end{figure}

The main technical ingredient for the proof of Theorem~\ref{thm: topology} is the following statement:

\begin{lemma}\label{lem: change equations}
	The group $Y$ from Theorem \ref{thm: hom} can be described as 
	\[
	Y = \{({\bf v}, {\bf z}) \in (G^{\Ac})^{\mathcal{O} \sqcup \mathcal{U}} \times (G^{\Bc})^{\mathcal{S}} \mid \text{$({\bf v}, {\bf z})$ verifies }(\star) \} ,
	\]
	where
	\[(\star)=\left\{\begin{array}{ll}
		\blacklozenge(t, {\bf j}), &{\bf j} \in \Bc\text{ and }t \in \mathcal S,\\
		\clubsuit({\bf i}), &{\bf i}\in\Ac'',\\
		\spadesuit({\bf i}), &{\bf i}\in\Ac'-\{(1,d-1)\},\\
		d\,{\rm tor}'_d({\bf v},{\bf z})=0.&
	\end{array}
	\right.\]
\end{lemma}
	
The system of equations $(\star)$ in Lemma \ref{lem: change equations} should be considered as the analog of the row-echelon form for the system of linear equations that describes the group $Y$. Assuming temporarily this technical statement, we can now prove Theorem \ref{thm: topology}.
	
\begin{proof}[Proof of Theorem \ref{thm: topology}]
	We start by describing our candidate minimal set of components in $(G^{\Ac})^{\mathcal{O} \sqcup \mathcal{U}} \times (G^{\Bc})^{\mathcal{S}}$ that uniquely characterizes every element $({\bf v}, {\bf z})$ in $Y$. Let ${\bf j}'\in\Bc$ be the triple given by
		\[{\bf j}':=\left\{\begin{array}{ll}
			(\frac{d-1}{2},1,\frac{d-1}{2})&\text{if }d\text{ is odd},\\
			(\frac{d-2}{2},2,\frac{d-2}{2})&\text{if }d\text{ is even},
		\end{array}\right.\]
		and notice that ${\bf j}'\in\Bc-\Bc''$, see Figure \ref{fig:boundary-rectangles}. We fix once and for all a plaque $\overline{T}\in\Delta$ of $\lambda$ and a right unorientable rectangle $\overline{R}\in\mathcal U^r$. (Recall from Section \ref{ssub:max tree} that, given any maximal tree $M$ inside $N$, there exists at least one unorientable rectangle $\overline{R}$ for $M$. Then by reversing the orientation on the ties of $M$ if necessary, we can assume that $\overline{R}$ is a right unorientable rectangle for $M$.)
		We claim that the map
		\[I_2:Y \longrightarrow (G^{\Ac})^{\mathcal{O} \sqcup \mathcal{U}-\{\overline{R}\}}\times G^{\Ac-\Ac'}\times(G^{\Bc})^{\Delta-\{\overline{T}\}}\times G^{\Bc-(\Bc''\cup\{{\bf j}'\})} \times  G_d , \]
		that sends every $({\bf v},{\bf z})\in Y$ to
		\[
		\Big((v_R)_{R\in\mathcal O\sqcup\mathcal U-\{\overline{R}\}},(v_{\overline{R}}^{\bf i})_{{\bf i}\in\Ac-\Ac'},(z_{t(T)})_{T\in\Delta-\{\overline{T}\}},(z_{t(\overline{T})}^{\bf j})_{{\bf j}\in\Bc-(\Bc''\cup\{{\bf j}'\})},{\rm tor}'_d({\bf v},{\bf z})\Big) ,
		\]
		is the desired group isomorphism. 
		
		First, recall from Section \ref{background} that the number of plaques of any maximal lamination of $S$ is equal to $|\Delta|=4g-4$. Moreover, since $N$ is a train track neighborhood for $\lambda$ and $M$ is a maximal tree in $N$, the number of rectangles of $N$ that are not contained in $M$ is equal to $|\mathcal O\sqcup\mathcal U|=6g-5$. Furthermore, the subsets of indices $\Ac$, $\Ac'$, $\Bc$, and $\Bc''$ verify:
		\[
		|\Ac|=d-1, \quad |\Ac'|=\left\lfloor\frac{d-1}{2}\right\rfloor, \quad |\Bc|=\frac{(d-1)(d-2)}{2} , \quad |\Bc''|=\left\lceil\frac{d-3}{2}\right\rceil .
		\]
			In particular,
			\begin{align*}
				&|\Ac|(|\mathcal O\sqcup\mathcal U|-1)+|\Ac-\Ac'|+|\Bc|(|\Delta|-1)+|\Bc-(\Bc''\cup\{{\bf j}'\})|\\
				&=|\Ac||\mathcal O\sqcup\mathcal U|+|\Bc||\Delta|-(|\Ac'|+|\Bc''|)-1\\
				&=(d^2-1)(2g-2).
			\end{align*}
			It follows that the codomain of the map $I_2$ is isomorphic to the group $G^{(d^2-1)(2g-2)} \times G_d$, which is the target group in the statement of Theorem \ref{thm: topology}. In what follows, we will refer to the codomain of $I_2$ by $G^{(d^2-1)(2g-2)} \times G_d$, with abuse of notation.
			
			It is evident from the definition of ${\rm tor}'_d$ that $I_2$ is a group homomorphism, so it suffices to show that $I_2$ is a bijection. To do so, we need to show that it admits a (left and right) inverse. Notice that the projection of $I_2({\bf v}, {\bf z})$ onto $G^{(d^2-1)(2g-2)}$ does not depend on the components of $({\bf v}, {\bf z})$ given by:
			\begin{enumerate}
				\item  $z_{t(\overline{T})}^{\bf j}$, $z_{t(\overline{T})_+}^{{\bf j}_+}$, and $z_{t(\overline{T})_-}^{{\bf j}_-}$ for ${\bf j}={\bf j}^0$ if $d$ is even,
				\item  $z_{t(\overline{T})}^{\bf j}$, $z_{t(\overline{T})_+}^{{\bf j}_+}$, and $z_{t(\overline{T})_-}^{{\bf j}_-}$ for ${\bf j}={\bf j}'$,
				\item  $z_{t(\overline{T})}^{\bf j}$, $z_{t(\overline{T})_+}^{{\bf j}_+}$, and $z_{t(\overline{T})_-}^{{\bf j}_-}$ for ${\bf j}\in\Bc'$, and
				\item $v_{\overline{R}}^{\bf i}$ for ${\bf i}\in\Ac'$. 
			\end{enumerate}
			(Recall that $\Bc'' = \Bc'$ if $d$ is odd, and $\Bc'' = \Bc' \cup \{{\bf j}^0\}$ when $d$ is even.) 
			
			Set
			\[I_2^{-1}: G^{(d^2-1)(2g-2)} \times  G_d\longrightarrow Y \]
			to be the map that sends 
			\[\left((u_R)_{R\in\mathcal{O} \sqcup\mathcal{U}-\{\overline{R}\}},(u^{\bf i})_{{\bf i}\in\Ac-\Ac'},(w_T)_{T\in\Delta-\overline{T}},(w^{\bf j})_{{\bf j}\in\Bc-(\Bc''\cup\{{\bf j}'\})},\epsilon\right) \in G^{(d^2-1)(2g-2)} \times  G_d \]
			to the element $({\bf v},{\bf z})\in Y$ defined by the following steps:
			
			\medskip
			
			\noindent {\bf Step 0:} Define $v_R$ for all $R\in\mathcal{O} \sqcup\mathcal{U}-\{\overline{R}\}$, $v_{\overline{R}}^{\bf i}$ for all ${\bf i}\in\Ac-\Ac'$, $z_{t(T)}$ for all $T\in\Delta-\overline{T}$, and $z_{t(\overline{T})}^{\bf j}$ for all ${\bf j}\in\Bc-(\Bc''\cup\{{\bf j}'\})$ by 
			\begin{align*}
				v_R:=u_R&\quad\text{for all }R\in\mathcal{O} \sqcup\mathcal{U}-\{\overline{R}\},\\
				v_{\overline{R}}^{\bf i}:=u^{\bf i}&\quad\text{for all }{\bf i}\in\Ac-\Ac',\\
				z^{\bf j}_{t(T)}=z^{{\bf j}_+}_{t(T)_+}=z^{{\bf j}_-}_{t(T)_-}:=w_T^{\bf j}&\quad\text{for all }T\in\Delta-\{\overline{T}\}\text{ and for all }{\bf j}\in\Bc,\\
				z_{t(\overline{T})}^{\bf j}=z_{t(\overline{T})_+}^{{\bf j}_+}=z_{t(\overline{T})_-}^{{\bf j}_-}:=w^{\bf j}&\quad\text{for all }{\bf j}\in\Bc-(\Bc''\cup\{{\bf j}'\}).
			\end{align*}
			It remains to specify the components of $({\bf v}, {\bf z})$ listed in items (1)--(4) above. We will do so in Steps 1--4 below, following the order.
			
			\medskip
			
			\noindent {\bf Step 1:} If $d$ is odd, skip this step. On the other hand, if $d$ is even, notice that ${\bf i}^0\in\Ac-\Ac'$ and
			\[(\Bc^0\cup\Bc^0_+\cup\Bc^0_-)\cap (\Bc''\cup\{{\bf j}'\})=\left\{{\bf j}^0\right\},\]
			where $\Bc^0_{\pm}:=\{{\bf j}_\pm\in\Bc:{\bf j}\in\Bc^0\}$. Also, notice that equation $\clubsuit({\bf i}^0)$ is
			\[2\sum_{R \in \mathcal{U}^r} v_R^{\mathbf{i}_0} - 2\sum_{R \in \mathcal{U}^\ell} v_R^{\mathbf{i}_0}= \sum_{t\in \mathcal{S}^\ell} \sum_{{\bf j} \in \Bc^0} z^{\bf j}_t - \sum_{t\in \mathcal{S}^r} \sum_{{\bf j} \in \Bc^0} z^{\bf j}_t.\]
			Recall that ${\bf j}^0_-\in \mathcal B^0$. Thus, in Step 0, we have already specified all but one of the variables that appear in the equation $\clubsuit({\bf i}^0)$, namely $z_{t(\overline{T})_-}^{{\bf j}^0_-}$. We specify $z_{t(\overline{T})_-}^{{\bf j}^0_-}$ as follows:
			\begin{itemize}
				\item when $t(\overline{T})_-$ is a right vertical boundary component, let
				\begin{align*}
					z_{t(\overline{T})_-}^{{\bf j}^0_-} := 2\sum_{R\in\mathcal U^\ell}&v_R^{{\bf i}^0}-2\sum_{R\in\mathcal U^r}v_R^{{\bf i}^0}+\sum_{t\in\mathcal S^\ell}\sum_{{\bf j}\in\Bc^0}z_t^{\bf j}-\sum_{t\in\mathcal S^r-\{t(\overline{T})_-\}}\sum_{{\bf j}\in\Bc^0}z_t^{\bf j}-\sum_{{\bf j}\in\Bc^0-\{{\bf j}^0_-\}}z_{t(\overline{T})_-}^{\bf j};
				\end{align*}
				\item when $t(\overline{T})_-$ is a left vertical boundary component, let 
				\begin{align*}
					z_{t(\overline{T})_-}^{{\bf j}^0_-} := -2\sum_{R\in\mathcal U^\ell}&v_R^{{\bf i}^0}+2\sum_{R\in\mathcal U^r}v_R^{{\bf i}^0}-\sum_{t\in\mathcal S^\ell-\{t(\overline{T})_-\}}\sum_{{\bf j}\in\Bc^0}z_t^{\bf j}-\sum_{{\bf j}\in\Bc^0-\{{\bf j}^0_-\}}z_{t(\overline{T})_-}^{\bf j}+\sum_{t\in\mathcal S^r}\sum_{{\bf j}\in\Bc^0}z_t^{\bf j}.
				\end{align*}
			\end{itemize}
			Then note that equation $\clubsuit({\bf i}_0)$ holds. Also, set
			\begin{align*}
				z_{t(\overline{T})}^{{\bf j}^0}=z_{t(\overline{T})_+}^{{\bf j}^0_+}=z_{t(\overline{T})_-}^{{\bf j}^0_-}.
			\end{align*}
			
			\medskip
			
			\noindent {\bf Step 2:} In Steps 0 and 1, we have specified $z_{t(\overline{T})}^{\bf j}$, $z_{t(\overline{T})_+}^{{\bf j}_+}$, and $z_{t(\overline{T})_-}^{{\bf j}_-}$ for all ${\bf j}\in\Bc-(\Bc'\cup\{{\bf j}'\})$. Since
			\[{\bf i}^0\in\Ac-\Ac',\quad \Bc^*\cap\left(\Bc'\cup\{{\bf j}'\}\right)=\{{\bf j}'\},\] 
			and $\Bc^0\subset\Bc-(\Bc'\cup\{{\bf j}'\})$ when $d$ is even, we have already specified all but one of the variables that appear in the expression for ${\rm tor}_d'({\bf v},{\bf z})$, namely $z^{{\bf j}'}_{t(\overline{T})}$. Then:
			\begin{itemize}
				\item When $d$ is even, set
				\begin{align*}
					z^{{\bf j}'}_{t(\overline{T})} := - \epsilon - \sum_{T\in\Delta - \{\overline{T}\}} \sum_{{\bf j} \in \Bc^*} z^{\bf j}_{t(T)} - \sum_{{\bf j} \in \Bc^* - \{{\bf j}'\}} z^{\bf j}_{t(\overline{T})} + \sum_{R \in \mathcal{U}^r} v_{R}^{{\bf i}^0} - \sum_{R \in \mathcal{U}^\ell} v_{R}^{{\bf i}^0} - \sum_{t\in\mathcal S^\ell}\sum_{{\bf j}\in\Bc^0}z_t^{\bf j}.
				\end{align*}
				\item When $d$ is odd, set
				\begin{align*}
					z^{{\bf j}'}_{t(\overline{T})} := -\epsilon-\sum_{T\in\Delta-\{\overline{T}\}} \sum_{{\bf j} \in \Bc^*} z^{\bf j}_{t(T)}-\sum_{{\bf j} \in \Bc^*-\{{\bf j}'\}} z^{\bf j}_{t(\overline{T})}.
				\end{align*}
			\end{itemize}
			Then $\epsilon={\rm tor}_d'({\bf v},{\bf z})$, and in particular, $d\,{\rm tor}_d'({\bf v},{\bf z})=0$.
			Also, set
			\[z^{{\bf j}'}_{t(\overline{T})}=z^{{\bf j}'_+}_{t(\overline{T})_+}=z^{{\bf j}'_-}_{t(\overline{T})_-}.\]
			
			\medskip
			
			\noindent {\bf Step 3:} Notice that in Steps 0, 1, and 2, we have specified $z_{t(\overline{T})}^{\bf j}$, $z_{t(\overline{T})_+}^{{\bf j}_+}$, and $z_{t(\overline{T})_-}^{{\bf j}_-}$ for all ${\bf j}\in\Bc-\Bc'$. Now, for all ${\bf i}\in\Ac'-\{(1,d-1)\}$, rewrite the equation $\spadesuit({\bf i})$ so that the terms on the right hand side are already specified and the terms on the left hand side are not yet specified. Let $\Bc_{\pm}':=\{{\bf j}_\pm\in\Bc:{\bf j}\in\Bc'\}$. Then equation $\spadesuit(\lfloor\frac{d-1}{2}\rfloor,\lceil\frac{d+1}{2}\rceil)$ is rewritten as
			\begin{align*}z_{t(\overline{T})_-}^{(\lfloor\frac{d-3}{2}\rfloor,\lceil\frac{d+1}{2}\rceil,1)}=& \sum_{t \in \mathcal S} \sum_{{\bf j} \in \Bc : j_2 = \lfloor\frac{d-1}{2}\rfloor} z_t^{\bf j}\\
				&-\sum_{t \in \mathcal S-\{t(\overline{T})_-\}} \sum_{{\bf j} \in \Bc : j_2 =\lceil\frac{d+1}{2}\rceil} z_t^{\bf j}-\sum_{{\bf j} \in \Bc-\Bc'_- : j_2 = \lceil\frac{d+1}{2}\rceil} z_{t(\overline{T})_-}^{\bf j},
			\end{align*}
			and for all ${\bf i}\in\Ac'-\{(\lfloor\frac{d-1}{2}\rfloor,\lceil\frac{d+1}{2}\rceil),(1,d-1)\}$, equation $\spadesuit({\bf i})$ is rewritten as
			\begin{align*}z_{t(\overline{T})_-}^{(i_1-1,d-i_1,1)}-z_{t(\overline{T})_+}^{(1,i_1,d-1-i_1)}=& \sum_{t \in \mathcal S-\{t(\overline{T})_{+}\}} \sum_{{\bf j} \in \Bc : j_2 = i_1} z_t^{\bf j}+\sum_{{\bf j} \in \Bc-\Bc'_+ : j_2 = i_1} z_{t(\overline{T})_+}^{\bf j}\\
				&-\sum_{t \in \mathcal S-\{t(\overline{T})_-\}} \sum_{{\bf j} \in \Bc : j_2 = i_2} z_t^{\bf j}-\sum_{{\bf j} \in \Bc-\Bc'_- : j_2 = i_2} z_{t(\overline{T})_-}^{\bf j}.
			\end{align*}
			For all ${\bf i}\in\Ac'-\{(1,d-1)\}$, let $M_{\bf i}$ denote the expression on the right hand side of the rewritten version of equation $\spadesuit({\bf i})$.
			Set
			\[z_{t(\overline{T})_-}^{(\lfloor\frac{d-3}{2}\rfloor,\lceil\frac{d+1}{2}\rceil,1)}:=M_{\lfloor\frac{d-1}{2}\rfloor,\lceil\frac{d+1}{2}\rceil}\]
			and set
			\[z_{t(\overline{T})}^{(\lceil\frac{d+1}{2}\rceil,1,\lfloor\frac{d-3}{2}\rfloor)}=z_{t(\overline{T})_+}^{(1,\lfloor\frac{d-3}{2}\rfloor,\lceil\frac{d+1}{2}\rceil)}=z_{t(\overline{T})_-}^{(\lfloor\frac{d-3}{2}\rfloor,\lceil\frac{d+1}{2}\rceil,1)}.\]
			Then for all ${\bf i} =(i_1,i_2)\in\Ac'-\{(\lfloor\frac{d-1}{2}\rfloor,\lceil\frac{d+1}{2}\rceil),(1,d-1)\}$, iteratively define (in decreasing order of $i_1$)
			\[z_{t(\overline{T})_-}^{(i_1-1,d-i_1,1)}:=M_{\bf i}+z_{t(\overline{T})_+}^{(1,i_1,d-1-i_1)}\]
			and set
			\[z_{t(\overline{T})}^{(d-i_1,1,i_1-1)}=z_{t(\overline{T})_+}^{(1,i_1-1,d-i_1)}=z_{t(\overline{T})_-}^{(i_1-1,d-i_1,1)}.\]
			By definition, equation $\spadesuit({\bf i})$ holds for all ${\bf i}\in\Ac'-\{(1,d-1)\}$.
			
			\medskip
			
			\noindent {\bf Step 4:} Finally, notice that in Steps 0, 1, 2, and 3, we have specified, for each ${\bf i}\in\Ac'$, all but one of the variables that appear in the equation $\clubsuit({\bf i})$, namely $v_{\overline{R}}^{\bf i}$. Thus, for each ${\bf i}\in\Ac'$, set 
			\[v_{\overline{R}}^{\bf i}:=\sum_{t\in \mathcal{S}^\ell} \sum_{{\bf j} \in \Bc : j_2 = i_2} z^{\bf j}_t - \sum_{t\in \mathcal{S}^r} \sum_{{\bf j} \in \Bc : j_2 = i_1} z^{\bf j}_t+\sum_{R \in \mathcal{U}^\ell} (v_R^{\mathbf{i}} + v_R^{\widehat{\mathbf{i}}} )-\sum_{R \in \mathcal{U}^r-\{\overline{R}\}} ( v_R^{\mathbf{i}} + v_R^{\widehat{\mathbf{i}}})-v_{\overline{R}}^{\widehat{\bf i}},\]
			and notice that equation $\clubsuit({\bf i})$ holds.

			With this, we have completely specified $({\bf v},{\bf z})$.
			
			By definition, $({\bf v},{\bf z})$ satisfies the system of equations $(\star)$, and so Lemma \ref{lem: change equations} implies that $I_2^{-1}$ is well-defined. It is also straightforward to deduce from the definitions that $I_2^{-1}$ is the inverse of $I_2$. Indeed, for every $p\in G^{(d^2-1)(2g-2)} \times  G_d$, $I_2^{-1}(p)$ was defined by solving for the unique $({\bf v},{\bf z})\in Y$ such that $I_2({\bf v},{\bf z})=p$. It follows that $I_2$ is a bijection.
		\end{proof}
		
		\subsection{Another description of $Y$}
		The remainder of this section is the proof of Lemma \ref{lem: change equations}. We say that two systems of equations on $(G^{\Ac})^{\mathcal{O} \sqcup \mathcal{U}} \times (G^{\Bc})^{\mathcal{S}}$ are \emph{equivalent} if they have the same solution locus.

		First, we observe that we can replace ``half" of the equations of the form $\clubsuit({\bf i})$ with equations of the form $\spadesuit({\bf i})$.
		
		\begin{lemma}\label{rmk:equivalent description of Y}
			On $(G^{\Ac})^{\mathcal{O} \sqcup \mathcal{U}} \times (G^{\Bc})^{\mathcal{S}}$, the system of equations 
			\[\left\{\begin{array}{ll}
				\blacklozenge(t, {\bf j}), &{\bf j} \in \Bc\text{ and }t \in \mathcal S,\\
				\clubsuit({\bf i}),&{\bf i} \in \Ac
			\end{array}\right.\]
			is equivalent to the system of equations
			\[\left\{\begin{array}{ll}
				\blacklozenge(t, {\bf j}), &{\bf j} \in \Bc\text{ and }t \in \mathcal S,\\
				\clubsuit({\bf i}),&{\bf i} \in \Ac'',\\
				\spadesuit({\bf i}),&{\bf i} \in \Ac'.
			\end{array}
			\right.\]
		\end{lemma}
		\begin{proof}
			Notice that for all ${\bf i}\in\Ac$ the left hand side of equations $\clubsuit({\bf i})$ and $\clubsuit(\widehat{\bf i})$ are equal, and setting their right hand sides to be equal gives the equation $\spadesuit({\bf i})$. Thus, equations $\clubsuit({\bf i})$ and $\clubsuit(\widehat{\bf i})$ both hold if and only if equations $\clubsuit({\bf i})$ and $\spadesuit({\bf i})$ both hold. The lemma follows.
		\end{proof}
		
		To finish the proof, we need to show that we can replace the equation $\spadesuit(1,d-1)$ with the equation $d\,{\rm tor}'_d({\bf v},{\bf z})=0$. The main computation that allows us to do so is the following lemma.
		
		\begin{lemma}\label{lem:a nice combination of triangle invariants}
			Let ${\bf z} = (z_t^{\bf j}) \in (G^{\Bc})^{\mathcal{S}}$ be an element satisfying equation $\blacklozenge(t, {\bf j})$ for every ${\bf j} \in \Bc$ and $t \in \mathcal S$. Then for every $t \in \mathcal S$, we have that
			\begin{align*}
				\sum_{{\bf i} \in \Ac'} i_1 \left( \sum_{{\bf j} \in \Bc : j_2 = i_1} (z^{\bf j}_t + z^{\bf j}_{t^+} + z^{\bf j}_{t^-}) - \sum_{{\bf j} \in \Bc : j_2 = i_2} (z^{\bf j}_t + z^{\bf j}_{t^+} + z^{\bf j}_{t^-}) \right) \\
				= \begin{cases}
					\displaystyle d \sum_{{\bf j} \in \Bc^*} z_t^{\bf j} &\quad\text{if } $d$ \text{ odd},\\
					\displaystyle d \sum_{{\bf j} \in \Bc^*} z_t^{\bf j} + \frac{d}{2}\sum_{{\bf j} \in \Bc^0} (z_t^{\bf j} + z_{t^+}^{\bf j} + z_{t^-}^{\bf j}) &\quad\text{if } $d$ \text{ even}.
				\end{cases}
			\end{align*}
		\end{lemma}
		
		\begin{proof}
			From the relations $\blacklozenge(t, {\bf j})$ it follows that $z^{\bf j}_{t^+} = z^{{\bf j}^-}_t$ and $z^{\bf j}_{t^-} = z^{{\bf j}^+}_t$ for every ${\bf j} \in \Bc$ and $t \in \mathcal S$. Thus, using the facts that $j_1+j_2+j_3=d=i_1+i_2$ and $\lfloor\frac{d-1}{2}\rfloor+\lceil\frac{d-1}{2}\rceil=d$, we obtain
			\begin{align*}
				\sum_{{\bf i}\in\Ac'}\left(i_1\sum_{{\bf j}\in\Bc:j_2=i_1}z_{t^+}^{\bf j}\right)&=\sum_{{\bf i}\in\Ac'}\left(i_1\sum_{{\bf j}\in\Bc:j_3=i_1}z_t^{\bf j}\right)=\sum_{{\bf j}\in\Bc:j_3\le\lfloor\frac{d-1}{2}\rfloor}j_3z_t^{{\bf j}},\\
				\sum_{{\bf i}\in\Ac'}\left(i_1\sum_{{\bf j}\in\Bc:j_2=i_1}z_{t^-}^{\bf j}\right)&=\sum_{{\bf i}\in\Ac'}\left(i_1\sum_{{\bf j}\in\Bc:j_1=i_1}z_t^{\bf j}\right)=\sum_{{\bf j}\in\Bc:j_1\le\lfloor\frac{d-1}{2}\rfloor}j_1z_t^{{\bf j}},\\
				\sum_{{\bf i}\in\Ac'}\left(i_1\sum_{{\bf j}\in\Bc:j_2=i_2}z_{t^+}^{\bf j}\right)&=\sum_{{\bf i}\in\Ac'}\left(i_1\sum_{{\bf j}\in\Bc:j_3=i_2}z_t^{\bf j}\right)=\sum_{{\bf j}\in\Bc:j_3\ge\lceil\frac{d+1}{2}\rceil}(j_1+j_2)z_t^{{\bf j}},\quad\text{and}\\
				\sum_{{\bf i}\in\Ac'}\left(i_1\sum_{{\bf j}\in\Bc:j_2=i_2}z_{t^-}^{\bf j}\right)&=\sum_{{\bf i}\in\Ac'}\left(i_1\sum_{{\bf j}\in\Bc:j_1=i_2}z_t^{\bf j}\right)=\sum_{{\bf j}\in\Bc:j_1\ge\lceil\frac{d+1}{2}\rceil}(j_2+j_3)z_t^{{\bf j}}
			\end{align*}
			for all $t\in\mathcal S$. At the same time, notice that
			\begin{align*}
				\sum_{{\bf i}\in\Ac'}\left(i_1\sum_{{\bf j}\in\Bc:j_2=i_1}z_t^{\bf j}\right)&=\sum_{{\bf j}\in\Bc:j_2\le\lfloor\frac{d-1}{2}\rfloor}j_2z_t^{{\bf j}},\quad\text{and}\\
				\sum_{{\bf i}\in\Ac'}\left(i_1\sum_{{\bf j}\in\Bc:j_2=i_2}z_t^{\bf j}\right)&=\sum_{{\bf j}\in\Bc:j_2\ge\lceil\frac{d+1}{2}\rceil}(j_1+j_3)z_t^{{\bf j}}
			\end{align*}
			for all $t\in \mathcal S$, so we may deduce that
			\begin{align*}
				&\sum_{{\bf i} \in \Ac'} i_1 \left( \sum_{{\bf j} \in \Bc : j_2 = i_1} (z^{\bf j}_t + z^{\bf j}_{t^+} + z^{\bf j}_{t^-}) - \sum_{{\bf j} \in \Bc : j_2 = i_2} (z^{\bf j}_t + z^{\bf j}_{t^+} + z^{\bf j}_{t^-}) \right) \\
				=&\sum_{{\bf j} \in \Bc : j_2 \le\lfloor\frac{d-1}{2}\rfloor} j_2z^{\bf j}_t + \sum_{{\bf j}\in\Bc:j_3\le\lfloor\frac{d-1}{2}\rfloor}j_3z_t^{{\bf j}} + \sum_{{\bf j}\in\Bc:j_1\le\lfloor\frac{d-1}{2}\rfloor}j_1z_t^{{\bf j}} \\
				&\hspace{0.5cm}- \sum_{{\bf j} \in \Bc : j_2 \ge\lceil\frac{d+1}{2}\rceil} (j_1+j_3)z^{\bf j}_t - \sum_{{\bf j}\in\Bc:j_3\ge\lceil\frac{d+1}{2}\rceil}(j_1+j_2)z_t^{{\bf j}} -\sum_{{\bf j}\in\Bc:j_1\ge\lceil\frac{d+1}{2}\rceil}(j_2+j_3)z_t^{{\bf j}}.
			\end{align*}	
			
			We will now compute, for every ${\bf j}\in\Bc$, the coefficient for $z_t^{\bf j}$ in the expression on the right hand side. We will consider the following three cases separately:
			\begin{enumerate}
				\item[(I)] ${\bf j}\in\Bc$ satisfies $j_k\leq \lfloor\frac{d-1}{2}\rfloor$ for all $k\in\{1,2,3\}$,
				\item[(II)] ${\bf j}\in\Bc$ satisfies $j_k\ge\lceil\frac{d+1}{2}\rceil$ for some $k\in\{1,2,3\}$,
				\item[(III)] ${\bf j}\in\Bc$ satisfies $j_k=\frac{d}{2}$ for some $k\in\{1,2,3\}$ (this only happens when $d$ is even).
			\end{enumerate}
			
			If (I) holds, then the term $z_t^{\bf j}$ appears in each of the first three summands with coefficients $j_2$, $j_3$, and $j_1$ respectively, but not in the last three summands. Thus, the net coefficient of $z_t^{\bf j}$ is 
			\[j_1 + j_2 + j_3 = d.\]
			
			If (II) holds, then 
			\[j_{k-1}+j_{k+1}=d-j_k\le d-\left\lceil\frac{d+1}{2}\right\rceil=\left\lfloor\frac{d-1}{2}\right\rfloor\]
			(arithmetic in the subscripts are done modulo $3$), so $j_{k-1}$ and $j_{k+1}$ are both at most $\lfloor\frac{d-1}{2}\rfloor$. This implies that the term $z_t^{\bf j}$ appears in two of the first three summands with coefficients $j_{k-1}$ and $j_{k+1}$, and one of the last three summands with coefficients $-(j_{k-1}+j_{k+1})$. Thus, the net coefficient of $z_t^{\bf j}$ is 
			\[j_{k-1} + j_{k+1} -(j_{k-1}+j_{k+1}) = 0.\]
			
			Finally, if (III) holds, then $d$ is even and
			\[j_{k-1}+j_{k+1}=d-j_k=\frac{d}{2},\]
			so $j_{k-1}$ and $j_{k+1}$ are both at most $\frac{d}{2}-1=\lfloor\frac{d-1}{2}\rfloor$. At the same time, $j_k=\frac{d}{2}$ is strictly between $\lfloor\frac{d-1}{2}\rfloor$ and $\lceil\frac{d+1}{2}\rceil$. This implies that the term $z_t^{\bf j}$ appears in two of the first three summands with coefficients $j_{k-1}$ and $j_{k+1}$, but not in the last three summands. Thus, the net coefficient of $z_t^{\bf j}$ is 
			\[j_{k-1} + j_{k+1} = d-j_k= \frac{d}{2}.\]
			This concludes the proof of the lemma.
		\end{proof}
		
		Using Lemma \ref{lem:a nice combination of triangle invariants}, we prove Lemma \ref{lem: change equations}.
		
		\begin{proof}[Proof of Lemma \ref{lem: change equations}]
			By Lemma \ref{rmk:equivalent description of Y}, it suffices to prove that on $(G^{\Ac})^{\mathcal{O} \sqcup \mathcal{U}} \times (G^{\Bc})^{\mathcal{S}}$, the system of equations
			\[\left\{\begin{array}{ll}
				\blacklozenge(t,{\bf j}),&{\bf j} \in \Bc, t\in\mathcal S,\\
				\clubsuit({\bf i}),&{\bf i} \in \Ac'',\\
				\spadesuit({\bf i}),&{\bf i} \in \Ac'.
			\end{array}
			\right.\]
			is equivalent to 
			\[\left\{\begin{array}{ll}
				\blacklozenge(t,{\bf j}),&{\bf j} \in \Bc, t\in\mathcal S,\\
				\clubsuit({\bf i}),&{\bf i} \in \Ac'',\\
				\spadesuit({\bf i}),&{\bf i} \in \Ac'-\{(1,d-1)\},\\
				d\,{\rm tor}'_d({\bf v},{\bf z})=0.
			\end{array}
			\right.\]

			First note that for every ${\bf i}\in\Ac$, the equation $\spadesuit({\bf i})$ can be written as
			\[\sum_{T\in\Delta}\sum_{{\bf j} \in \Bc : j_2 = i_1} (z_{t(T)}^{\bf j}+z_{t(T)_+}^{\bf j}+z_{t(T)_-}^{\bf j} )-\sum_{T\in\Delta}\sum_{{\bf j} \in \Bc : j_2 = i_2} (z_{t(T)}^{\bf j}+z_{t(T)_+}^{\bf j}+z_{t(T)_-}^{\bf j} )=0.\]
			Thus, by Lemma \ref{lem:a nice combination of triangle invariants}, the equation $\sum_{{\bf i}\in\Ac'}i_1\spadesuit({\bf i})$ can be written as follows, depending on the parity of $d$:
			\begin{itemize}
				\item[(i)] If $d$ is odd, then equation $\sum_{{\bf i}\in\Ac'}i_1\spadesuit({\bf i})$ is of the form
				\[d\sum_{T\in\Delta} \sum_{{\bf j} \in \Bc^*} z^{\bf j}_{t(T)}=0.\]	
				\item[(ii)] If $d$ is even, then equation $\sum_{{\bf i}\in\Ac'}i_1\spadesuit({\bf i})$ is of the form
				\[d \sum_{T\in\Delta}\sum_{{\bf j} \in \Bc^*} z_{t(T)}^{\bf j} + \frac{d}{2}\sum_{T\in\Delta}\sum_{{\bf j} \in \Bc^0} (z_{t(T)}^{\bf j} + z_{t(T)^+}^{\bf j} + z_{t(T)^-}^{\bf j})=0,\]
				or equivalently, 
				\[d \sum_{T\in\Delta}\sum_{{\bf j} \in \Bc^*} z_{t(T)}^{\bf j} + \frac{d}{2}\sum_{t\in\mathcal S}\sum_{{\bf j} \in \Bc^0} z_t^{\bf j}=0.\]
			\end{itemize}
			
			From this, notice that when $d$ is odd, the equation $\sum_{{\bf i}\in\Ac'}i_1\spadesuit({\bf i})$ is exactly $-d\,{\rm tor}'_d({\bf v},{\bf z})=0$. On the other hand, if $d$ is even, the equation $\clubsuit({\bf i}_0)$ is
			\[\sum_{t\in \mathcal{S}^\ell} \sum_{{\bf j} \in \Bc^0} z^{\bf j}_t - \sum_{t\in \mathcal{S}^r} \sum_{{\bf j} \in \Bc^0} z^{\bf j}_t-2\sum_{R \in \mathcal{U}^r} v_R^{\mathbf{i}^0} + 2\sum_{R \in \mathcal{U}^\ell} v_R^{\mathbf{i}^0}=0 ,\]
			so the equation $\sum_{{\bf i}\in\Ac'}i_1\spadesuit({\bf i})+\frac{d}{2}\clubsuit({\bf i}_0)$ is
			\[d \sum_{T\in\Delta}\sum_{{\bf j} \in \Bc^*} z_{t(T)}^{\bf j}+d\sum_{t\in\mathcal S^\ell}\sum_{{\bf j} \in \Bc^0} z_t^{\bf j}-d\sum_{R \in \mathcal{U}^r} v_R^{\mathbf{i}^0} + d\sum_{R \in \mathcal{U}^\ell} v_R^{\mathbf{i}^0} =0,\]
			which is exactly $-d\, {\rm tor}'_d({\bf v},{\bf z})=0$.
			
			In either case, the equation $\spadesuit(1,d-1)$ appears exactly once with coefficient $1$ in both $\sum_{{\bf i}\in\Ac'}i_1\spadesuit({\bf i})$ when $d$ is odd and $\sum_{{\bf i}\in\Ac'}i_1\spadesuit({\bf i})+\frac{d}{2}\clubsuit({\bf i}_0)$ when $d$ is even. It follows that the lemma holds.
		\end{proof}
		
\section{Connected Components of the character variety}\label{sec:connected}

In \cite{Li}, Jun Li describes a bijection between the connected components of ${\rm Hom}(\Gamma,\PGL_d(\Cb))$ and the group $\Zb_d$ of $d$--th roots of unity in $\Cb/2\pi i\Zb$. In this section, we describe a map
\[
{\rm ob}_d\colon{\rm Hom}(\Gamma,\PGL_d(\Cb))\to\Zb_d
\] 
constant on connected components which gives an alternative definition of the bijection from \cite{Li} that is well-suited for our purposes and will be used in future sections, see Proposition \ref{ob bijection}.

Recall that $S$ is a closed, connected, oriented hyperbolic surface, and $\Gamma$ is the deck group of its universal cover. Let $\mathcal G$ be an embedded graph in $S$ such that $S\setminus\mathcal G$ is homeomorphic to an open disk, let $\mathcal{G}' \subset \mathcal{G}$ be a maximal tree, and let $q \in \mathcal{G}'$ be a basepoint.

\begin{figure}[h!]
	\centering
	\def\svgwidth{0.9\textwidth}
	%% Creator: Inkscape 1.2.2 (b0a84865, 2022-12-01), www.inkscape.org
%% Accompanies image file 'figure3.pdf' (pdf, eps, ps)
\begingroup%
  \makeatletter%
  \providecommand\color[2][]{%
    \errmessage{(Inkscape) Color is used for the text in Inkscape, but the package 'color.sty' is not loaded}%
    \renewcommand\color[2][]{}%
  }%
  \providecommand\transparent[1]{%
    \errmessage{(Inkscape) Transparency is used (non-zero) for the text in Inkscape, but the package 'transparent.sty' is not loaded}%
    \renewcommand\transparent[1]{}%
  }%
  \providecommand\rotatebox[2]{#2}%
  \newcommand*\fsize{\dimexpr\f@size pt\relax}%
  \newcommand*\lineheight[1]{\fontsize{\fsize}{#1\fsize}\selectfont}%
  \ifx\svgwidth\undefined%
    \setlength{\unitlength}{535.19998169bp}%
    \ifx\svgscale\undefined%
      \relax%
    \else%
      \setlength{\unitlength}{\unitlength * \real{\svgscale}}%
    \fi%
  \else%
    \setlength{\unitlength}{\svgwidth}%
  \fi%
  \global\let\svgwidth\undefined%
  \global\let\svgscale\undefined%
  \makeatother%
  \begin{picture}(1,0.38161437)%
    \lineheight{1}%
    \setlength\tabcolsep{0pt}%
    \put(0,0){\includegraphics[width=\unitlength,page=1]{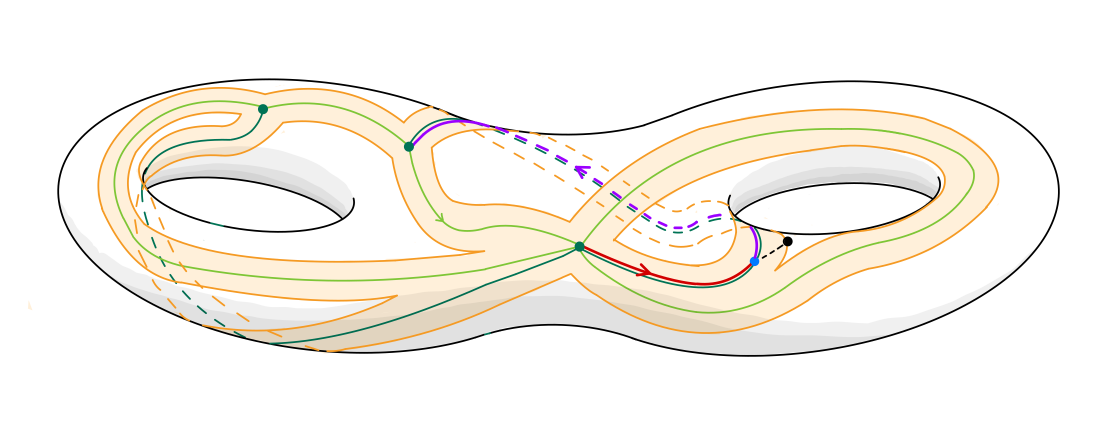}}%
    \put(0.71120239,0.18132928){\color[rgb]{0.08627451,0.49803922,0.39215686}\makebox(0,0)[lt]{\lineheight{1.25}\smash{\begin{tabular}[t]{l}{\small \color[HTML]{000000} $p$}\end{tabular}}}}%
    \put(0.6112539,0.14131174){\color[rgb]{0,0.45098039,0.33333333}\makebox(0,0)[lt]{\lineheight{1.25}\smash{\begin{tabular}[t]{l}{\small \color[HTML]{D20000} $\mathsf{a}_2^+$}\end{tabular}}}}%
    \put(0.40652378,0.18416677){\color[rgb]{0,0.45098039,0.33333333}\makebox(0,0)[lt]{\lineheight{1.25}\smash{\begin{tabular}[t]{l}{\small \color[HTML]{7EC636} $\mathsf{e}_2$}\end{tabular}}}}%
    \put(0.43350563,0.28333654){\color[rgb]{0,0.45098039,0.33333333}\makebox(0,0)[lt]{\lineheight{1.25}\smash{\begin{tabular}[t]{l}{\small \color[HTML]{9B00FF} $\mathsf{a}_2^-$}\end{tabular}}}}%
    \put(0.67243205,0.12732016){\color[rgb]{1,1,1}\makebox(0,0)[lt]{\lineheight{1.25}\smash{\begin{tabular}[t]{l}{\small \color[HTML]{007AFF} $q$}\end{tabular}}}}%
  \end{picture}%
\endgroup%

	\caption{\small The graph $\mathcal G$ (in green), the maximal tree $\mathcal G'$ (in dark green), the basepoint $q\in\mathcal G'$ (in blue), the neighborhood $U$ (in orange) and the basepoint $p\in \partial U$ (in black). Also pictured the paths $\mathsf e_2$, $\mathsf a_2^-$ and $\mathsf a_2^+$.}
	\label{fig:graph}
\end{figure}
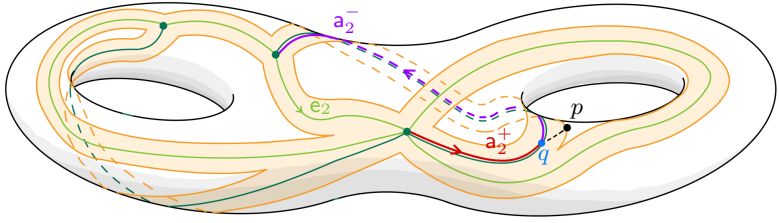

Let $U\subset S$ be a neighborhood of $\mathcal G$ for which there is a strong deformation retract $H:\overline{U}\times[0,1]\to \overline{U}$ of $\overline U$ (the closure  of $U$ in $S$) onto $\mathcal G$, such that the fibers of the map $H(\cdot,1)|_{\partial \overline U}$ is finite-to-one at the vertices of $\mathcal G$, and two-to-one at all the other points of $\partial \overline U$. Let $p\in\partial\overline{U}$ be a point such that $H(p,1)=q$. Since $S\setminus\mathcal G$ is homeomorphic to an open disk, it follows that $S\setminus U$ is homeomorphic to a closed disk whose interior is $S\setminus \overline{U}$. Thus, $\partial\overline{U}=\partial (S\setminus U)$ is a topological circle in $S$, and we can select a counterclockwise parameterization $\csf$ for this circle about $S\setminus U$ that is based at $p$. 

We can then define 
\[\bsf:=H(\cdot,1)\circ\csf,\]
which is a loop in $\mathcal G$ based at $q$, so it defines an element $\delta\in\pi_1(\mathcal G,q)$. Also, $\bsf$ passes through every point in $\mathcal G$ finitely many times, and passes through every point in $\mathcal G$ that is not a vertex exactly twice. We may thus write the loop $\mathsf b$ as a concatenation
\[\mathsf b=\mathsf f_1\cdot\mathsf e_1\cdot \mathsf f_2\cdot \mathsf e_2\cdot \ldots \cdot \mathsf f_k\cdot \mathsf e_k\cdot \mathsf f_{k+1},\]
where for each $i$, $\mathsf f_i$ is a path inside $\mathcal G'$ and $\mathsf e_i$ is a parametrization of an edge in $\mathcal G\setminus\mathcal G'$. (Note that it is possible for $\mathsf f_i$ to be a single vertex of $\mathcal G$.) For future reference, we call this the \emph{graph decomposition of $\mathsf b$ associated to $(\mathcal G,\mathcal G')$}. 

We now want to define for each $\mathsf e_i$ above, an associated element in $\pi_1(\mathcal G,q)$. Let $\mathsf e_i^+$ and $\mathsf e_i^-$ be, respectively, the starting and ending point of the oriented segment $\mathsf e_i$, and let $\mathsf a_{i,+}$ and $\mathsf a_{i,-}$ be the oriented segments in $\mathcal G'$ from $\mathsf e_i^+$ to $q$ and from $q$ to $\mathsf e_i^-$. For all $i$, we define the loop
\[\mathsf b_i:=\mathsf a_{i,-}\cdot\mathsf e_i\cdot\mathsf a_{i,+}\]
based at $q$, and corresponding to an element $\gamma_i\in\pi_1(\mathcal G,q)$. 

Notice that $\mathsf b$ traverses each edge of $\mathcal G$ exactly twice, once in each direction. Thus, for each edge $e$ of $\mathcal G\setminus\mathcal G'$, there are unique indices $i,j\in\{1,\dots,k\}$ such that $\mathsf e_i$ and $\mathsf e_j$ parametrize the same edge $e$ with opposite orientations. In particular, $k=2|\mathcal G\setminus\mathcal G'|$. Furthermore, since $\mathcal G'$ is a maximal tree in $\mathcal G$, it follows that $\mathsf a_{i,-}$ is homotopic to the reverse of $\mathsf a_{j,+}$ and  $\mathsf a_{j,-}$ is homotopic to the reverse of $\mathsf a_{i,+}$. Thus,  $\gamma_i=\gamma_j^{-1}$. 

Also, since $\mathcal G'$ is a maximal tree in $\mathcal G$, the quotient of $\mathcal G$ that identifies $\mathcal G'$ to a point is homeomorphic to a wedge of circles, where each circle corresponds to an edge in $\mathcal G\setminus\mathcal G'$. Thus, if we choose a subset $L\subset \{\gamma_1,\dots,\gamma_k\}$ such that $L\cap L^{-1}$ is empty and $L\cup L^{-1}=\{\gamma_1,\dots,\gamma_k\}$, then $L$ is a minimal generating set for $\pi_1(\mathcal G,q)$. 

Observe that the path $\mathsf b$ is homotopic to the concatenation
\[\mathsf b_1\cdot \mathsf b_2\cdot \ldots\cdot \mathsf b_k,\]
so we may write the corresponding element in $\pi_1(\mathcal G,q)$ as
\[\delta=\gamma_1\dots\gamma_k.\]
By Van Kampen's theorem, $\langle L \mid \delta\rangle$ is a presentation of $\pi_1(S,q)$. Since $H_1(S,\mathbb{Z})=\Zb^{2g}$, it now follows that $|\mathcal G\setminus\mathcal G'|=|L|=2g$ and $k=4g$. We refer to the sequence $\gamma_1,\dots,\gamma_{4g}$ as the \emph{relation sequence} associated to $(\mathcal G,\mathcal G',\mathsf b)$. This concludes the description of the topological framework necessary to describe the map ${\rm ob}_d$.

Now, choose a point $\widetilde{q}\in\pi_S^{-1}(q)$. This defines an identification between $\pi_1(S,q)$ and the deck group $\Gamma$ of the universal cover of $S$. For any representation $\rho:\Gamma\to\PGL_d(\Cb)$, choose $A_1,\dots,A_{4g}\in\SL_d(\Cb)$ such that $\rho(\gamma_i)$ is the projectivization of $A_i$ for all $i$, and so that $A_i=A_j^{-1}$ whenever $\gamma_i=\gamma_j^{-1}$. Note then that the product $A_1\cdot\ldots\cdot A_{4g}$ does not depend on this choice of representatives. Since $\id=\rho(\delta)$ is the projectivization of the product $A_1\cdot\ldots\cdot A_{4g}$, we may define ${\rm ob}_d(\rho)\in\mathbb{Z}_d$ to be the element that satisfies
\[\exp\left({\rm ob}_d(\rho)\right)\id=A_1\cdot\ldots\cdot A_{4g}.\]
Notice that $\rho$ lifts to a representation from $\Gamma$ to $\SL_d(\Cb)$ if and only if ${\rm ob}_d(\rho)=0$, so one can think of ${\rm ob}_d(\rho)$ as the obstruction to the existence of such a lift.

Since ${\rm ob}_d$ is continuous, it is constant on each connected component of $\Hom(\Gamma,\PGL_d(\Cb))$. Furthermore:

\begin{proposition}\label{ob bijection}
	The map ${\rm ob}_d$ descends to a bijection between the set of connected components of $\Hom(\Gamma,\PGL_d(\Cb))$ and $\Zb_d$ which does not depend on the choice of the graph $\mathcal G$ and its maximal subtree $\mathcal G'$.
\end{proposition}

\begin{proof}
	Following Steenrod \cite[Section~35]{steenrod1999}, the primary obstruction to the existence of global sections for a principal $\PGL_d(\Cb)$--bundle $E$ over a closed oriented surface $S$ is a cohomology class $c(E) \in H^2(S;\Zb_d)$. For completeness, we briefly recall the definition of $c(E)$ by specifying a $2$--cocycle representing it.
	
	Fix any CW-decomposition $K$ of $S$. Since $\PGL_d(\Cb)$ is connected, any section of $E|_{K^{(0)}}$ extends to a section of $E|_{K^{(1)}}$, where $K^{(0)}$ and $K^{(1)}$ are the $0$--skeleton and $1$--skeleton of $K$, respectively. Select arbitrarily a section $\sigma$ of $E|_{K^{(1)}}$ (which exists because the lie group $\PGL_d(\Cb)$ is connected). The $2$--cocycle representing $c(E)$ will depend on this choice.
	
	Let $f_c : D^2 \to K$ be the map associated to a $2$--cell $c$ of the CW-complex $K$, with gluing map $\partial f_c : \partial D^2 \to K^{(1)}$. Since $D^2$ is contractible, the $\PGL_d(\Cb)$--bundle $f_c^*(E)$ admits a trivialization $f_c^* (E) \cong D^2 \times \PGL_d(\Cb)$. Consider now 
	\[\sigma \circ \partial f_c : S^1 \to f_c^*(E) \cong D^2 \times \PGL_d(\Cb).\] 
	Since $\partial f_c$ can be thought of as a based loop oriented counterclockwise about $c$, by projecting onto the second component of the trivialization, the map $\sigma \circ \partial f_c$ gives an element $\gamma_c \in \pi_1(\PGL_d(\Cb))\cong\Zb_d$. (Since $\Zb_d$ is abelian, $\gamma_c$ does not depend on the choice of the base point. Also, $\gamma_c$ is independent of the choice of the trivialization.) Explicitly, if we lift the based loop $\sigma\circ\partial f_c$ in $\PGL_d(\Cb)$ to a path in $\SL_d(\Cb)$, then the starting and ending points of this path are two elements $A_-$ and $A_+$ in $\SL_d(\Cb)$ that have the same projectivization, and $\gamma_c\in\Zb_d$ is the number such that $A_+=\exp(\gamma_c)A_-$. (Recall that we think of the cyclic group $\Zb_d$ as a subgroup of $\Cb/2 \pi i \Zb$.) The association $c \mapsto \gamma_c$ determines a cellular $2$--cocycle in $C^2(S;\Zb_d)$, and hence a cohomology class in $c(E)\in H^2(S;\Zb_d)$, which does not depend on the choice of $\sigma$ and of cellular decomposition of $S$, see \cite[Corollary~35.8]{steenrod1999}.
	
	For any representation $\rho : \Gamma \to \PGL_d(\Cb)$, the associated flat bundle
	\[S \times_\rho \PGL_d(\Cb) \to S\] 
	is the quotient of the trivial bundle $\widetilde S\times\PGL_d(\Cb)\to\widetilde S$ by the following $\Gamma$--action: for all $\gamma\in\Gamma$ and $(p,g)\in \widetilde S\times\PGL_d(\Cb)$, $\gamma\, (p,g)=(\gamma\,p,\rho(\gamma)g)$. Li \cite{Li} proved that the map
	\[{\rm ob}_d':{\rm Hom}(\Gamma,\PGL_d(\Cb))\to H^2(S;\Zb_d)\cong\Zb_d\]
	given by ${\rm ob}_d'(\rho):=c(S \times_\rho \PGL_d(\Cb))$ is a surjection whose fibers are the connected components of ${\rm Hom}(\Gamma,\PGL_d(\Cb))$. It now suffices to verify that for all representations $\rho : \Gamma \to \PGL_d(\Cb)$, ${\rm ob}_d'(\rho)={\rm ob}_d(\rho)$, where ${\rm ob}_d(\rho)$ was defined above.
	
	In order to do so, consider the cellular decomposition of $S$ whose $1$--skeleton is equal to $\mathcal{G}$, and that has a unique $2$--cell $c$ whose interior is $S \setminus \mathcal{G}$. To use the description of ${\rm ob}_d'(\rho)$ given above, we define a section $\sigma$ of $S \times_\rho \PGL_d(\Cb)|_{\mathcal{G}}$. Since $S \times_\rho \PGL_d(\Cb) \to S$ is a flat bundle and the maximal tree $\mathcal{G}'$ inside $\mathcal{G}$ is contractible, we can set the section $\sigma$ to be constant on $\mathcal{G}'$ (e.g. $\sigma(p) = \mathrm{id}$ for all $p \in \mathcal{G}'$, with respect to some fixed trivialization on $\mathcal{G}'$). Then, for every oriented edge $\mathsf{e}_i$ in $\mathcal{G} \setminus\mathcal{G}'$, we select a path inside $\SL_d(\Cb)$ from $\mathrm{id}$ to the chosen lift $A_i$ and project it to a path $\alpha_i$ inside $\PGL_d(\Cb)$. Selecting a trivialization of $S\times_\rho\PGL_d(\Cb)$  along $\mathsf{e}_i$ in which $\sigma(\mathsf{e}_i(0)) = \mathrm{id}$, we then set $\sigma(\mathsf{e}_i(t)) = \alpha_i(t) \in \PGL_d(\Cb)$ for all $t \in [0,1]$. One can then check that this defines a continuous section $\sigma$ on $\mathcal{G}$ and that the evaluation of the $2$--cocycle contructed with such section on the unique $2$--cell is equal to $\mathrm{ob}_d(\rho)$, proving that ${\rm ob}_d'(\rho)={\rm ob}_d(\rho)$, as desired.
	\end{proof}

\section{Outline of proof of Theorem \ref{ThmD}}\label{tor=ob}

In Sections \ref{sec:topology_y} and \ref{sec:connected}, we defined the maps
\[{\rm tor}_d:\mathcal R(\lambda,d)\to\mathbb{Z}_d\qquad\text{ and }\qquad
{\rm ob}_d:{\rm Hom}(\Gamma,\PGL_d(\Cb))\to\mathbb{Z}_d,\]
respectively. (Recall that $\Zb_d$ here denotes the $d$--th roots of unity in $\Cb/2\pi i\Zb$.) We will now describe the strategy to prove Theorem \ref{ThmD} (and hence Theorem \ref{ThmB}) in the Introduction, which we restate here. 

\begin{theorem}\label{thm: final}
	For any $d$--pleated surface $\rho:\Gamma\to\PGL_d(\Cb)$, we have 
	\[{\rm ob}_d(\rho)={\rm tor}_d(\rho).\]
	In particular, every connected component of the representation variety $\Hom(\Gamma,\PGL_d(\Cb))$ contains exactly one connected component of the space $\mathcal R(\lambda,d)$ of $d$--pleated surfaces with pleating locus $\lambda$.
\end{theorem}

We now give a brief description of the strategy of the proof of Theorem~\ref{thm: final} that we will implement in the remainder of this paper. Recall that we fixed a train track neighborhood $N$ of $\lambda$, and a maximal tree $M$ in $N$.

\medskip

\noindent \textbf{The graph $\mathcal G$.} Given the train track neighborhood $N$ and a maximal tree $M$ in $N$, we define the embedded graph $\mathcal G \subset S$ as follows. For each connected component $C$ of $S\setminus N$, choose a point $p_C \in C$.  For each rectangle $R$ of $N$ that does not lie in $M$, let $C_{R,1}$ and $C_{R,2}$ be the connected components of $S\setminus N$ whose closures contain the two horizontal boundary components of $R$ (it is possible that $C_{R,1}=C_{R,2}$), and let $\check R$ be the truncated rectangle in $R$ (see Section \ref{ssub:max tree} for the necessary terminology). Then let $e_R$ be an embedded arc in $C_{R,1}\cup C_{R,2}\cup \check R$ such that the endpoints of $e_R$ are $p_{C_{R,1}}$ and $p_{C_{R,2}}$, $e_R$ intersects the leaves of $\lambda$ transversely, and $e_R$ intersects each horizontal boundary component of $R$ exactly once. We will also assume that for all distinct rectangles $R_1,R_2\subset N\setminus M$, $e_{R_1}$ and $e_{R_2}$ can only intersect at their endpoints. Let $\mathcal G\subset S$ be the graph whose set of vertices is $\{p_C:C\in\pi_0(S\setminus N)\}$ and whose set of edges is $\{e_R:R\text{ is a rectangle in }N \text{ but not in } M\}$, see Figure \ref{fig:G}.

\smallskip

\noindent \textbf{The loops $\mathsf c$ and $\mathsf b$.} Notice that $U:=S\setminus M$ is a neighborhood of $\mathcal G$ for which there is a strong deformation retract $H:\overline{U}\times[0,1]\to \overline{U}$ of $\overline U$ onto $\mathcal G$, such that the fibers of $H(\cdot,1)|_{\partial M}$ are finite, and consist of two points away from the vertices of $\mathcal G$, see Figure \ref{fig: deformation}. As discussed in Section \ref{sec:connected}, we denote by $\csf$ some fixed counterclockwise parameterization of $\partial \overline U$ and by $\bsf=H\circ \csf$ its retraction to $\mathcal G$. Note that $\partial \overline U = \partial M$. 

\begin{figure}[h!]
	\includegraphics[scale=1]{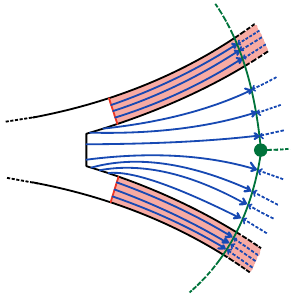}
	\caption{\label{fig: deformation}\small Maximal tree $M$ is in white, the graph $\mathcal G$ is in green, and the strong deformation retract $H$ is represented by the blue arrows.}
\end{figure}

Fix once and for all a maximal tree $\mathcal G'\subset \mathcal G$, and an endpoint $p$ of an exit of $M$ such that $q:=H(p,1)$ lies in $\mathcal G'$. Choose $\widetilde q\in\pi_S^{-1}(q)$; this defines an identification between the fundamental group $\pi_1(S,q)$ and the deck group $\Gamma$ of the universal cover $\pi_S:\widetilde S\to S$ as discussed in Sections \ref{sec: max geod} and \ref{sec:connected}. Let $\widetilde p\in\pi_S^{-1}(p)$ be the point such that the lift of the path $t\mapsto H(t,p)$ to $\widetilde S$ based at $\widetilde p$ has $\widetilde q$ as its other endpoint. Then let $\widetilde{\mathsf b}$ and $\widetilde{\mathsf c}$ be the lifts to $\widetilde S$ of $\mathsf b$ and $\mathsf c$ based at $\widetilde q$ and $\widetilde p$ respectively. 

\begin{figure}[h!]
	\includegraphics[width=\textwidth]{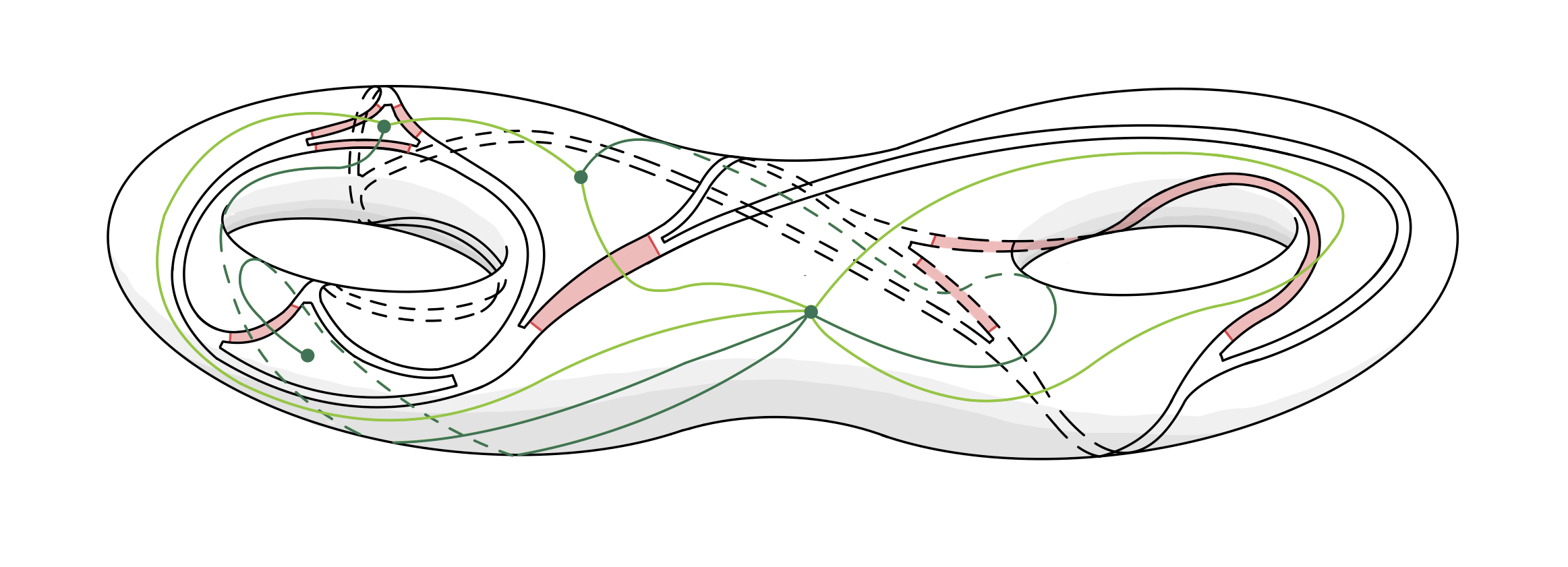}
	\caption{\small The graph $\mathcal G\subset S$ (drawn in green) with a maximal subtree $\mathcal{G}'$ (in dark green) from the choice of the maximal tree $M$ in $N$. Rectangles in $N$ but not in $M$ are drawn in pink.}
	\label{fig:G} 
\end{figure}

\smallskip

\noindent \textbf{The cutting sequence of $\widetilde{\csf}$ and the slithering coefficients.} Let $\Sigma$ denote the slithering map compatible with the $\lambda$--limit map of $\rho$ (see Section \ref{d-pleated}). We will use $\widetilde{\bsf}$ and $\widetilde{\csf}$ to specify a finite sequence of leaves $g_0,\dots,g_\ell$ of $\widetilde\lambda$ and a finite sequence of bases ${\bf v}(0),\dots,{\bf v}(\ell)$ of $\Cb^d$ with the following properties.
\begin{enumerate}
	\item[(I)] $\exp({\rm ob}_d(\rho))\,{\bf v}(0)={\bf v}(\ell)$.
	\item[(II)] For each $j\in\{1,\dots,\ell\}$ and $m\in\{1,\dots,d\}$, $\Sigma(g_j,g_{j-1})$ sends the vector $v_m(j-1)$ of the basis ${\bf v}(j-1)=(v_1(j-1),\dots,v_d(j-1))$ to a multiple $a_m(j)\in\Cb\setminus\{0\}$ of the vector $v_m(j)$ in the basis ${\bf v}(j)=(v_1(j),\dots,v_d(j))$, i.e.
	\begin{align}\label{eqn: slithering coefficients}
		\Sigma(g_j,g_{j-1})\, v_m(j-1)=a_m(j)\,v_m(j).
	\end{align}
\end{enumerate}
We refer to the sequence $g_0,\dots,g_\ell$ as the \emph{cutting sequence of $\widetilde{\csf}$}, the product 
\[\Sigma_\rho(\widetilde{\csf}):=\Sigma(g_\ell,g_{\ell-1})\cdot\ldots\cdot\Sigma(g_2,g_1)\cdot\Sigma(g_1,g_0)\in\SL_d(\Cb)\]
as the \emph{slithering along $\widetilde{\csf}$ associated to $\rho$}, the sequence ${\bf v}(0),\dots,{\bf v}(\ell)$ as a \emph{sequence of bases along $\widetilde{\csf}$ associated to $\rho$}, and the data 
\[\{a_m(j):j\in\{1,\dots,\ell-1\}\text{ and }m\in\{1,\dots,d\}\}\]
as the \emph{slithering coefficients of ${\bf v}(0),\dots,{\bf v}(\ell)$}. Notice that property (II) implies that for all $m\in\{1,\dots,d\}$, we have
\[\Sigma_\rho(\widetilde{\csf})\,v_m(0)=\left(\prod_{j=1}^\ell a_m(j)\right)v_m(\ell).\]

\smallskip

\noindent \textbf{Triviality of $\Sigma_\rho(\widetilde{\csf})$.} We will then prove that $\Sigma_\rho(\widetilde{\csf})=\id$, which together with property (I) implies that for all $m\in\{1,\dots,d\}$, we have
\[\exp({\rm ob}_d(\rho))\left(\prod_{j=1}^\ell a_m(j)\right)\,v_m(\ell)=\exp({\rm ob}_d(\rho))\,v_m(0)=v_m(\ell),\]
or equivalently, that for all $m\in\{1,\dots d\}$
\begin{equation}\label{eqn: op slithering}
	{\rm ob}_d(\rho)=-\log \left(\prod_{j=1}^\ell a_m(j)\right).
\end{equation}

\smallskip

\noindent \textbf{Slithering coefficients in terms on the bending data.} Finally, we compute 
\[\log\left(\prod_{j=1}^\ell a_{\lfloor\frac{d+1}{2}\rfloor}(j)\right)\] 
explicitly in terms of the data of the shear-bend $\lambda$--cocyclic pair of $\rho$, and observe that the expression we obtain is the expression for ${\rm tor}_d(\rho)$ given in Remark~\ref{tor formula}.

\medskip

The rest of this paper is organized as follows. In Section \ref{sec: cutting}, we define the cutting sequence of $\widetilde{\mathsf c}$ and prove that $\Sigma_\rho(\widetilde{\csf})=\id$, see Proposition \ref{prop: slithering is trivial} (In fact, we do this more generally for any tree in $N$.) Then, in Section~\ref{sec: family}, we define the bases along $\widetilde{\csf}$ associated to $\rho$ and observe from their definition that they satisfy properties (I) and (II), see Proposition \ref{prop: properties of sequences}. Finally, in Section \ref{sec: final}, we compute the slithering coefficients of ${\bf v}(0),\dots,{\bf v}(\ell)$ in terms of the bending data of $\rho$.

\section{Cutting sequences and slithering along boundaries of trees}\label{sec: cutting}

Recall that we fixed a train track neighborhood $N$ of the maximal geodesic lamination $\lambda$. In this section we define and study the notions of cutting sequence and slithering along the boundary of any tree inside $N$ associated with the choice of a $d$--pleated surface $\rho$ with pleating locus $\lambda$.

To provide a brief description of these objects, let us introduce the following terminology. Given a subtree $L$ of $N$ (see Section \ref{ssub:max tree}), we say that a point in $\partial L$ is a \emph{corner of $L$} if it is either an endpoint of an exit of $L$ or an endpoint of a vertical boundary component of $N$ that lies in $L$. Fix arbitrarily a corner $p_L$ of $L$ and let $\mathsf c_L$ be a loop based at $p_L$ that parametrizes the boundary of $L$ according to its counterclockwise orientation about $L$ (this makes sense because $L$ is topologically a disk and $S$ is oriented). For any choice of a lift $\widetilde{\mathsf c}_L$ of $\mathsf c_L$ to $\widetilde S$, the cutting sequence of $\widetilde{\mathsf c}_L$ is a specific ordered sequence of leaves $g_0,\dots,g_\ell$ of $\widetilde\lambda$ that are crossed by the loop $\widetilde{\mathsf c}_L$, and the slithering along $\widetilde{\mathsf{c}}_L$ associated to $\rho$ is the product
\[\Sigma_\rho(\widetilde{\mathsf c}_L):=\Sigma(g_\ell,g_{\ell-1})\cdot\ldots\cdot\Sigma(g_2,g_1)\cdot\Sigma(g_1,g_0)\in\SL_d(\Cb) ,\]
where $\Sigma$ denotes the slithering map compatible with the $\lambda$--limit map $\xi$ of $\rho$ (see Section \ref{d-pleated}). The main result of this section is Proposition \ref{prop: slithering is trivial}, which states that the linear transformation $\Sigma_\rho(\widetilde{\mathsf c}_L)$ is equal to the identity.

In particular, when the subtree $L$ is equal to the maximal tree $M$, Proposition \ref{prop: slithering is trivial} implies that $\Sigma_\rho(\widetilde{\csf}) = \mathrm{id}$, establishing one of the technical steps required for the proof of Theorem \ref{ThmD} (compare with Section \ref{tor=ob}).

\subsection{Type decomposition and cutting sequences}\label{sec: type}
Let $L\subset N$ be a tree. We write $\mathsf c_L$ as a concatenation of segments
\[\mathsf c_L=\mathsf k_1\cdot\ldots\cdot\mathsf k_\ell,\]
where the image of each $\mathsf k_j$ is one of the following:
\begin{enumerate}
	\item an exit of $L$, in which case $\mathsf k_j$ is of \emph{rectangle type};
	\item a vertical boundary component of $N$, in which case $\mathsf k_j$ is of \emph{switch type};
	\item a maximal subsegment of a horizontal boundary of $N$, in which case $\mathsf k_j$ is of \emph{leaf type}.
\end{enumerate}
We refer to the above decomposition of $\mathsf c_L$ as its \emph{type decomposition}, see Figure \ref{fig:type}. Notice that the cyclic sequence $\ksf_1,\dots,\ksf_\ell$ alternates between leaf type paths and paths that are either of rectangle or of switch type. Also, the type decomposition of $\mathsf c_L$ defines a type decomposition of the lift
\[
\widetilde{\mathsf c}_L= \widetilde{\mathsf k}_1\cdot\ldots\cdot\widetilde{\mathsf k}_\ell ,
\]
where for all $j\in\{1,\dots,\ell\}$, each segment $\widetilde{\mathsf k}_j$ is a lift of $\mathsf k_j$, see Figure \ref{fig:type}. If $\mathsf k_j$ is of rectangle, switch, or  leaf type, then we say the same for $\widetilde{\mathsf k}_j$.

For each $j\in\{1,\dots,\ell\}$, let $T_j$ be the plaque of $\widetilde{\lambda}$ that contains the forward endpoint $p_j$ of $\widetilde{\mathsf{k}}_j$ and let $s_j$ be the unoriented tie of $\widetilde{L}$ that contains $p_j$. Notice that we have four cases:
\begin{enumerate}
	\item $\widetilde\ksf_j$ is of rectangle type and $\widetilde\ksf_{j+1}$ is of leaf type,
	\item $\widetilde\ksf_j$ is of leaf type and $\widetilde\ksf_{j+1}$ is of rectangle type,
	\item $\widetilde\ksf_j$ is of switch type and $\widetilde\ksf_{j+1}$ is of leaf type,
	\item $\widetilde\ksf_j$ is of leaf type and $\widetilde\ksf_{j+1}$ is of switch type.
\end{enumerate}
In cases (1) and (2), the tie $s_j$ intersects exactly one edge of the plaque $T_j$. Denote that edge by $g_j$. In cases (3) and (4), the tie $s_j$ intersects exactly two edges of $T_j$, and we let $g_j$ be the edge of $T_j$ such that $p_j$ lies between $g_j\cap s_j$ and the other endpoint of the vertical boundary component of $N$ that lies in $s_j$, see Figure \ref{fig:type}. Finally, set $g_0:=g_\ell$. The \emph{cutting sequence} of $\widetilde{\mathsf c}_L$ is the finite sequence of leaves $g_0,\dots,g_\ell$ of $\widetilde\lambda$ specified above.

\begin{figure}[h!]
	\includegraphics[width=0.8\textwidth]{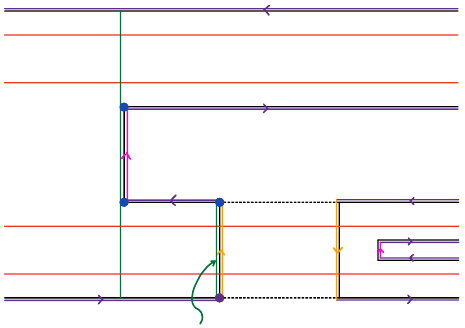}
	\put(-150,9){\small $\wt p_L$}
	\put(-223,6){\small $\wt{\ksf}_\ell$}
	\put(-178,-3){\small $s_1=s_\ell$}
	\put(-314,30){\small $g_\ell=g_0$}
	\put(-148,45){\small $\wt{\ksf}_1$}
	\put(-314,59){\small $g_1=g_2$}
	\put(-148,80){\small $p_1$}
	\put(-223,80){\small $p_2$}
	\put(-179,82){\small $\wt{\ksf}_2$}
	\put(-206,102){\small $\wt{\ksf}_3$}
	\put(-223,136){\small $p_3$}
	\put(-122,124){\small $\wt{\ksf}_4$}
	\put(-293,148){\small $g_3$}
	\put(-212,164){\small $s_2=s_3$}
	\caption{\small The type decompositions of $\widetilde{\mathsf c}_L=\widetilde{\ksf}_1\cdot\ldots\cdot\widetilde{\ksf}_\ell$. $\wt p_L$ is a lift of the corner $p_L$. The rectangle, switch, and leaf type segments are in orange, pink, and purple respectively. The ties $s_1=s_\ell$ and $s_2=s_3$ are in green, the corners $p_1$, $p_2$, and $p_3$ are in blue, and the first terms $g_0=g_\ell$, $g_1=g_2$ and $g_3$ of the cutting sequence of $\widetilde{\mathsf c}_L$ are in red.}
	\label{fig:type} 
\end{figure}
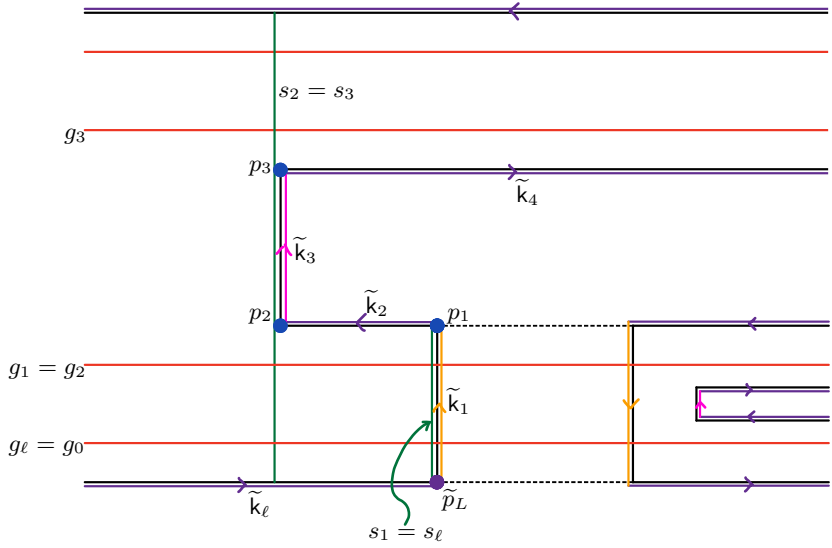

\subsection{Slithering along boundaries of trees}
Using the cutting sequence of $\widetilde{\mathsf c}_L$, we will now define the slithering along $\widetilde{\mathsf c}_L$ associated to $\rho$.

\begin{definition}\label{def:slithering along c}
	Let $\rho$ be a $d$--pleated surface with pleating locus $\lambda$, and let $\Sigma$ be the slithering map compatible with the $\lambda$--limit map of $\rho$. The \emph{slithering along $\widetilde{\mathsf c}_L$ associated to $\rho$} is the map
	\[\Sigma_\rho(\widetilde{\mathsf c}_L):=\Sigma(g_\ell,g_{\ell-1})\circ\dots\Sigma(g_3,g_2)\circ\Sigma(g_2,g_1)\circ\Sigma(g_1,g_0)\in\SL_d(\Cb).\]
\end{definition}

The next proposition states that the slithering along $\widetilde{\mathsf c}_L$ associated to $\rho$ is always trivial.

\begin{proposition}\label{prop: slithering is trivial} Let $\rho$ be a $d$--pleated surface with pleating locus $\lambda$ and let $L\subset N$ be a tree. Then, $\Sigma_\rho(\widetilde{\mathsf c}_L)=\id$.
\end{proposition}

\begin{proof} We prove this by induction on the number of truncated rectangles in $L$. First, note that by the $\rho$--equivariance of slithering maps, changing the choice of lift $\widetilde{\mathsf c}_L$ of $\mathsf c_L$ results in $\Sigma_\rho(\widetilde{\mathsf c}_L)$ being conjugated by $\rho(\gamma)$ for some $\gamma\in\Gamma$. Also, changing the corner of $L$ at which $\mathsf c_L$ is based results in $\Sigma_\rho(\widetilde{\mathsf c}_L)$ being conjugated by a subproduct of the product of slithering maps used to define $\Sigma_\rho(\widetilde{\mathsf c}_L)$. As such, it suffices to prove the base case and the inductive step for some fixed $\mathsf c_L$ and $\widetilde{\mathsf c}_L$, but we may assume that the inductive hypothesis holds for all $\mathsf c_L$ and $\widetilde{\mathsf c}_L$.
	
	In the base case $L$ has no truncated rectangles, i.e. $L$ is a stumpy switch. Let $\widetilde L$ be the connected component of $\pi_S^{-1}(L)$ that is bounded by $\widetilde{\mathsf c}_L$. We may assume that $\widetilde{\mathsf c}_L$ is based at the backward endpoint of the vertical boundary component in $\widetilde L$. In this case, the cutting sequence of $\widetilde{\mathsf c}_L$ is of the form $(g_0,\dots,g_8)$, where $h_0:=g_7=g_8=g_0$, $h_1:=g_1=g_2$, $h_2:=g_3=g_4$, $h_3:=g_5=g_6$, $h_1$ separates $h_0$ and $h_2$, and $h_0$ separates $h_3$ and $h_2$, see Figure \ref{fig: slithering is trivial}. Then by repeated applications of property (1) of Theorem \ref{thm: slithering map}, we have
	\begin{align*}
		\Sigma_\rho(\widetilde{\mathsf c}_L)&=\Sigma(h_0,h_3)\circ\Sigma(h_3,h_2)\circ\Sigma(h_2,h_1)\circ\Sigma(h_1,h_0)=\Sigma(h_0,h_2)\Sigma(h_2,h_0)=\id.
	\end{align*}
	This finishes the base case.
	\begin{figure}[h!]
		\centering
		\includegraphics[scale=1.2]{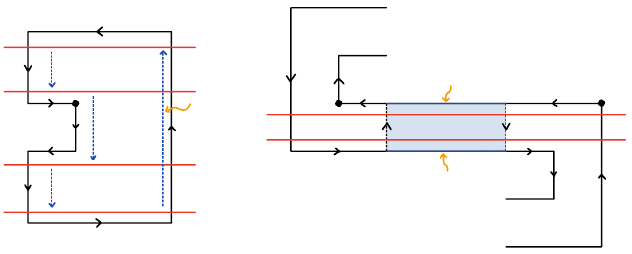}
		\put(-310,133){\small $L$}
		\put(-369,123){\small $h_3$}
		\put(-369,97){\small $h_0$}
		\put(-369,55){\small $h_1$}
		\put(-369,27){\small $h_2$}
		\put(-334,105){\small $\Sigma_{h_0,h_3}$}
		\put(-310,70){\small $\Sigma_{h_1,h_0}$}
		\put(-334,35){\small $\Sigma_{h_2,h_1}$}
		\put(-255,85){\small $\Sigma_{h_3,h_2}$}
		\put(-332,79){\small $\wt p_L$}
		\put(-161,69){\small $\ksf_{\ell-1}$}
		\put(-68,51){\small $\ksf_{3}'$}
		\put(-14,41){\small $\ksf_{m}'$}
		\put(-69,69){\small $\ksf_{2}'$}
		\put(-54,90){\small $\ksf_{1}'$}
		\put(-153,90){\small $\ksf_{\ell}$}
		\put(-112,69){\small $\check R$}
		\put(-155,130){\small $L_1$}
		\put(-178,104){\small $\ksf_1$}
		\put(-180,83){\small $\wt p_{L_1}$}
		\put(-14,88){\small $\wt p_{L_2}$}
		\put(0,78){\small $g_0$}
		\put(0,64){\small $g_{\ell-3}$}
		\put(-180,52){\small $\ksf_{\ell-2}$}
		\put(-112,99){\small $w_1$}
		\put(-112,43){\small $w_2$}
		\put(-70,10){\small $L_2$}
		\caption{\small On the left, the base case where $L$ is a stumpy switch. On the right, the inductive step with the truncated rectangle $\check R$ (blue) and the subtrees $L_1$ and $L_2$.}
	\label{fig: slithering is trivial}
\end{figure}

For the inductive step, let $\check R$ be an arbitrary truncated rectangle contained in $L$. There are exactly two connected components of $\partial_h N\cap L$ that intersect $\check R$, call them $w_1$ and $w_2$. Let $\mathsf{w}_1$ and $\mathsf{w}_2$ be $w_1$ and $w_2$ respectively, equipped with the orientation compatible with the orientation on $\mathsf c_L$, and we may assume that $\mathsf c_L$ is based at the forward endpoint of $\mathsf{w}_1$. Let $\widetilde L$ be the connected component of $\pi_S^{-1}(L)$ that is bounded by $\widetilde{\mathsf c}_L$.

Let $L_1$ and $L_2$ denote the two disjoint subtrees of $L$ such that $L=L_1\cup \check R\cup L_2$, and such that $L_1$ contains the forward endpoint of $\mathsf{w}_1$. Let $\mathsf c_{L_1}$ (respectively, $\mathsf c_{L_2}$) be a counterclockwise parameterization of the boundary of $L_1$ (respectively, $L_2$) that is based at the corner of $L_1$ (respectively, $L_2$) so that if $\mathsf c_{L_1}=\ksf_1\cdot\ldots\cdot\ksf_\ell$ (respectively, $\mathsf c_{L_2}=\ksf_1'\cdot\ldots\cdot\ksf_m'$) is its type decomposition, then the image of $\ksf_{\ell-1}$ (respectively, $\ksf_2'$) is $\check R\cap L_1$ (respectively, $\check R\cap L_2$). See Figure \ref{fig: slithering is trivial}. Then let $\widetilde{\mathsf c}_{L_1}$ and $\widetilde{\mathsf c}_{L_2}$ be the lifts of $\mathsf c_{L_1}$ and $\mathsf c_{L_2}$ whose images lie in $\widetilde L$, and let $(g_0,\dots,g_\ell)$ and $(g_0',\dots,g_m')$ be the cutting sequences of $\widetilde{\mathsf c}_{L_1}$ and $\widetilde{\mathsf c}_{L_2}$ respectively. 

Notice that $\ksf_{\ell-2}$, $\ksf_\ell$, $\ksf_1'$, and $\ksf_3'$ are of leaf type. It follows that $g_{\ell-3}=g_{\ell-2}=g_2'=g_3'$ and $g_{\ell-1}=g_{\ell}=g_0=g_m'=g_0'=g_1'$. Also, notice that
\[\mathsf c_L=\ksf_1\cdot\ldots\cdot\ksf_{\ell-3}\cdot\mathsf{w}_2\cdot \ksf_4'\cdot\ldots\cdot\ksf_m'\cdot\mathsf{w}_1\]
is the type decomposition of $\widetilde{\mathsf c}_L$, and $\mathsf w_1$ and $\mathsf w_2$ are both of leaf type. Thus, the cutting sequence of $\widetilde{\mathsf c}_L$ is $(g_0,\dots,g_{\ell-3},g_3',\dots,g_m',g_0)$. It follows that
\begin{align*}
	\Sigma_\rho(\widetilde{\mathsf c}_L)&=\Sigma(g_0,g_m')\circ\Sigma(g_m',g_{m-1}')\circ\cdots\circ\Sigma(g_4',g_3')\\
	&\quad\circ\Sigma(g_3',g_{\ell-3})\circ\Sigma(g_{\ell-3},g_{\ell-4})\circ\dots\circ\Sigma(g_1,g_0)\\
	&=\Sigma_\rho(\widetilde{\mathsf c}_{L_2})\circ(\Sigma(g_3',g_2')\circ\Sigma(g_2',g_1')\circ\Sigma(g_1',g_0'))^{-1}\\
	&\quad\circ(\Sigma(g_{\ell},g_{\ell-1})\circ\Sigma(g_{\ell-1},g_{\ell-2})\circ\Sigma(g_{\ell-2},g_{\ell-3}))^{-1}\circ \Sigma_\rho(\widetilde{\mathsf c}_{L_1})\\
	&=\Sigma_\rho(\widetilde{\mathsf c}_{L_2})\circ\Sigma(g_1',g_2')\circ\Sigma(g_2',g_1')\circ \Sigma_\rho(\widetilde{\mathsf c}_{L_1})\\
	&=\Sigma_\rho(\widetilde{\mathsf c}_{L_2})\circ \Sigma_\rho(\widetilde{\mathsf c}_{L_1}).
\end{align*}
Since $L_1$ and $L_2$ contain strictly fewer truncated rectangles than $L$, the inductive hypothesis implies that $\Sigma_\rho(\widetilde{\mathsf c}_{L_2})=\id= \Sigma_\rho(\widetilde{\mathsf c}_{L_1})$, so the inductive step follows.
\end{proof}

\section{A family of bases for $\widetilde{\mathsf c}$ associated to $\rho$} \label{sec: family}
Let $\rho$ be a $d$--pleated surface with pleating locus $\lambda$, and let $\widetilde{\csf}$ be the lift to $\widetilde{S}$ of the counterclockwise parametrization of $\partial M$, as defined in Section \ref{tor=ob}. In this section, we will specify a finite sequence ${\bf v}(0),\dots,{\bf v}(\ell)$ of bases of $\Cb^d$, called a \emph{sequence of bases along $\widetilde{\mathsf c}$ associated to $\rho$}. 
Then, we relate ${\bf v}(0),\dots,{\bf v}(\ell)$ to ${\rm ob}_d(\rho)$ and the slithering map of $\rho$ by showing that it satisfies properties (I) and (II) stated in Section \ref{tor=ob} (see Proposition~\ref{prop: properties of sequences}). This defines its slithering coefficients.
	
	\subsection{Bases of $\Cb^d$ adapted to triples of flags}\label{bases}
	
	As a preliminary step to specify a sequence of bases along $\widetilde{\mathsf c}$ associated to $\rho$, we take a small detour to set up some terminology for bases of $\Cb^d$ associated to triples of flags that are in general position.
	
	First, we define the notion of a basis of $\Cb^d$ adapted to a triple of flags in $\Fc(\Cb^d)$ that are in general position.
	
	We start by defining the notion of a basis of $\Cb^d$ adapted to a triple of flags in $\Fc(\Cb^d)$ in general position.
	
	\begin{definition}\label{def: adapted}
		Let ${\bf F}=(F_1,F_2,F_3)$ be a triple of flags in $\Fc(\Cb^d)$ in general position. A basis ${\bf g}=(g_1,\dots,g_d)$ of $\Cb^d$ is \emph{adapted to} ${\bf F}$ if $g_m\in F_1^m\cap F_3^{d-m+1}$ for all $m\in\{1,\dots,d\}$, and $g_1+\dots+g_d\in F_2^1$, see Figure \ref{fig:adapted}.
	\end{definition}
	
	\begin{figure}[h!]
		\centering
		\def\svgwidth{0.7\textwidth}
		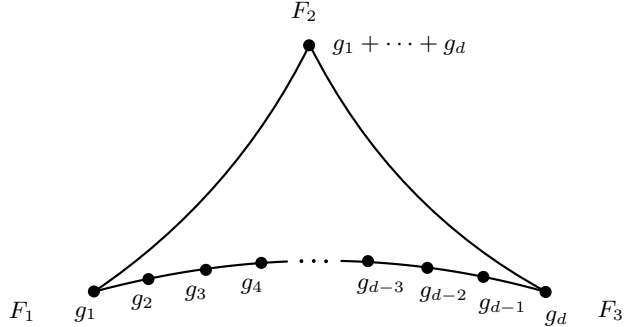 
		\caption{\small A basis $(g_1,\dots,g_d)$ of $\Cb^d$ adapted to $(F_1,F_2,F_3)$.}\label{fig:adapted} 
	\end{figure}
	
	Observe that any two bases adapted to the same triple of flags in general position are $\Cb^*$--scalar multiples of each other, where $\Cb^* = \Cb \setminus \{0\}$. 
	
	The following proposition describes how the unipotent element fixing $F_2$ and sending $F_1$ to $F_3$ relates the bases of $\Cb^d$ that are adapted to $(F_2,F_3,F_1)$ and $(F_3,F_1,F_2)$. Recall that $\Bc$ denotes the set of positive triples of integers ${\bf j}=(j_1,j_2,j_3)$ that sum to $d$. 
	
	\begin{proposition} \label{prop: unipotent}
		Let ${\bf F}=(F_1,F_2,F_3)$ be a triple of flags in $\Fc(\Cb^d)$ that are in general position, and let $u\in\mathsf{SL}_d(\mathbb C)$ be the unipotent linear transformation that fixes $F_2$ and sends $F_1$ to $F_3$. If ${\bf f}=(f_1,\dots, f_d)$ and ${\bf f}'=(f_1',\dots, f_d')$ are bases of $\Cb^d$ that are adapted to $(F_2,F_3,F_1)$ and $(F_3,F_1,F_2)$ respectively and such that $f_1=f_d'$ (see Figure \ref{fig:unipotent}), then for all $m\in\{1,\dots, d\}$ we have:
		\begin{equation}\label{eq:unipotent}
			u(f_m)=(-1)^{m-1}\exp\left(\sum_{{\bf j}\in\mathcal B\colon j_2<m}\tau^{\bf j}({\bf F})\right)f_{d-m+1}'.
		\end{equation}
	\end{proposition}
	
	\begin{figure}[h!]
		\centering
		\def\svgwidth{0.6\textwidth}
		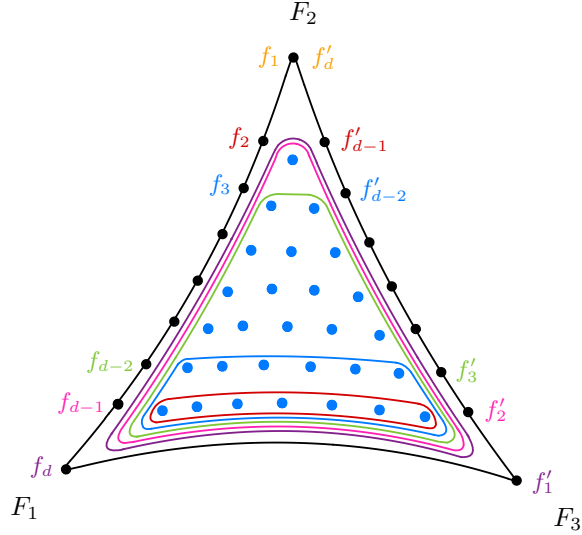
		\caption{\small The basis $(f_1,\dots,f_d)$ (resp. $(f'_1,\dots,f'_d)$) is adapted to $(F_2,F_3,F_1)$ (resp. $(F_3,F_1,F_2)$). For each pair of vectors $(f_m,f_{d-m+1}')$, the dots within the trapezoid of the same color describe the indices that appear in the sum in equation \eqref{eq:unipotent}.}
		\label{fig:unipotent} 
	\end{figure}
	
	\begin{proof} 
		We break down the proof into three steps. First, we establish the case $d=2$. Second, we use the first step to prove the statement for general $d$, but with the assumption that $\tau^{\bf j}({\bf F})=0$ for all ${\bf j}\in\Bc$. Finally, we use the second step to prove the proposition in general.
		
		\smallskip
		
		{\bf Step 1.} Assume that $d=2$. Since ${\bf f}$ and ${\bf f}'$ are adapted to $(F_2,F_3,F_1)$ and $(F_3,F_1,F_2)$ respectively, there exist $k,k'\in\Cb^*$ such that $f_1'=k'(f_1+f_2)$ and $kf_2=f_1'+f_2'$.
		Since $f_1=f_2'$, we have $k=k'=-1$. Since $u$ is a unipotent transformation that fixes $[f_1]$ and sends $[f_2]$ to $[f_1']$, it follows that $u(f_1)=f_1=f_2'$ and $u(f_2)=f_1+f_2=-f_1'$.
		
		\smallskip
		
		{\bf Step 2.} Assume that $\tau^{\bf j}({\bf F})=0$ for all ${\bf j}\in\mathcal B$. Let  $\text{Sym}^{d-1}(\mathbb C^2)$ denote the $d$--th symmetric power of $\mathbb C^2$, which is isomorphic to $\Cb^d$. The natural $\SL_2(\Cb)$--action on $\text{Sym}^{d-1}(\mathbb C^2)$ induces an irreducible representation $\phi\colon \mathsf{SL}_2(\mathbb C)\to \mathsf{SL}(\text{Sym}^{d-1}(\mathbb C^2))$, which in turn induces a $\phi$--equivariant embedding $\bar\phi:\Fc(\Cb^2)\to\Fc(\text{Sym}^{d-1}(\mathbb C^2))$. On the one hand, $\tau^{\bf j}({\bf F})=0$ for all ${\bf j}\in\mathcal B$ by hypothesis. On the other hand, every triple ${\bf F'} = (F_1',F_2',F_3') \in \mathcal{F}(\text{Sym}^{d-1}(\mathbb C^2))$ of pairwise distinct flags lying in the image of $\bar{\phi}$ satisfies $\tau^{\bf j}({\bf F}') = 0$ for all ${\bf j} \in \mathcal{B}$ (see for example, \cite[Section 5]{inagaki2021}). By Proposition \ref{prop: flag invariants}, $\{\tau^{\bf j}:{\bf j}\in\Bc\}$ is a complete collection of projective invariants for triples of flags, so by choosing an appropriate $\Cb$--linear isomorphism $\text{Sym}^{d-1}(\mathbb C^2)\cong\mathbb C^d$, we may ensure that there is a triple of flags ${\bf G}=(G_1,G_2,G_3)$ in $\Fc(\Cb^2)$ such that $\bar\phi({\bf G})={\bf F}\in \Fc(\text{Sym}^{d-1}(\mathbb C^2))\cong\Fc(\Cb^d)$. 
		
		Let $(g_1,g_2)$ and $(g_1',g_2')$ be bases of $\Cb^2$ that are adapted to $(G_2,G_3,G_1)$ and $(G_3,G_1,G_2)$, respectively, such that $g_1=g_2'.$ Let $w\in\SL_2(\Cb)$ be the unipotent element that fixes $G_2$ and sends $G_1$ to $G_3$. For each $m\in\{1,\dots,d\}$, let $h_m$ and $h_m'$ be the vectors in $\Cb^d$ that are identified with 
		\[{d-1 \choose m-1}g_1^{d-m}g_2^{m-1}\quad\text{and}\quad{d-1 \choose m-1}(g_1')^{d-m}(g_2')^{m-1}\]
		via our chosen isomorphism $\text{Sym}^{d-1}(\mathbb C^2)\cong\mathbb C^d$ and where $g^{k}$ denotes the $k$--th symmetric power of $g$. 
		
		As $\mathcal{F}(\Cb^2)$ is the set of projective classes of non-zero vectors of $\Cb^2$, we can describe the map $\bar{\phi} : \mathcal{F}(\Cb^2) \to \mathcal{F}(\text{Sym}^{d-1}(\mathbb C^2))$ as follows: for all $k\in\{1,\dots,d-1\}$ and all non-zero $v\in\Cb^2$, the $k$--dimensional subspace $\bar{\phi}([v])^k\subset \text{Sym}^{d-1}(\mathbb C^2)$ is equal to the subspace of symmetric tensors that are divided by $v^{d - k}$, i.e. of the form $v^{d - k} \, y$ for some $y \in \text{Sym}^{k - 1}(\mathbb C^2)$ (where by convention $\text{Sym}^{0}(\mathbb C^2) : = \Cb$). It follows that $h_m\in F_2^m\cap F_1^{d-m+1}$ and $h_m'\in F_3^m\cap F_2^{d-m+1}$ for all $m\in\{1,\dots,d\}$. 
		
		Note that $h_1=g_1^{d-1}=(g_2')^{d-1}=h_d'$. Furthermore,
		\[h_1'=(g_1')^{d-1}=(-1)^{d-1}(w(g_2))^{d-1}=(-1)^{d-1}(g_1+g_2)^{d-1}=(-1)^{d-1}\sum_{m=1}^{d-1}h_m,\]
		where the second equality holds by Step 1, and the third equality holds because $w$ is the unipotent tranformation of $\Cb^2$ that fixes $[g_1]$ and sends $[g_2]$ to $[g_1']=[g_1+g_2]$.
		Similarly, we have
		\[h_d=(-1)^{d-1}\sum_{m=1}^{d-1}h_m'.\]
		Thus, ${\bf h} = (h_1, \dots, h_d)$ and ${\bf h}' = (h_1', \dots, h_d')$ are bases of $\Cb^d$ that are adapted to $(F_2,F_3,F_1)$ and $(F_3,F_1,F_2)$ respectively, and they satisfy $h_1=h_d'$.
		
		Since the pair of bases $({\bf f},{\bf f}')$ from the statement of this proposition and $({\bf h},{\bf h}')$ are $\Cb^*$--multiples of each other, to finish Step 2, it suffices to verify that
		\[u(h_m)=(-1)^{m-1}h_{d-m+1}'\]
		for all $m\in\{1,\dots,d\}$ (recall that we assumed $\tau^{\bf j}({\bf F})=0$ for every ${\bf j} \in \Bc$). Indeed, since $u=\phi(w)$, 
		\begin{align*}
			u(h_m)&={d-1 \choose m-1}\phi(w)(g_1^{d-m}g_2^{m-1})={d-1 \choose d-m}w(g_1)^{d-m}w(g_2)^{m-1}\\
			&=(-1)^{m-1}{d-1 \choose d-m}(g_1')^{m-1}(g_2')^{d-m}=(-1)^{m-1}h'_{d-m+1} ,
		\end{align*}
		where the third equality holds because $w(g_1) = g_1 = g_2'$ and $w(g_2) = - g_1'$ by Step 1. This concludes the proof of Step 2.
		
		\smallskip
		
		{\bf Step 3:} Let $H_3\in\Fc(\Cb^d)$ be the flag such that $H_3^1=F_3^1$ and $\tau^{\bf j}(F_1,F_2,H_3)=0$ for all $\textbf{j}\in\mathcal B$. Let $a\in\mathsf{PGL}_d(\mathbb C)$ be the projective transformation that fixes $F_1^1$ and $F_2$, and sends $H_3$ to $F_3$. Then let $\bar{a}\in\GL_d(\Cb)$ be the linear representative of $a$ that fixes $f_1=f_d'$. Observe that if ${\bf f}'':=(f_1'',\dots,f_d'')$ is the basis of $\Cb^d$ that is adapted to $(H_3,F_1,F_2)$ such that $f_d''=f_1$, then $\bar{a}({\bf f}'')={\bf f}'$. Notice also that, if we decompose
		\[\bar{a}=\bar{s}\bar{u},\]
		where $\bar{u}$ is unipotent and $\bar{s}$ is diagonal in the basis ${\bf f}'$, then $\bar{u}$ fixes $F_2$ and sends $H_3$ to $F_3$.
		
		Let $u'\in\SL_d(\Cb)$ be the unipotent element that fixes $F_2$ and sends $F_1$ to $H_3$. By construction, the unipotent linear transformation that fixes $F_2$ and sends $F_1$ to $F_3$ is $u = \bar{u} u'$. Notice also that the basis ${\bf f}$, which by hypothesis is adapted to $(F_2,F_3,F_1)$, is also adapted to $(F_2,H_3,F_1)$ since $H_3^1 = F_3^1$. Since $\tau^{\bf j}(F_2,H_3,F_1) = 0$ for all ${\bf j}$, we can apply Step 2 to $u'$ and conclude that
		\[
		u'(f_m) = (-1)^{m-1} f_{d-m+1}'' 
		\]
		for all $m \in \{1,\dots,m\}$. We deduce that
		\[u(f_m)=\bar{u}u'(f_m)=(-1)^{m-1}\bar{s}^{-1}\bar{a}(f''_{d-m+1})=(-1)^{m-1}\bar{s}^{-1}(f'_{d-m+1}).\]
		
		It thus suffices to show that for all $m\in\{1,\dots,d\}$,
		\[\bar{s}(f'_{d-m+1})=\exp\left(-\sum_{{\bf j}\in\mathcal B\colon j_2<m}\tau^{\bf j}({\bf F})\right)f'_{d-m+1}.\]
		
		By \cite[Proposition 9.4(2)]{MMMZ1}, $a$ has a linear representative $\wh a$ in $\GL_d(\Cb)$ that is represented in the basis ${\bf f}'$ by a lower triangular matrix whose $(d-m+1)$--th diagonal entry down the diagonal is
		\[
		\exp\left(\sum_{{\bf j}\in\mathcal B:j_2\ge m}2r^{\bf j}-\sum_{{\bf j}\in\mathcal B:j_2<m}r^{\bf j}\right).
		\]
		where for all ${\bf j}\in\mathcal B$, $r^{\bf j}\in\Cb/2\pi i\Zb$ satisfies $3r^{\bf j}=\tau^{\bf j}({\bf F})$. (Notice that the different choices of $r^{\bf j}$'s produce matrices that differ by multiplication by a common multiple of $\exp(\frac{2 \pi i}{3})$.) Since $\bar{a}$ fixes $f_d'$, it follows that 
		\[\bar{a}=\exp\left(- 2 \sum_{{\bf j} \in \Bc} r^{\bf j} \right) \wh a,\]
		so $\bar s$ is represented in the basis ${\bf f}'$ by the diagonal matrix whose $(d-m+1)$--th diagonal entry down the diagonal is
		\[
		\exp\left(- 2 \sum_{{\bf j} \in \Bc} r^{\bf j} \right) \exp\left(\sum_{{\bf j}\in\mathcal B:j_2\ge m}2r^{\bf j}-\sum_{{\bf j}\in\mathcal B:j_2<m}r^{\bf j}\right) = \exp\left(-\sum_{{\bf j}\in\mathcal B:j_2<m}\tau^{\bf j}(\textbf{F})\right).
		\]
		Thus, the required equality holds.
	\end{proof}
	
	Motivated by Proposition \ref{prop: unipotent}, we make the following definition.
	
	\begin{definition}\label{def: compatible}
		Let ${\bf F}=(F_1,F_2,F_3)$ be a triple of flags that is in general position, and let $r\in\Cb/2\pi i\Zb$ be such that
		\[3r=\sum_{\textbf{j}\in\mathcal B}\tau^{\bf j}(\bf F).\]
		We say that a triple of bases 
		\[({\bf f},{\bf g},{\bf h})=((f_1,\dots,f_d),(g_1,\dots,g_d),(h_1,\dots,h_d))\] 
		of $\Cb^d$ is \emph{compatible} with $({\bf F},r)$ if:
		\begin{enumerate}
			\item  ${\bf f}$, ${\bf g}$, and ${\bf h}$ are adapted to $(F_3,F_1,F_2)$, $(F_1,F_2,F_3)$, and $(F_2,F_3,F_1)$ respectively, see Figure \ref{fig:adapted2}. 
			\item $\displaystyle \exp\left(2r\right)\,f_1=g_d$ and $\displaystyle \exp\left(2r\right)\,g_1=h_d$.
		\end{enumerate}
	\end{definition}
	
	\begin{figure}[h!]
		\centering
		\def\svgwidth{0.6\textwidth}
		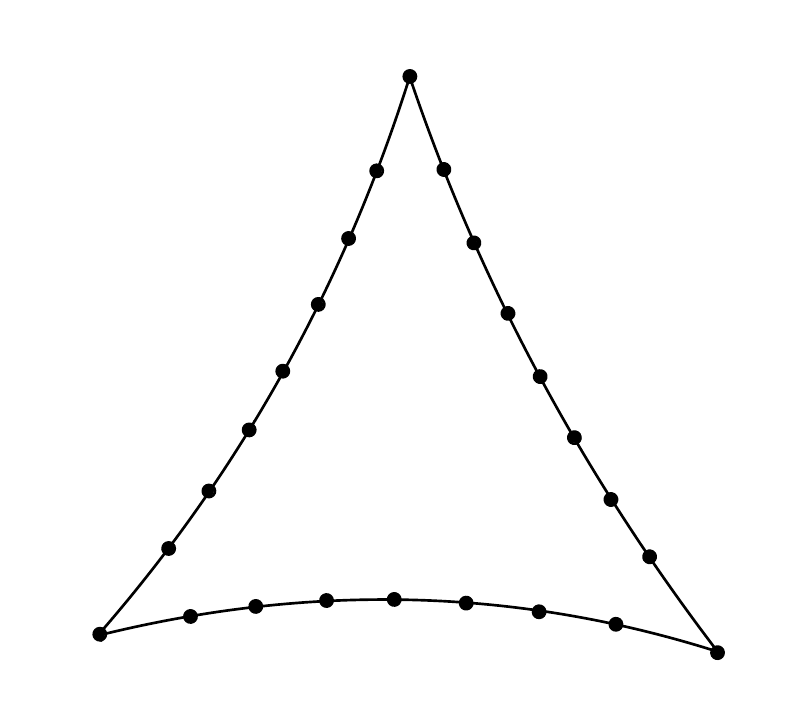
		\caption{\small Bases $({\bf f},{\bf g},{\bf h})$ of $\Cb^d$ compatible with $({\bf F},r)$.}
		\label{fig:adapted2} 
	\end{figure}
	
	The reason we impose Property (2) in Definition \ref{def: compatible} is to guarantee a certain cyclic symmetry in the choice of a compatible basis. More precisely, the implication (1) $\implies$ (2) in Corollary~\ref{cor: rotate} below can be reformulated as saying that Property (2) in Definition \ref{def: compatible} implies that the equality $\exp(2r)\, h_1 = f_d$ also holds.
	
	Let ${\bf F}=(F_1,F_2,F_3)$ be a triple of flags that is in general position. Notice that for any basis ${\bf g}$ of $\Cb^d$ that is adapted to ${\bf F}$ and any $r\in\Cb/2\pi i\Zb$ such that
	\[3r=\sum_{\textbf{j}\in\mathcal B}\tau^{\bf j}(\bf F),\] 
	there is a pair of bases ${\bf f}$ and ${\bf h}$ of $\Cb^d$ such that $({\bf f},{\bf g},{\bf h})$ is compatible with $({\bf F},r)$. Furthermore, any triple of basis of $\Cb^d$ which is compatible with $({\bf F},r)$ is a $\Cb^*$--multiple of $({\bf f},{\bf g},{\bf h})$.
	
	For any basis ${\bf f}=(f_1,\dots,f_d)$ of $\Cb^d$, denote
	\[{\bf f}^{\rm op}:=(f_d,\dots,f_1).\]
	
	\begin{corollary}\label{cor: rotate}
		Let $({\bf f},{\bf g},{\bf h})$ be a triple of bases of $\Cb^d$, let ${\bf F}=(F_1,F_2,F_3)$ be a triple of flags that is in general position, and let $r\in\Cb/2\pi i\Zb$ be such that
		\[3r=\sum_{\textbf{j}\in\mathcal B}\tau^{\bf j}(\bf F).\] 
		Then the following are equivalent:
		\begin{enumerate}
			\item $({\bf f},{\bf g},{\bf h})$ is compatible with $({\bf F},r)$. 
			\item $({\bf g},{\bf h},{\bf f})$ is compatible with $((F_2,F_3,F_1),r)$.
			\item $({\bf h}^{\rm op},{\bf g}^{\rm op},{\bf f}^{\rm op})$ is compatible with $((F_3,F_2,F_1),-r)$. 
		\end{enumerate}
	\end{corollary}
	
	\begin{proof}
		The equivalence between (1) and (3) is immediate from the observation that 
		\[\sum_{\textbf{j}\in\mathcal B}\tau^{\bf j}({\bf F})=-\sum_{\textbf{j}\in\mathcal B}\tau^{\bf j}((F_3,F_2,F_1)).\]
		
		To prove that (1) and (2) are equivalent, it suffices to show that (1) implies (2). Suppose that (1) holds. Since 
		\[\sum_{{\bf j}\in\mathcal B}\tau^{\bf j}({\bf F})=\sum_{{\bf j}\in\mathcal B}\tau^{\bf j}((F_2,F_3,F_1)),\]
		we need only to verify that $f_d=\exp(2r)h_1$. For $k=1,2,3$, let $u_k\in\SL_d(\Cb)$ be the unipotent element that fixes $F_k$ and sends $F_{k-1}$ to $F_{k+1}$ (arithmetic in the subscripts is done modulo $3$). Let ${\bf g}'=\exp(-2r)\,{\bf g}$  and ${\bf h}'=\exp(-2r)\,{\bf h}$, and note that $g_d'=f_1$ and $h_d'=g_1$. By applying Proposition \ref{prop: unipotent} twice (once on the pair of bases ${\bf g}$, ${\bf h}'$ and then on the pair ${\bf f}$, ${\bf g}'$) we deduce that
		\[h_1=\exp(2r)h_1'=(-1)^{d-1}\exp(-r)u_1(g_d)\quad\text{and}\quad g_1=(-1)^{d-1}\exp(-r)u_3(f_d).\]
		
		Since ${\bf g}$ is adapted to ${\bf F}$,
		\[[u_1(g_d)]=[h_1]=\left[\sum_{m=1}^dg_m\right]=[f_d]=[u_3^{-1}(g_1)].\]
		In particular, $u_1(g_d)$ and $u_3^{-1}(g_1)$ are both non-zero multiples of $\sum_{m=1}^dg_m$. Notice that in the basis ${\bf g}$, $u_1$ is unipotent and upper triangular while $u_3^{-1}$ is unipotent and lower triangular, so
		\[\displaystyle u_1(g_d)=\sum_{m=1}^dg_m=u_3^{-1}(g_1).\] 
		It now follows that 
		\[h_1=(-1)^{d-1}\exp(-r)u_1(g_d)=(-1)^{d-1}\exp(-r)u_3^{-1}(g_1)=\exp(-2r)f_d\]
		as required.
	\end{proof}
	
	\subsection{Enhanced cutting sequence of $\widetilde{\mathsf c}$}\label{ssec: bases}
	Now, we return to the problem of specifying a sequence of bases along $\widetilde{\mathsf c}$ associated to the $d$--pleated surface $\rho$ with pleating locus $\lambda$. Recall that $M$ denotes the fixed maximal tree in the train track neighborhood $N$ of $\lambda$, and $\pi_S:\widetilde S\to S$ denotes the covering map. Previously in Section \ref{tor=ob}, we constructed a particular graph $\mathcal G\subset S$ for which there is a strong deformation retract 
	\[H:\overline{S\setminus M}\times[0,1]\to\overline{S\setminus M}\] 
	onto $\mathcal G$ whose fibers are finite and consist of two points away from the vertices of $\mathcal G$. We also fixed once and for all a maximal tree $\mathcal G'\subset\mathcal G$, a point $p\in\partial M$ such that $q:=H(p,1)\in\mathcal G'$, and denote by $\widetilde q\in\pi_S^{-1}(q)$ and $\widetilde p\in\pi_S^{-1}(p)$ the lifts determined by our chosen  identification between $\pi_1(S,q)$ with the deck group $\Gamma$ of $\pi_S:\widetilde S\to S$, see Section \ref{tor=ob}. We then denoted by $\mathsf c$ the boundary of $M$ based at $p$ oriented counterclockwise about $M$, defined $\mathsf b:=H(\cdot,1)\circ\mathsf c$, and denoted by $\widetilde{\csf}$ and $\widetilde{\bsf}$  the lifts of $\csf$ and $\bsf$ to $\widetilde S$ based at $\widetilde p$ and $\widetilde q$ respectively. Also, let $\widetilde M$ denote the connected component of $\pi_S^{-1}(M)$ that is bounded by $\widetilde{\mathsf c}$, and notice that the chosen orientation on the ties of $M$ induces an orientation on the ties of $\widetilde M$.
	
	Let $\mathsf c=\mathsf k_1\cdot\ldots\cdot\mathsf k_\ell$ and $\widetilde{\mathsf c}=\widetilde{\mathsf k}_1\cdot\ldots\cdot\widetilde{\mathsf k}_\ell$ be the type decompositions, and let $g_0,\dots,g_\ell$ be the cutting sequence of $\widetilde{\mathsf c}$ (see Section \ref{sec: type} for definitions). Recall that $\widetilde{\Delta}^o$ denotes the set of triples of endpoints of leaves of $\widetilde\lambda$ that are the vertices of plaques of $\widetilde\lambda$, see Section~\ref{sec: max geod}. To define the sequence of bases ${\bf v}(0),\dots,{\bf v}(\ell)$, we will need to first specify a finite sequence of triples ${\bf x}(0),\dots,{\bf x}(\ell)\in\widetilde\Delta^o$, called the \emph{enhanced cutting sequence} of $\widetilde{\mathsf c}$. For each $j\in\{0,\dots,\ell\}$, the basis ${\bf v}(j)$ that we specify later will be adapted to the triple of flags $\xi({\bf x}(j))$, where $\xi$ is the $\lambda$--limit map of $\rho$.
		
		For each $j\in\{1,\dots,\ell\}$, let $T_j$ be the plaque of $\widetilde\lambda$ that contains the forward endpoint of $\widetilde{\mathsf k}_j$, and let $\gsf_j$ be the leaf $g_j$ equipped with the orientation so that the ties of $\widetilde{M}$, with the orientations induced by the chosen orientations on the ties of $M$, pass from the left to the right of $\gsf_j$, see Figure \ref{fig:x(j)}. Notice that $\gsf_j$ is an oriented edge of $T_j$, so we may denote by 
		\[{\bf x}(j)=(x_1(j),x_2(j),x_3(j))\in\widetilde\Delta^o\] 
		the triple such that $x_1(j)$ and $x_3(j)$ are respectively the forward and backward endpoints of $\gsf_j$, and $x_2(j)$ is the vertex of $T_j$ that is not an endpoint of $\gsf_j$. Set also $\gsf_0:=\gsf_\ell$, $T_0:=T_\ell$, and ${\bf x}(0) : = {\bf x}(\ell)$. We refer to the sequence of triples
		\[{\bf x}(0),\dots,{\bf x}(\ell)\in\widetilde\Delta^o\]
		as the \emph{enhanced cutting sequence} of $\widetilde{\mathsf c}$.
		
		\begin{figure}[h!]
			\centering
			\hspace*{-1.5cm}\resizebox{1.2\textwidth}{!}{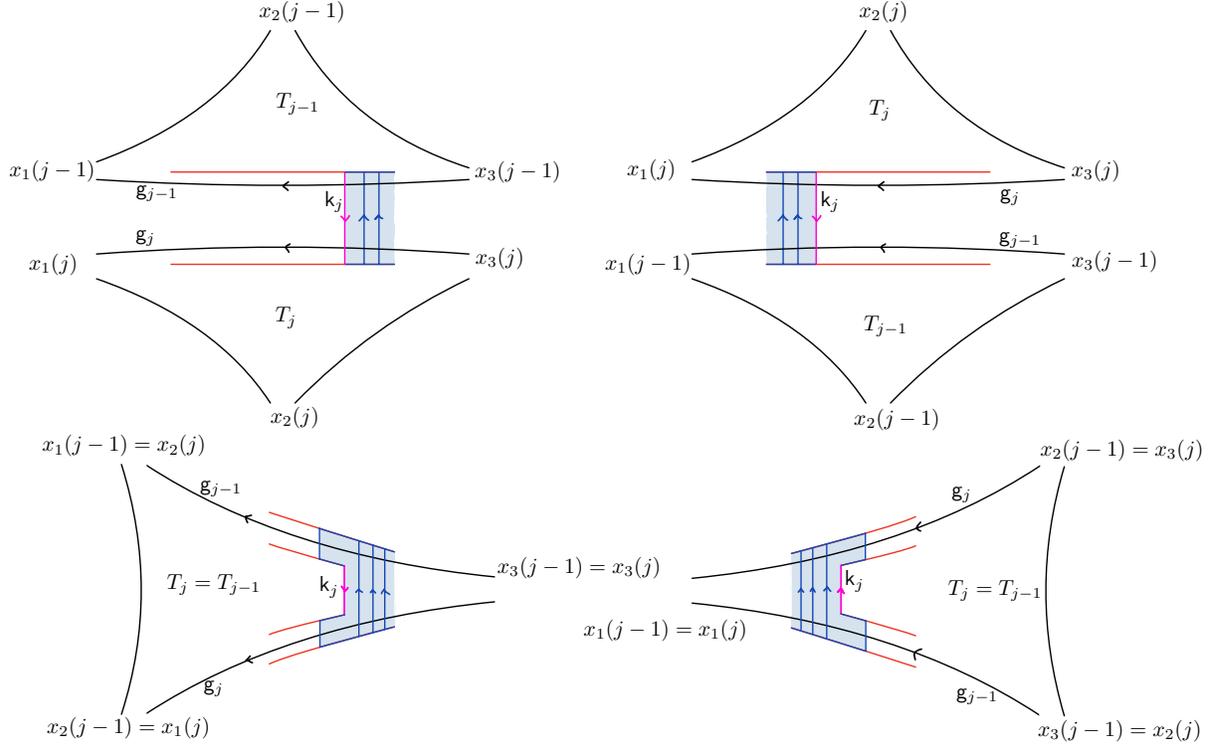}
			\caption{\small From top left to bottom right: $\mathsf k_j$ lies in a  left exit, right exit, left vertical boundary component, and right vertical boundary component. The blue shaded regions lie in $M$.}
			\label{fig:x(j)} 
		\end{figure}
		
		Since $\mathsf c$ is oriented counterclockwise about $M$, we observe:
		\begin{enumerate}
			\item If $\ksf_j$ is of rectangle type, then the image of $\ksf_j$ is a right (respectively, left) exit if and only if the orientation on $\ksf_j$ agrees with (respectively, opposes) the orientation on the ties of $M$, or equivalently, if and only if $\widetilde{\ksf}_j$ passes from the left to the right (respectively, the right to the left) of both the oriented leaves $\gsf_{j-1}$ and $\gsf_j$. See top two figures of Figure \ref{fig:x(j)}.
			\item If $\ksf_j$ is of switch type, then the image of $\ksf_j$ is a right (respectively, left) vertical boundary component of $N$ if and only if the orientation on $\ksf_j$ agrees with (respectively, opposes) the orientation on the ties of $M$, or equivalently, if and only if $\gsf_{j-1}$ and $\gsf_j$ share a common forward (respectively, backward) endpoint. See bottom two figures of Figure \ref{fig:x(j)}.
			\item If $\ksf_j$ is of leaf type, then $\gsf_{j-1}=\gsf_j$, $T_{j-1}=T_j$, and ${\bf x}(j-1)={\bf x}(j)$.
		\end{enumerate}
		
		\subsection{Families of bases along $\widetilde{\mathsf b}$}\label{sec: families}
		Since $M$ is a closed disk and since there is a strong deformation retract of $S\setminus M$ onto $\mathcal G$, it follows that $S\setminus \mathcal G$ is an open disk. Thus, we are now in the situation where we can apply the description of ${\rm ob}_d(\rho)$ given in Section~\ref{sec:connected}. Specifically, let $\gamma_1,\dots,\gamma_{4g}$ be the relation sequence associated to $(\mathcal G,\mathcal G',\mathsf b)$, and fix a choice of $A_1,\dots, A_{4g}\in\SL_d(\Cb)$ such that $\rho(\gamma_i)$ is the projectivization of $A_i$ for all $i$, and $A_i=A_j^{-1}$ whenever $\gamma_i=\gamma_j^{-1}$. Also, let $(\alpha,\theta)$ be the shear-bend $\lambda$--cocyclic pair for $\rho$, i.e. $(\alpha,\theta):=\mathfrak{sb}_d([\rho]) \in \mathcal{Y}(\lambda, d; \Cb/2\pi i\Zb)$, and fix a $\Gamma$--invariant map $r:\widetilde\Delta\to\Cb/2\pi i\Zb$ such that
		\[3r(T)= \sum_{\textbf{j}\in\mathcal B}\theta^{\bf j}({\bf x}_T),\]
		where we recall that $\widetilde\Delta$ is the set of plaques of $\widetilde\lambda$, and ${\bf x}_T\in\widetilde\Delta^o$ is some (any) clockwise ordering of the vertices of some (any) lift to $\widetilde S$ of $T$, and $\mathcal B$ denotes the set of triples of positive integers that sum to $d$. Recall that the right hand side does not depend on the choice of ${\bf x}_T$ because of the symmetry and $\Gamma$--invariance of $\theta$, see (2) and (4) of Definition \ref{def_cocycle}. Using these choices, we will now associate a family of bases of $\Cb^d$ to each connected component of $\pi_S^{-1}(\mathcal G')$ that $\widetilde{\bsf}$ passes through. 
		
		More precisely, consider the graph decomposition 
		\[\mathsf b=\mathsf f_1\cdot\mathsf e_1\cdot\mathsf f_2\cdot\mathsf e_2\cdot\ldots\cdot\mathsf f_{4g}\cdot\mathsf e_{4g}\cdot\mathsf f_{4g+1}\]
		of $\mathsf{b}$ associated to $(\mathcal G,\mathcal G')$, where $\mathsf{f}_{i}$ are in $\mathcal G'$, while $\mathsf{e}_{i}$ are in $\mathcal G \setminus \mathcal G'$. Since $\widetilde{\mathsf b}$ is the lift of $\mathsf b$ to $\widetilde S$ that is based at $\widetilde q$, we may decompose
		\[\widetilde{\mathsf b}=\widetilde{\mathsf f}_1\cdot\widetilde{\mathsf e}_1\cdot\widetilde{\mathsf f}_2\cdot\widetilde{\mathsf e}_2\cdot\ldots\cdot\widetilde{\mathsf f}_{4g}\cdot\widetilde{\mathsf e}_{4g}\cdot\widetilde{\mathsf f}_{4g+1}\]
		where each $\widetilde{\mathsf f}_i$ is a lift of $\mathsf f_i$ and each $\widetilde{\mathsf e}_i$ is a lift of $\mathsf e_i$. By construction, each $\widetilde{\mathsf f}_{i+1}$ lies in a connected component of $\pi_S^{-1}(\mathcal G')$, which we denote by $\mathcal C_i$. Since $\mathcal G'\subset \mathcal G$ is a maximal tree, the restriction of $\pi_S$ to each $\mathcal C_i$ is a homeomorphism onto $\mathcal G'$, see Figure \ref{fig: junk}.
		
		\begin{figure}[h!]
			\centering
			\def\svgwidth{0.9\textwidth}
			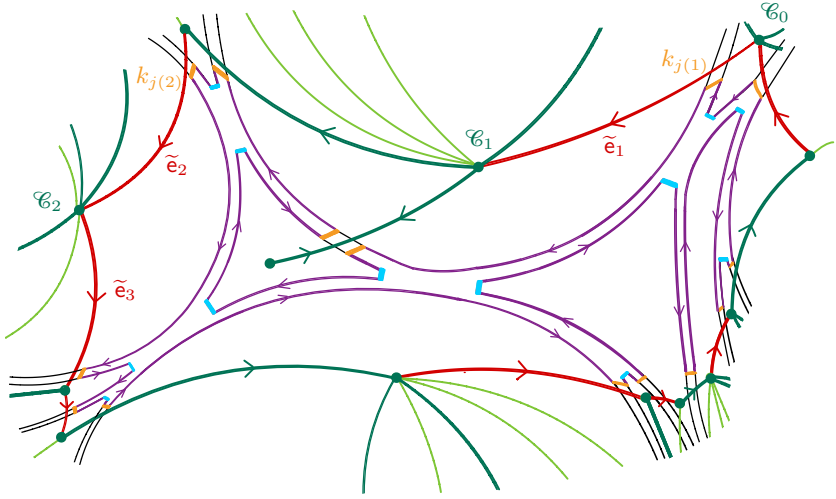
			\caption{\label{fig: junk} \small The path $\widetilde{\mathsf b}$ is a concatenation of $\widetilde{\mathsf e}_i$'s (in red) and the $\widetilde{\mathsf f}_i$'s (in dark green,  with arrows). Notice $\widetilde{\mathsf b}$ need not be injective. Each $\mathcal C_i$ is a connected union of dark green edges. The image of $\widetilde{\mathsf c}$ consists of the orange, cyan, and purple curves, which are respectively the rectangle type, switch type, and leaf type subsegments. }
		\end{figure}
		
		Set $\omega_0:=\id$, and for each $i\in\{1,\dots,4g\}$, let 
		\[\omega_i:=\gamma_1\ldots\gamma_i\quad\text{and}\quad\gamma_i':=\omega_{i-1}\gamma_i\omega_{i-1}^{-1}.\] 
		By definition, for all $i\in\{1,\dots,4g\}$, the loop ${\mathsf b}_i$ in $\mathcal G$ based at $q$ that represents $\gamma_i$ can be written as
		\[{\mathsf b}_i={\mathsf a}_{i,-}\cdot\mathsf e_i\cdot{\mathsf a}_{i,+},\]
		where ${\mathsf a}_{i,-}$ and ${\mathsf a}_{i,+}$ lie in $\mathcal G'$ (compare with Section~\ref{sec:connected}). Also, the backward and forward endpoints of the oriented edge $\widetilde{\mathsf e}_i$ are vertices of $\mathcal C_{i-1}$ and $\mathcal C_i$ respectively. Thus, the lift $\widetilde{\mathsf b}_i$ of $\mathsf b_i$ to $\widetilde S$ that contains $\widetilde{\mathsf e}_i$ as a subsegment has its backward and forward endpoints in $\mathcal C_{i-1}$ and $\mathcal C_i$ respectively. Note that the forward endpoint of $\widetilde{\mathsf b}_i$ is $\omega_i\,\widetilde q$ and so
		\[\mathcal C_i=\omega_i\,\mathcal C_0=\gamma_i'\,\mathcal C_{i-1}.\]
		
		For each $i\in\{0,\dots,4g\}$, let $\widetilde\Delta^o(\mathcal C_i)$ denote the set of ordered triples of vertices of plaques of $\widetilde\lambda$ that contain a vertex of $\mathcal C_i$.  Let $\mathfrak B$ denote the set of bases of $\Cb^d$. For all ${\bf x}=(x_1,x_2,x_3)$, we denote ${\bf x}_+:=(x_2,x_3,x_1)$, ${\bf x}_-:=(x_3,x_1,x_2)$, and ${\bf x}^{\rm op}:=(x_3,x_2,x_1)$. Choose a map 
		\[{\bf v}_0=(v_{0,1},\dots,v_{0,d}):\widetilde \Delta^o(\mathcal C_0)\to\mathfrak B\] 
		such that
		\begin{enumerate}
			\item For each ${\bf x}\in\widetilde\Delta^o(\mathcal C_0)$ that is oriented clockwise, $({\bf v}_0({\bf x}_-),{\bf v}_0({\bf x}),{\bf v}_0({\bf x}_+))$ is a triple of bases compatible with $(\xi({\bf x}),r)$ in the sense of Definition \ref{def: compatible}.
			\item For each ${\bf x}\in\widetilde\Delta^o(\mathcal C_0)$, ${\bf v}_0({\bf x}^{\rm op})={\bf v}_0({\bf x})^{\rm op}$, i.e.
			\[v_{0,m}(x_3,x_2,x_1)=v_{0,d-m+1}(x_1,x_2,x_3)\] 
			for all $m\in\{1,\dots,d\}$.
		\end{enumerate}
		By the equivalence between (1) and (2) of Corollary \ref{cor: rotate}, such a ${\bf v}_0$ exists. Then for each $i\in\{1,\dots,4g\}$, define the map
		\[{\bf v}_i:\widetilde \Delta^o(\mathcal C_i)\to\mathfrak B\]
		by ${\bf v}_i({\bf x}):=B_i\,{\bf v}_0(\omega_i^{-1}\,{\bf x})$, where $B_i:=A_1 A_2\ldots A_i$. The equivalence between (1) and (3) of Corollary \ref{cor: rotate} implies that for each ${\bf x}\in\widetilde\Delta^o(\mathcal C_i)$ that is oriented counterclockwise, $({\bf v}_i({\bf x}_-),{\bf v}_i({\bf x}),{\bf v}_i({\bf x}_+))$ is a triple of bases compatible with $(\xi({\bf x}),-r)$.
		
		\subsection{The sequence of bases along $\widetilde{\mathsf c}$ associated to $\rho$ and their slithering coefficients}\label{ssec:ctildetype} 
		To relate the enhanced cutting sequence ${\bf x}(0),\dots,{\bf x}(\ell)$ of $\widetilde{\mathsf c}$ and the families of bases ${\bf v}_0,\dots,{\bf v}_{4g}$ along $\widetilde{\mathsf b}$, we have to relate the type decomposition 
		\[\widetilde{\mathsf c}=\widetilde\ksf_1\cdot\ldots\cdot\widetilde\ksf_\ell\]
		of $\widetilde{\csf}$ defined in Section \ref{sec: type} to the graph decomposition 
		\[\widetilde{\mathsf b}=\widetilde{\mathsf f}_1\cdot\widetilde{\mathsf e}_1\cdot\widetilde{\mathsf f}_2\cdot\widetilde{\mathsf e}_2\cdot\ldots\cdot\widetilde{\mathsf f}_{4g}\cdot\widetilde{\mathsf e}_{4g}\cdot\widetilde{\mathsf f}_{4g+1}\]
		of $\widetilde{\bsf}$ described above. 
		
		Note that the type decomposition of $\widetilde{\csf}$ induces a type decomposition
		\[\widetilde{\mathsf b}=\widetilde\ksf'_1\cdot\ldots\cdot\widetilde\ksf'_\ell,\]
		where $\widetilde\ksf'_j:=\widetilde H(\cdot,1)\circ\widetilde{\ksf}_j$ for all $j\in\{1,\dots,\ell\}$, and $\widetilde H$ is the lift of the strong deformation retract $H$ to $\widetilde S\setminus \pi_S^{-1}(M)$. By choosing $H$ appropriately, we may further assume that if $x\in \overline{S\setminus M}$ is an endpoint of an exit of $M$, then for all $t\in[0,1]$, $H(x,t)$ lies in the horizontal boundary of $N$.  It then follows that for each $i\in\{1,\dots,4g\}$, there is a unique $j(i)\in\{1,\dots,\ell\}$ such that $\widetilde\ksf_{j(i)}$ is of rectangle type, and $\widetilde\ksf_{j(i)}'$ is an oriented subsegment of $\widetilde{\mathsf e}_i$, see Figure \ref{fig: junk}.
		Also, note that
		\[1\le  j(1)<j(2)<\dots<j(4g)\le\ell,\]
		and if we set $j(0):=0$ and $j(4g+1):=\ell+1$, then for all $i\in\{0,\dots,4g\}$, the plaques $T_{j(i)},T_{j(i)+1},\dots,T_{j(i+1)-1}$ all contain vertices of $\mathcal C_i$.
		
		With this, we may now define
		\begin{equation*}\label{eqn: v}
			{\bf v}=(v_1,\dots,v_d):\{0,\dots,\ell\}\to\mathfrak B
		\end{equation*} 
		by ${\bf v}(j):={\bf v}_{i}({\bf x}(j))$, where $i\in\{0,\dots,4g\}$ is the number such that $j(i)\leq j<j(i+1)$ (equivalently, such that ${\bf x}(j)$ is the triple of vertices of a plaque that contains a vertex of $\mathcal C_i$). Notice that the sequence ${\bf v}(0),\dots,{\bf v}(\ell)$ is not unique to $\widetilde{\mathsf c}$ and $\rho$; it depends on the choice of the lifts $A_1,\dots,A_{4g}$, the choice of the $\Gamma$--invariant function $r:\widetilde\Delta\to\Cb/2\pi i\Zb$, and the choice of the family of bases ${\bf v}_0$. 
		
		\begin{definition}\label{sequence of bases}
			A finite sequence ${\bf v}(0),\dots,{\bf v}(\ell)$ of bases of $\Cb^d$ is called a \emph{sequence of bases along $\widetilde{\mathsf c}$ associated to $\rho$} if it arises from the above construction for some choice of $A_1,\dots,A_{4g}$, $r$, and ${\bf v}_0$.
		\end{definition}
		
		We will now verify that the sequence of bases ${\bf v}(0),\dots,{\bf v}(\ell)$ satisfy conditions (I) and (II).
		
		\begin{proposition}\label{prop: properties of sequences}
			Let ${\bf v}(0),\dots,{\bf v}(\ell)$ be a sequence of bases along $\widetilde{\mathsf c}$ associated to $\rho$. Then
			\begin{enumerate}
				\item[(I)] $\exp\left({\rm ob}_d(\rho)\right)\,{\bf v}(0)={\bf v}(\ell)$.
				\item[(II)] For each $j\in\{1,\dots,\ell\}$ and $m\in\{1,\dots,d\}$, $\Sigma(g_j,g_{j-1})$ sends the vector $v_m(j-1)$ of the basis ${\bf v}(j-1)=(v_1(j-1),\dots,v_d(j-1))$ to a non-zero multiple of the vector $v_m(j)$ in the basis ${\bf v}(j)=(v_1(j),\dots,v_d(j))$.
			\end{enumerate}
		\end{proposition}
		
		\begin{proof}
			(I) Since $\omega_{4g}=\id$ and ${\bf x}(0)={\bf x}(\ell)$,
			\[{\bf v}(\ell)={\bf v}_{4g}({\bf x}(\ell))={\bf v}_{4g}(\omega_{4g}\, {\bf x}(0))=A_1\cdot\ldots\cdot A_{4g}\,{\bf v}_0({\bf x}(0))=\exp\left({\rm ob}_d(\rho)\right){\bf v}(0).\]
			The third equality holds by the definition of ${\bf v}_{4g}$, and the last equality holds by the definition of ${\rm ob}_d(\rho)$.

			(II) Recall that for all $j\in\{0,\dots,\ell\}$, if ${\bf x}(j)=(x_1(j),x_2(j),x_3(j))$, then $x_1(j)$ and $x_2(j)$ are respectively the forward and backward endpoints of $\mathsf g_j$, where $\mathsf g_j$ is the leaf $g_j$ of $\widetilde\lambda$ equipped with the orientation such that the ties of $\widetilde M$, with the orientation induced by the chosen orientation on the ties of $M$, pass from the left to the right of $\mathsf g_j$. Notice that for all $j\in\{1,\dots,\ell\}$, the pair $(\gsf_{j-1},\gsf_j)$ is coherently oriented (see Section \ref{sec: max geod}), so $\Sigma(g_j,g_{j-1})$ sends the flags $\xi(x_1(j-1))$ and $\xi(x_2(j-1))$ to $\xi(x_1(j))$ and $\xi(x_2(j))$ respectively, see (4) of Theorem \ref{thm: slithering map}. In particular, for all $m\in\{1,\dots,d\}$, $\Sigma(g_j,g_{j-1})$ sends the line $\xi(x_1(j-1))^m\cap\xi(x_2(j-1))^{d-m+1}$ to the line $\xi(x_1(j))^m\cap\xi(x_2(j))^{d-m+1}$. By construction, $v_m(j-1)$ and $v_m(j)$ are non-zero vectors in $\xi(x_1(j-1))^m\cap\xi(x_2(j-1))^{d-m+1}$ and $\xi(x_1(j))^m\cap\xi(x_2(j))^{d-m+1}$ respectively, so (II) holds.
		\end{proof}
		
\section{Computation of slithering coefficients}\label{sec: final}
Let $\rho$ be a $d$--pleated surface with pleating locus $\lambda$, let $N$ be our chosen train track neighborhood of $\lambda$ and let $M$ be our chosen maximal tree in $N$. 
Then let $\mathsf c$ be the boundary of $M$ based at the endpoint of an exit of $M$, oriented counterclockwise about $M$, and let $\widetilde{\mathsf c}$ be its lift to $\widetilde S$ that is based at the lift of the basepoint of $\mathsf c$ corresponding to our chosen identification between $\Gamma$ and the deck group of the universal cover $\pi_S:\widetilde S\to S$. See Section \ref{tor=ob} for more details on the definition of $\mathsf c$ and $\widetilde{\mathsf c}$.
Let ${\bf v}(0),\dots,{\bf v}(\ell)$ be a sequence of bases for $\widetilde{\mathsf c}$ associated to $\rho$, see Definition~\ref{sequence of bases}. Recall that in Section \ref{tor=ob}, we defined the \emph{slithering coefficients 
	\[\{a_m(j) \mid j\in\{1,\dots,\ell-1\}\text{ and }m\in\{1,\dots,d\}\}\]
	of ${\bf v}(0),\dots,{\bf v}(\ell)$}. 
We will now complete the final step of the proof of Theorem \ref{thm: final}, which is to show that the product
\[\log\left(\prod_{j=1}^\ell a_{\lfloor\frac{d+1}{2}\rfloor}(j)\right)\] 
is equal to the expression for ${\rm tor}_d(\rho)$ given in Remark \ref{tor formula}.

To state this more formally, recall the following:
\begin{itemize}
	\item $\Delta$ denotes the set of plaques of $\lambda$, and $\widetilde\Delta^o$ denotes the set of orderings of the vertices of the plaques of $\widetilde{\lambda}$ (see Section \ref{sec: max geod}),
	\item $\mathcal U^\ell$ and $\mathcal U^r$ respectively denote the sets of left unorientable rectangles and right unorientable rectangles for $M$ (see Section \ref{ssub:max tree}).
	\item $\mathcal S^\ell$ and $\mathcal S^r$ respectively denote the sets of  left vertical boundary components and right vertical boundary components of $M$ (see Section  \ref{ssub:max tree}).
	\item $\Bc^*:=\left\{{\bf j}=(j_1,j_2,j_3)\in\Bc:j_1,j_2,j_3\le \left\lfloor \frac{d-1}{2}\right\rfloor\right\}$ (see Section \ref{sec:topology_y} and Figure \ref{fig:boundary-rectangles}).
	\item  When $d$ is even,
	${\bf i}^0:=(\frac{d}{2},\frac{d}{2})$ and $\Bc^0:=\{{\bf j}=(j_1,j_2,j_3)\in\Bc:j_2=\frac{d}{2}\}$ (see Section \ref{sec:topology_y} and Figure \ref{fig:boundary-rectangles}).
\end{itemize}

We prove:
\begin{proposition}\label{cor final} Let $\rho$ be a $d$--pleated surface with pleating locus $\lambda$, and let $(\alpha,\theta)$ be the shear-bend $\lambda$--cocyclic pair associated to $\rho$, i.e. $(\alpha,\theta):=\mathfrak{sb}_d([\rho])$.
	\begin{enumerate}
		\item[i)] If $d$ is odd, then
		\[
		\log\left(\prod_{j=1}^{\ell}a_{\frac{d+1}{2}}(j)\right)=\sum_{{\bf j}\in\mathcal B^*}\ \sum_{T\in\Delta}\theta^{\bf j}({\bf x}_T),
		\]
		where ${\bf x}_T\in\widetilde\Delta^o$ is some (any) clockwise ordering of the vertices of some (any) lift to $\widetilde S$ of $T$.
		\item[ii)] If $d$ is even, then
		\begin{align*}
			\displaystyle \log\left(\prod_{j=1}^{\ell}a_{\frac{d}{2}}(j)\right)=&\sum_{R\in\mathcal U^\ell}\alpha^{{\bf i}^0}({\bf T}_R)-\sum_{R\in\mathcal U^r}\alpha^{{\bf i}^0}({\bf T}_R)\\&+\sum_{{\bf j}\in\mathcal B^*}\ \sum_{T\in\Delta}\theta^{\bf j}({\bf x}_T)+\sum_{{\bf j}\in \mathcal B^0}\sum_{t \in\mathcal S^\ell}\theta^{\bf j}({\bf x}_t),
		\end{align*}
	\end{enumerate}
	where:
	\begin{itemize}
		\item ${\bf x}_T\in\widetilde\Delta^o$ is as in part i),
		\item ${\bf T}_R\in \wt\Delta^{2*}$ is some (any) ordering of the pair of plaques of $\widetilde\lambda$ that contains the horizontal boundary components of some (any) lift of $R$ to $\widetilde S$,
		\item ${\bf x}_t=(x_{t,1},x_{t,2},x_{t,3})\in\widetilde\Delta^o$ is the clockwise ordering of the vertices of the plaque of $\widetilde\lambda$ that contains some (any) lift $\widetilde t$ to $\widetilde S$ of $t$, such that the geodesic with $x_{t,1}$ and $x_{t,3}$ as its endpoints does not intersect the switch of $\widetilde N$ that contains $\widetilde t$.
	\end{itemize}
\end{proposition}

\subsection{Proof of Proposition \ref{cor final}}
The proof of Proposition \ref{cor final} requires three main ingredients. 

The first is a formula for each slithering coefficient $a_m(j)$ in terms of $(\alpha,\theta)$, as described in the following lemma. Recall that to specify the sequence ${\bf v}(0),\dots,{\bf v}(\ell)$ of bases along $\widetilde{\mathsf c}$ associated to $\rho$, we had to choose a $\Gamma$--invariant map $r\colon \widetilde\Delta\to \mathbb C/2\pi i\mathbb Z$ such that
\[
3r(T)=\sum_{{\bf j}\in\mathcal B}\theta^{\bf j}({\bf x}_T),
\]
where ${\bf x}_T\in\widetilde\Delta^o$ is some (any) clockwise ordering of the vertices of some (any) lift to $\widetilde S$ of $T$. Recall also that $\mathcal E^\ell$ and $\mathcal E^r$ respectively denote the set of left and right exits of $M$ (see Section  \ref{ssub:max tree}), and that 
\[\mathsf c:=\ksf_1\cdot\ldots\cdot\ksf_\ell\]
is the type decomposition of $\mathsf c$ (see Section \ref{sec: type}).

\begin{lemma}\label{lem: final switch} 
	Let $j\in \{1,\dots, \ell\}$.  
	\begin{enumerate}
		\item[(i)] If $\ksf_j$ is of leaf type, then for every $m\in\{1,\dots, d\}$, $a_m(j)=1$.
		\item[(ii)] If $\ksf_j$ is of switch type, let $t_j$ be the vertical boundary component of $N$ parametrized by $\ksf_j$. Then for every $m\in\{1,\dots, d\}$ we have
		\[
		\displaystyle a_m(j)=\left\{\begin{array}{ll}
			\displaystyle(-1)^{m-1}\exp\left(-2r(T_j)+\sum_{{\bf j}\in\mathcal B\colon j_2\le m-1}\theta^{\bf j}({\bf x}_{t_j}) \right)&\text{if }t_j\in\mathcal S^r,\\
			\displaystyle(-1)^{d-m}\exp\left(-2r(T_j)+\sum_{{\bf j}\in\mathcal B\colon j_2\le d-m}\theta^{\bf j}({\bf x}_{t_j}) \right)&\text{if }t_j\in\mathcal S^\ell.
		\end{array}\right.
		\]
		\item[(iii)] If $\ksf_j$ is of rectangle type, let $s_j$ be the exit of $M$ parametrized by $\ksf_j$, and set ${\bf T}_j:=(T_{j-1},T_j)$, the pair of plaques of $\widetilde\lambda$ whose vertices are ${\bf x}(j-1)$ and ${\bf x}(j)$ respectively. Then for every $m\in\{1,\dots, d\}$ we have:
		\[a_m(j)=\left\{\begin{array}{ll}
			\displaystyle (-1)^{m-1}a_1(j)\prod_{{\bf i}\in \Ac\colon i_1\le m-1}\exp(\alpha^{\bf i}({\bf T}_j))&\text{if }s_j\in\mathcal E^r,\\
			\displaystyle (-1)^{m-1}a_1(j)\prod_{{\bf i}\in \Ac\colon i_2\le m-1}\exp(-\alpha^{\bf i}({\bf T}_j))&\text{if }s_j\in\mathcal E^\ell.
		\end{array}\right.
		\]
	\end{enumerate}
\end{lemma}

The second required ingredient is the following lemma, which relates $a_m(j)$ and $a_m(j')$ when $\ksf_j$ and $\ksf_j'$ are of rectangle type and lie in the same rectangle of $N$ that does not lie in $M$.

\begin{lemma}\label{lem: amj on a rectangle}
	If $j,j'\in\{1,\dots,\ell\}$ are distinct integers such that $\ksf_j$ and $\ksf_{j'}$ are both of rectangle type and lie in a common rectangle $R$ of $N$, then for all $m\in\{1,\dots,d\}$,
	\begin{align*}
		a_m(j)a_m(j')=\left\{\begin{array}{ll}
			1&\text{if }R\in\mathcal O,\\
			(-1)^{d-1}\displaystyle\prod_{{\bf i}\in\mathcal A\colon m\leq i_1\le d-m}\exp\left(-\alpha_\rho^{\bf i}({\bf T}_R)\right)&\text{if }R\in\mathcal U^r\text{ and }m\le\lfloor\frac{d+1}{2}\rfloor,\\
			(-1)^{d-1}\displaystyle\prod_{{\bf i}\in\mathcal A\colon d-m+1\leq i_1\le m-1}\exp\left(\alpha_\rho^{\bf i}({\bf T}_R)\right)&\text{if }R\in\mathcal U^r\text{ and }m\ge\lceil\frac{d+1}{2}\rceil,\\
			(-1)^{d-1}\displaystyle\prod_{{\bf i}\in\mathcal A\colon m\le i_1\leq d-m}\exp\left(\alpha_\rho^{\bf i}({\bf T}_R)\right)&\text{if }R\in\mathcal U^\ell\text{ and }m\le\lfloor\frac{d+1}{2}\rfloor,\\
			(-1)^{d-1}\displaystyle\prod_{{\bf i}\in\mathcal A\colon d-m+1\le i_1\leq m-1}\exp\left(-\alpha_\rho^{\bf i}({\bf T}_R)\right)&\text{if }R\in\mathcal U^\ell\text{ and }m\ge\lceil\frac{d+1}{2}\rceil,\\
		\end{array}\right.
	\end{align*} 
	where ${\bf T}_R$ is some (any) ordering of the pair of plaques that contain the horizontal boundary components of $R$. In particular, if $d$ is odd, then
	\[a_{\frac{d+1}{2}}(j)a_{\frac{d+1}{2}}(j')=1,\]
	and if $d$ is even, then
	\begin{align*}
		a_{\frac{d}{2}}(j)a_{\frac{d}{2}}(j')=\left\{\begin{array}{ll}
			1&\text{if }R\in\mathcal O,\\
			-\exp\left(-\alpha^{\bf i_0}({\bf T}_R)\right)&\text{if }R\in\mathcal U^r,\\
			-\exp\left(\alpha^{\bf i_0}({\bf T}_R)\right)&\text{if }R\in\mathcal U^\ell.
		\end{array}\right.
	\end{align*} 
\end{lemma}

The third and final ingredient is the following lemma about the parity of $\abs{\mathcal U}+\abs{\mathcal S^r}$. Recall that $\mathcal U$ denotes the set of unorientable rectangles for $M$. 

\begin{lemma}\label{prop: count} $\abs{\mathcal U}+\abs{\mathcal S^r}$ is even.
\end{lemma}

Assuming Lemmas \ref{lem: final switch}, \ref{lem: amj on a rectangle}, and \ref{prop: count}, we prove Proposition \ref{cor final}.

\begin{proof}[Proof of Proposition \ref{cor final}] 
	First, we decompose $\{1,\dots,\ell\}= \mathfrak L \sqcup \mathfrak S \sqcup \mathfrak R$, where $\mathfrak L$, $\mathfrak S$, and $\mathfrak R$ are the sets of integers $j\in\{1,\dots,\ell\}$ such that $\ksf_j$ is of leaf type, switch type, and rectangle type respectively. By (i) of Lemma \ref{lem: final switch}, for any $m\in\{1,\dots,d\}$, 
	\begin{align}\label{eqn: product}
		\prod_{j=1}^{\ell}a_m(j)=\left(\prod_{j\in\mathfrak R}a_m(j)\right)\left(\prod_{j\in\mathfrak S}a_m(j)\right).
	\end{align}
	When $m=\lfloor\frac{d+1}{2}\rfloor$ (i.e.  $m=\frac{d+1}{2}$ and $\frac{d}{2}$ when $d$ is odd and even respectively), we will compute the two products on the right hand side of equation \eqref{eqn: product} separately.
	
	First, we compute $\prod_{j\in\mathfrak R}a_{\lfloor\frac{d+1}{2}\rfloor}(j)$. Observe that for any rectangle $R$ of $N$ that does not lie in $M$, there are exactly two integers $j,j'\in\{1,\dots,\ell\}$ such that $\ksf_j$ and $\ksf_{j'}$ are both of rectangle type and lie in $R$. It then follows from Lemma \ref{lem: amj on a rectangle} that when $d$ is odd, we have
	\begin{align}\label{eqn: odd3}
		\prod_{j\in\mathfrak R}a_{\frac{d+1}{2}}(j)=1,
	\end{align}
	and when $d$ is even, we have
	\begin{align}\label{eqn: even3}
		\prod_{j\in\mathfrak R}a_{\frac{d}{2}}(j)=(-1)^{|\mathcal U|}\exp\left(\sum_{R\in\mathcal U^\ell}\alpha^{{\bf i}_0}({\bf T}_R)-\sum_{R\in\mathcal U^r}\alpha^{{\bf i}_0}({\bf T}_R)\right).
	\end{align}
	
	Next, we compute $\prod_{j\in\mathfrak S}a_{\lfloor\frac{d+1}{2}\rfloor}(j)$. Lemma~\ref{lem: final switch} part (ii) implies that when $d$ is odd and $\ksf_j$ is of switch type, we have
	\[a_{\frac{d+1}{2}}(j)=(-1)^{\frac{d-1}{2}}\exp\left(-2r(T_j)+\sum_{{\bf j}\in\mathcal B\colon j_2\le\frac{d-1}{2}} \theta^{\bf j}({\bf x}_{t_j}) \right),\]
	and when $d$ is even and $\ksf_j$ is of switch type, we have
	\[a_{\frac{d}{2}}(j)=\left\{\begin{array}{ll}
		\displaystyle (-1)^{\frac{d}{2}-1}\exp\left(-2r(T_j)+\sum_{{\bf j}\in\mathcal B\colon j_2\le\frac{d}{2}-1} \theta^{\bf j}({\bf x}_{t_j}) \right)&\text{if }t_j\in\mathcal S^r,\\
		\displaystyle (-1)^{\frac{d}{2}}\exp\left(-2r(T_j)+\sum_{{\bf j}\in\mathcal B\colon j_2\le\frac{d}{2}} \theta^{\bf j}({\bf x}_{t_j}) \right)&\text{if }t_j\in\mathcal S^\ell.
	\end{array}\right.
	\]
	Observe that the path $\mathsf c$ passes through every vertical boundary component of $N$ exactly once. In particular, $|\mathfrak S|=|\mathcal S|=12g-12$ and
	\begin{align*}
		\sum_{j\in\mathfrak S}r(T_j)=3\sum_{T\in\Delta}r(T)=\sum_{T\in\Delta}\sum_{{\bf j}\in\mathcal B}\theta^{\bf j}({\bf x}_T).
	\end{align*} 
	Therefore, when $d$ is odd, we have
	\begin{align}\label{eqn: odd1}
		\prod_{j\in\mathfrak S}a_{\frac{d+1}{2}}(j)&=\exp\left(-2\sum_{T\in\Delta}\sum_{{\bf j}\in\mathcal B}\theta^{\bf j}({\bf x}_T)+\sum_{t\in\mathcal S}\sum_{{\bf j}\in\mathcal B\colon j_2\le\frac{d-1}{2}} \theta^{\bf j}({\bf x}_t) \right),
	\end{align}
	and when $d$ is even, we have	
	\begin{align}\label{eqn: even1}
		\prod_{j\in\mathfrak S}a_{\frac{d}{2}}(j)
		&=(-1)^{|\mathcal S^r|}\exp\left(-2\sum_{T\in\Delta}\sum_{{\bf j}\in\mathcal B}\theta^{\bf j}({\bf x}_T)+\sum_{t\in\mathcal S}\sum_{{\bf j}\in\mathcal B\colon j_2\le\frac{d}{2}-1} \theta^{\bf j}({\bf x}_t)\right.\\
		&\qquad+\left.\sum_{t\in\mathcal S^\ell}\sum_{{\bf j}\in\mathcal B\colon j_2=\frac{d}{2}} \theta^{\bf j}({\bf x}_t) \right).\nonumber
	\end{align}
	\begin{figure}[h!]
		\includegraphics[width=0.5\textwidth]{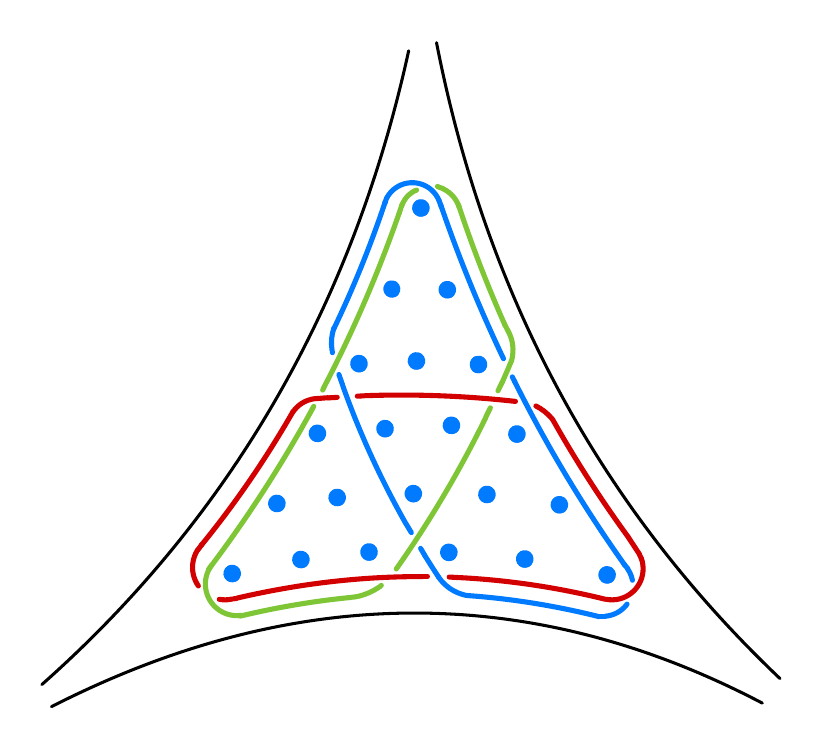}\includegraphics[width=0.5\textwidth]{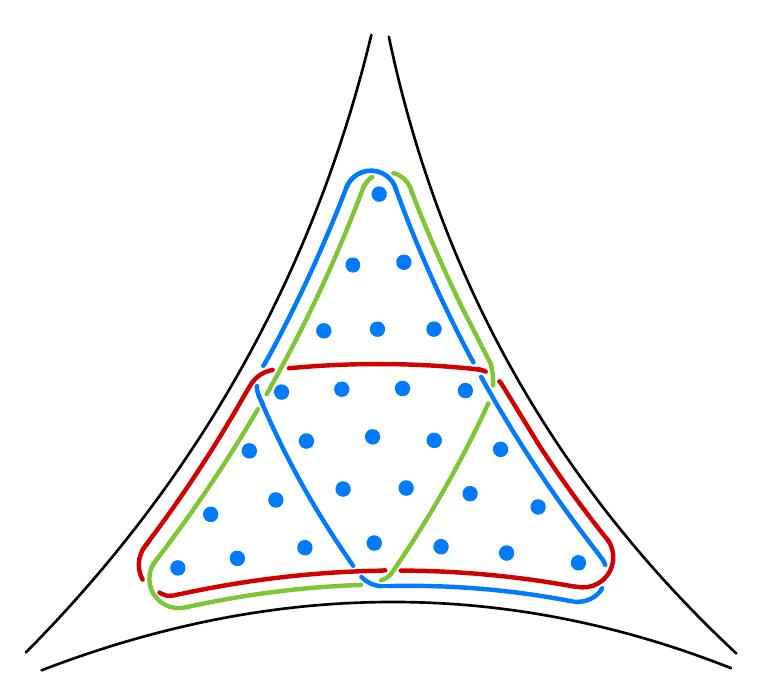}
		\caption{\label{fig: trapezoids} \small
			The dots in the blue, red, and green trapezoids are the triples $(j_1,j_2,j_3)\in\mathcal B$ with $j_1\le\lfloor\frac{d-1}{2}\rfloor$, $j_2\le\lfloor\frac{d-1}{2}\rfloor$, and $j_3\le\lfloor\frac{d-1}{2}\rfloor$ respectively ($d = 8$ on the left, $d=9$ on the right).}
	\end{figure}
	Notice that if $t, t', t''\in\mathcal S$ are the three vertical boundary components of $N$ that lie in a plaque $T$ of $\lambda$, then we have:
	\begin{align*}
		&\quad\sum_{{\bf j}\in\mathcal B\colon j_2\le\lfloor\frac{d-1}{2}\rfloor}\left(\theta^{\bf j}({\bf x}_{t})+\theta^{\bf j}({\bf x}_{t'})+\theta^{\bf j}({\bf x}_{t''})\right)\\
		&=\sum_{{\bf j}\in\mathcal B\colon j_1\le\lfloor\frac{d-1}{2}\rfloor}\theta^{\bf j}({\bf x}_T)+\sum_{{\bf j}\in\mathcal B\colon j_2\le\lfloor\frac{d-1}{2}\rfloor}\theta^{\bf j}({\bf x}_T)+\sum_{{\bf j}\in\mathcal B\colon j_3\le\lfloor\frac{d-1}{2}\rfloor}\theta^{\bf j}({\bf x}_T)\\
		&=3\sum_{\bf j\in\Bc}\theta^{\bf j}({\bf x}_T)-\sum_{{\bf j}\in\mathcal B\colon j_1>\lceil\frac{d+1}{2}\rceil}\theta^{\bf j}({\bf x}_T)-\sum_{{\bf j}\in\mathcal B\colon j_2>\lceil\frac{d+1}{2}\rceil}\theta^{\bf j}({\bf x}_T)-\sum_{{\bf j}\in\mathcal B\colon j_3>\lceil\frac{d+1}{2}\rceil}\theta^{\bf j}({\bf x}_T)\\
		&=2\sum_{\bf j\in\Bc}\theta^{\bf j}({\bf x}_T)+\sum_{{\bf j}\in\mathcal B^*}\theta^{\bf j}({\bf x}_T),
	\end{align*}
	see Figure \ref{fig: trapezoids}. It follows that
	\[\sum_{t\in\mathcal S}\sum_{{\bf j}\in\mathcal B\colon j_2\le\lfloor\frac{d-1}{2}\rfloor} \theta^{\bf j}({\bf x}_t) =\sum_{T\in\Delta}\left(2\sum_{\bf j\in\Bc}\theta^{\bf j}({\bf x}_T)+\sum_{{\bf j}\in\mathcal B^*}\theta^{\bf j}({\bf x}_T)\right),\]
	so the equations \eqref{eqn: odd1} and \eqref{eqn: even1} respectively reduce to
	\begin{align}\label{eqn: odd2}
		\prod_{j\in\mathfrak S}a_{\frac{d+1}{2}}(j)&=\exp\left(\sum_{T\in\Delta}\sum_{{\bf j}\in\mathcal B^*}\theta^{\bf j}({\bf x}_T) \right),
	\end{align}
	and
	\begin{align}\label{eqn: even2}
		\prod_{j\in\mathfrak S}a_{\frac{d}{2}}(j)
		&=(-1)^{|\mathcal S^r|}\exp\left(\sum_{T\in\Delta}\sum_{{\bf j}\in\mathcal B^*}\theta^{\bf j}({\bf x}_T)+\sum_{t\in\mathcal S^\ell}\sum_{{\bf j}\in\mathcal B\colon j_2=\frac{d}{2}} \theta^{\bf j}({\bf x}_t) \right).
	\end{align}
	
	Together, equations \eqref{eqn: product}, \eqref{eqn: odd3}, and \eqref{eqn: odd2} imply that when $d$ is odd,
	\[\prod_{j=1}^{\ell}a_{\frac{d+1}{2}}(j)=\exp\left(\sum_{{\bf j}\in\mathcal B^*}\ \sum_{T\in\Delta}\theta^{\bf j}({\bf x}_T)\right),\]
	while equations \eqref{eqn: product}, \eqref{eqn: even3}, and \eqref{eqn: even2} imply that when $d$ is even,
	\begin{align*}
		\displaystyle \prod_{j=1}^{\ell}a_{\frac{d}{2}}(j)&=(-1)^{|\mathcal U|+|\mathcal S^r|}\exp\Bigg(\sum_{R\in\mathcal U^\ell}\alpha^{{\bf i}^0}({\bf T}_R)-\sum_{R\in\mathcal U^r}\alpha^{{\bf i}^0}({\bf T}_R)\\&\quad+\sum_{{\bf j}\in\mathcal B^*}\ \sum_{T\in\Delta}\theta^{\bf j}({\bf x}_T)+\sum_{{\bf j}\in \mathcal B^0}\sum_{t \in\mathcal S^\ell}\theta^{\bf j}({\bf x}_t)\Bigg)\\
		&=\exp\Bigg(\sum_{R\in\mathcal U^\ell}\alpha^{{\bf i}^0}({\bf T}_R)-\sum_{R\in\mathcal U^r}\alpha^{{\bf i}^0}({\bf T}_R)\\&\quad+\sum_{{\bf j}\in\mathcal B^*}\ \sum_{T\in\Delta}\theta^{\bf j}({\bf x}_T)+\sum_{{\bf j}\in \mathcal B^0}\sum_{t \in\mathcal S^\ell}\theta^{\bf j}({\bf x}_t)\Bigg),
	\end{align*}
	where the second equality follows from Lemma \ref{prop: count}.
\end{proof}

It remains to prove the three main ingredients of the proof of Proposition \ref{cor final}. 

\subsection{Proof of Lemma \ref{lem: final switch}}
Let $(g_0,\dots,g_\ell)$ and $({\bf x}(0),\dots,{\bf x}(\ell))$ respectively denote the cutting sequence and enhanced cutting sequence of $\widetilde{\mathsf c}$, and for all $j\in\{0,\dots,\ell\}$, let  ${\bf x}(j)=(x_1(j),x_2(j),x_3(j))$. Recall that in Section \ref{tor=ob}, we used $M$ to construct a graph $\mathcal G\subset S$, and chose a maximal tree $\mathcal G'\subset\mathcal G$. Recall also that in Section \ref{sec: families} we constructed, using $\widetilde{\mathsf c}$, a sequence $(\mathcal C_0,\dots,\mathcal C_{4g})$ of connected components of $\pi_S^{-1}(\mathcal G')$, and defined for each $i\in\{0,\dots,4g\}$, a map ${\bf v}_i$ that assigns a basis of $\Cb^d$ to every ordering of the vertices of a plaque of $\widetilde\lambda$ that contains a vertex of $\mathcal C_i$. We will now prove the three parts of Lemma \ref{lem: final switch} separately.

\begin{proof}[Proof of (i) of Lemma \ref{lem: final switch}] Notice that if $\ksf_j$ is of leaf type, then ${\bf x}(j-1)={\bf x}(j)$. Thus, if $T_j$ and $T_{j-1}$ denote the plaques of $\widetilde\lambda$ whose vertices are ${\bf x}(j)$ and ${\bf x}(j-1)$ respectively, then $T_j=T_{j-1}$. Let $i\in\{1,\dots,4g\}$ be the unique integer such that $T_j=T_{j-1}$ contains a vertex of $\mathcal C_i$. Then by definition,
	\[{\bf v}(j-1)={\bf v}_i({\bf x}(j-1))={\bf v}_i({\bf x}(j))={\bf v}(j).\]
	Also, note that $g_{j-1}=g_j$, so $\Sigma(g_j,g_{j-1})=\id$. It now follows from the definition of the slithering coefficients that $a_m(j)=1$ for all $m\in\{1,\dots,d\}$, see equation \eqref{eqn: slithering coefficients}. 
\end{proof}

\begin{proof}[Proof of (ii) of Lemma \ref{lem: final switch}]
	Let $\mathsf g_{j-1}$ be the oriented geodesic in $\widetilde S$ whose forward and backward endpoints are $x_1(j-1)$ and $x_3(j-1)$ respectively, and let $\mathsf g_j$ be the oriented geodesic in $\widetilde S$ whose forward and backward endpoints are $x_1(j)$ and $x_3(j)$ respectively.  Recall that $\pi_S:\widetilde S\to S$ is the covering map, and $\widetilde M$ is the connected component of $\pi_S^{-1}(M)$ bounded by $\widetilde{\mathsf c}$. By the definition of the enhanced cutting sequence, the ties of $\widetilde M$, with the orientations induced by the chosen orientation on the ties of $M$, pass from left to right of both $\mathsf g_{j-1}$ and $\mathsf g_j$, so $x_{t_j,2}$ is the common forward (respectively, backward) endpoint of $\mathsf g_{j-1}$ and $\mathsf g_j$ if and only if $t_j$ is a right (respectively, left) vertical boundary component of $M$. Set 
	\[{\bf F}=(F_1,F_2,F_3):=\xi({\bf x}_{t_j})=(\xi(x_{t_j,1}),\xi(x_{t_j,2}),\xi(x_{t_j,3})).\]
	
	\begin{figure}[h!]
		\resizebox{0.95\textwidth}{!}{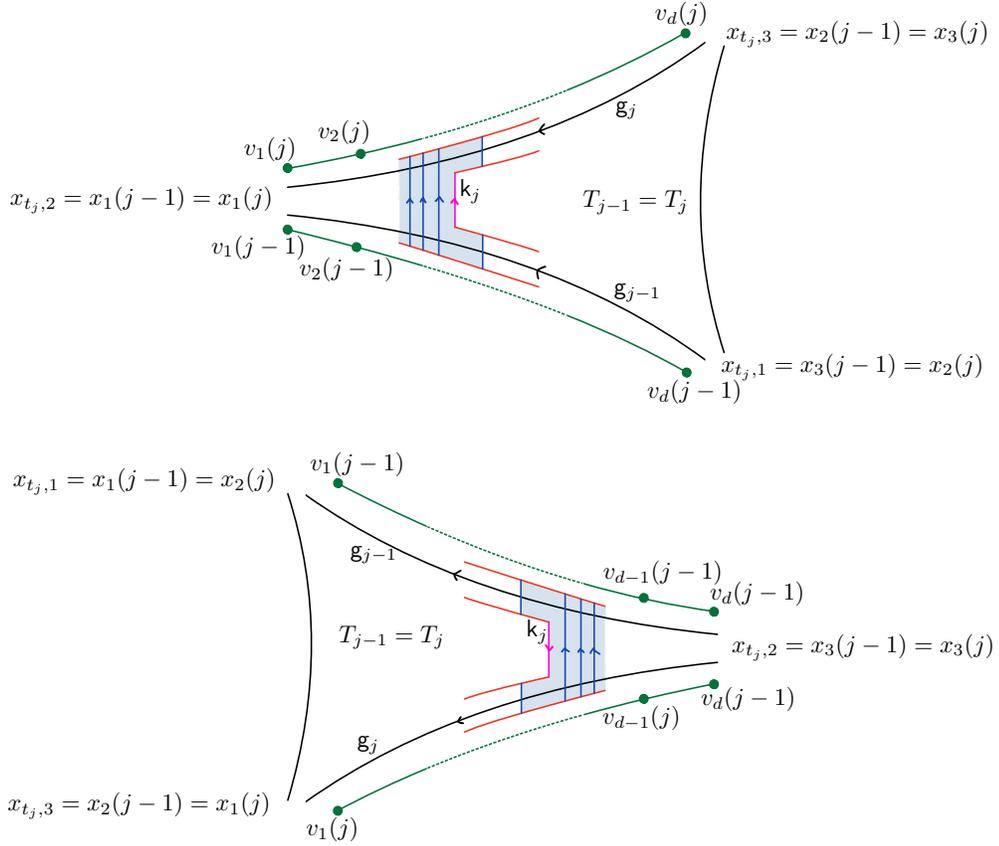}
		\caption{\small In the top (respectively, bottom) picture, $T_{j-1} = T_j$, $\gsf_{j-1}$, $\gsf_j$, ${\bf x}(j-1)$, ${\bf x}(j)$, and ${\bf x}_{t_j}$ are drawn in the case when $\ksf_j$ is a right (respectively, left) vertical boundary component of $M$.}
		\label{fig:dfjkg} 
	\end{figure}
	
	Since ${\bf x}_{t_j}$ is oriented clockwise, it follows that (see Figure \ref{fig:dfjkg}):
	\begin{itemize}
		\item if $t_j$ is a right vertical boundary component of $M$, then $\xi({\bf x}(j-1))=(F_2,F_3,F_1)$ and $\xi({\bf x}(j))=(F_2,F_1,F_3)$,
		\item if $t_j$  is a left vertical boundary component of $M$, then $\xi({\bf x}(j-1))=(F_1,F_3,F_2)$ and $\xi({\bf x}(j))=(F_3,F_1,F_2)$.
	\end{itemize}
	
	Thus, from the definitions of ${\bf v}(j-1)=(v_1(j-1),\dots,v_d(j-1))$ and ${\bf v}(j)=(v_1(j),\dots,v_d(j))$, we see that (see Figure \ref{fig:dfjkg}):
	\begin{itemize}
		\item if $t_j$ is a right vertical boundary component of $M$, then the basis ${\bf v}(j-1)$ is adapted to $(F_2,F_3,F_1)$, the basis ${\bf v}(j)^{\rm op} = (v_d(j),\dots,v_1(j))$ is adapted to $(F_3,F_1,F_2)$, and 
		$v_1(j)=\exp(2r(T_j))\,v_1(j-1)$.
		\item if $t_j$  is a left vertical boundary component of $M$, then the basis ${\bf v}(j-1)^{\rm op} = (v_d(j-1),\dots,v_1(j-1))$ is adapted to $(F_2,F_3,F_1)$, the basis ${\bf v}(j)$ is adapted to $(F_3,F_1,F_2)$, and $v_d(j)=\exp(2r(T_j))\,v_d(j-1)$.
	\end{itemize}

	Set
	\[{\bf f}=(f_1,\dots,f_d):=\left\{\begin{array}{ll}
		(v_1(j-1),\dots,v_d(j-1))&\text{if }t_j\in\mathcal S^r,\\
		(v_d(j-1),\dots,v_1(j-1))&\text{if }t_j\in\mathcal S^\ell,
	\end{array}\right.\]
	and set
	\[{\bf f}'=(f'_1,\dots,f'_d):=\left\{\begin{array}{ll}
		(\exp(-2r(T_j)) v_d(j),\dots,\exp(-2r(T_j)) v_1(j))&\text{if }t_j\in\mathcal S^r,\\
		(\exp(-2r(T_j)) v_1(j),\dots,\exp(-2r(T_j)) v_d(j))&\text{if }t_j\in\mathcal S^\ell.
	\end{array}\right.\]
	Then ${\bf f}$ and ${\bf f}'$ are adapted to $(F_2,F_3,F_1)$ and $(F_3,F_1,F_2)$ respectively, and $f_d'=f_1$. Let $u$ denote the unique unipotent matrix that fixes $F_2$ and sends $F_1$ to $F_3$. Applying Proposition \ref{prop: unipotent} to the triple of flags ${\bf F}$, we obtain that for all $m\in\{1,\dots, d\}$,
	\[u(f_m)=(-1)^{m-1}\exp\left(\sum_{{\bf j}\in\mathcal  B\colon j_2\le m-1}\tau^{\bf j}({\bf F})\right)f_{d-m+1}'.\]
	
	Notice that $\Sigma(g_j,g_{j-1})=u$ (see (3) and (4) of Theorem \ref{thm: slithering map}), and that $\theta^{\bf j}({\bf x}) = \tau^{\bf j}(\xi({\bf x}))$ for any ${\bf x}\in\widetilde\Delta^o$ and ${\bf j}\in\Bc$. Thus, from the definition of ${\bf f}$ and ${\bf f}'$, we have
	\[\Sigma(g_j,g_{j-1})\, v_m(j-1)=\left\{\begin{array}{ll}
		\displaystyle(-1)^{m-1}\exp\left(-2r(T_j)+\sum_{{\bf j}\in\mathcal  B\colon j_2\le m-1}\theta^{\bf j}({\bf x}_{t_j})\right)v_m(j)&\text{if }t_j\in\mathcal S^r,\\
		\displaystyle(-1)^{d-m}\exp\left(-2r(T_j)+\sum_{{\bf j}\in\mathcal  B\colon j_2\le d-m}\theta^{\bf j}({\bf x}_{t_j})\right)v_m(j)&\text{if }t_j\in\mathcal S^\ell.
	\end{array}\right.\]
	Now, (ii) follows from the definition of the slithering coefficients, see equation \eqref{eqn: slithering coefficients}.
\end{proof}

\begin{proof}[Proof of (iii) of Lemma \ref{lem: final switch}]
	Recall that for all ${\bf i}=(i_1,i_2)$ such that $i_1+i_2 = d$, the quantity $\alpha^{\bf i}({\bf T}_j)$ is the logarithm of the $\bf i$--double ratio, denoted $\sigma^{\bf i}$, of the flags obtained by applying the $\lambda$--limit map $\xi$ to the vertices of the plaques in ${\bf T}_j$ labeled according to the conventions established in  Section \ref{d-pleated}. More precisely, let $g_j$ (resp. $g_{j-1}$) denote the leaf of $\widetilde\lambda$ whose endpoints are $x_1(j-1)$ and $x_3(j-1)$ (resp. $x_1(j)$ and $x_3(j)$). Then we have the following two cases (see Figure \ref{fig:doubleratio}):
	\begin{itemize}
		\item if $s_j$ is a right exit, then for all $\bf i\in\mathcal A$
		\[
		\alpha^{\bf i}({\bf T}_j)=\sigma^{\bf i}\big(\xi(x_1(j)),\xi(x_3(j)), \Sigma(g_j,g_{j-1})\,\xi(x_2(j-1)),\xi(x_2(j))\big)
		\]
		\item if $s_j$ is a left exit, then for all $\bf i\in\mathcal A$
		\[
		\alpha^{\bf i}({\bf T}_j)=\sigma^{\bf i}\big(\xi(x_3(j)),\xi(x_1(j)), \Sigma(g_j,g_{j-1})\,\xi(x_2(j-1)),\xi(x_2(j))\big).
		\]
	\end{itemize}
	
	\begin{figure}[h!]
		\resizebox{.7\textwidth}{!}{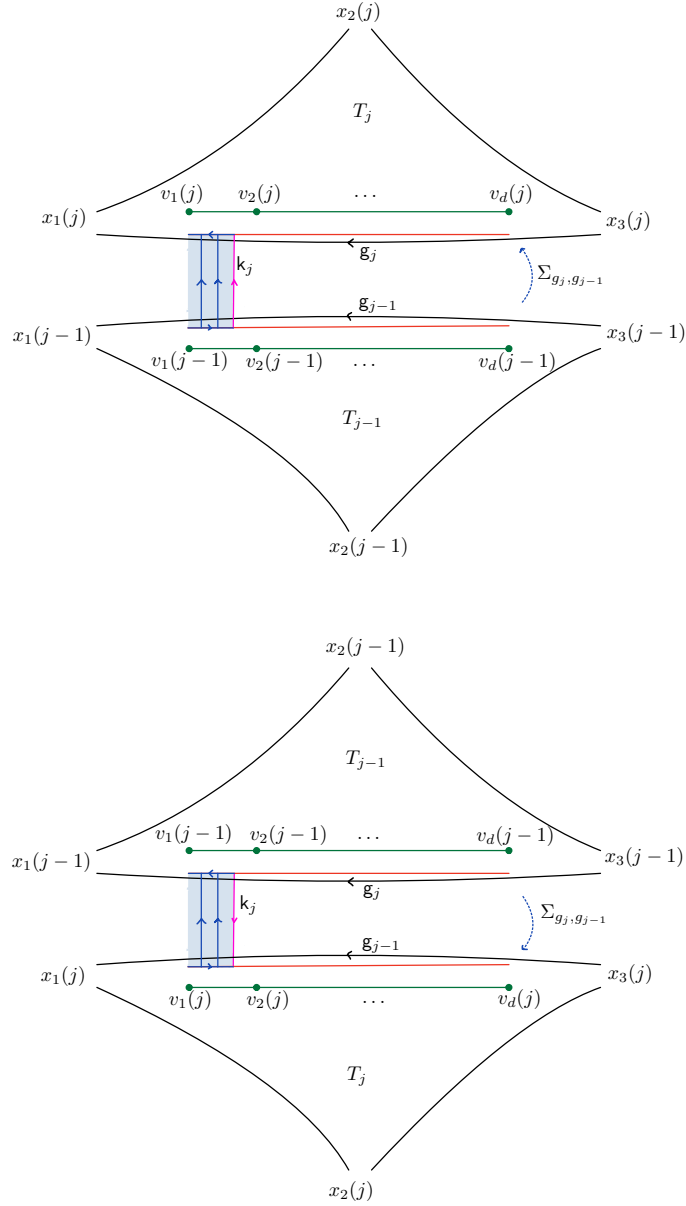}
		\caption{\label{fig:doubleratio} \small In the top (respectively, bottom) picture, $T_{j-1}$, $T_j$, $\gsf_{j-1}$, $\gsf_j$, ${\bf x}(j-1)$, and ${\bf x}(j)$ are drawn in the case when $\ksf_j$ is a right (respectively, left) exit of $M$.}
	\end{figure}
	
	Let ${\bf v}(j)=(v_1(j),\dots,v_d(j))$. Notice that for all $m\in\{1,\dots,d\}$, $(v_1(j),\dots,v_m(j))$ is a basis for $\xi(x_1(j))^m$, $(v_d(j),\dots,v_{d-m+1}(j))$ is a basis for $\xi(x_3(j))^m$, and $\sum_{m=1}^dv_m(j)$ spans $\xi(x_2(j))^1$. Also, from the definition of the slithering coefficients, see equation \eqref{eqn: slithering coefficients}, we have that
	\[
	\Sigma(g_j,g_{j-1})\left(\sum_{m=1}^dv_m(j-1)\right)=\sum_{m=1}^da_m(j)v_m(j)
	\] 
	spans $\Sigma(g_j,g_{j-1})\,\xi(x_2(j-1))^1$. Using the definition of $\sigma^{\bf i}$ given in Section \ref{sec: flags}, we may thus compute, that if $s_j$ is a right exit of $M$,
	\begin{align*}
		\exp\left(\alpha^{{\bf i}}({\bf T}_j)\right)
		&=-\frac{v_{1}(j)\wedge\dots\wedge v_{i_1}(j)\wedge v_{d}(j)\wedge\dots\wedge v_{i_1+2}(j)\wedge\displaystyle\sum_{m=1}^da_m(j)v_m(j)}{v_{1}(j)\wedge\dots\wedge v_{i_1}(j)\wedge v_{d}(j)\wedge\dots\wedge v_{i_1+2}(j)\wedge\displaystyle\sum_{m=1}^dv_m(j)}\\
		&\quad\quad  \cdot\frac{v_{1}(j)\wedge\dots\wedge v_{i_1-1}(j)\wedge v_{d}(j)\wedge\dots\wedge v_{i_1+1}(j)\wedge\displaystyle\sum_{m=1}^dv_m(j)}{v_{1}(j)\wedge\dots\wedge v_{i_1-1}(j)\wedge v_{d}(j)\wedge\dots\wedge v_{i_1+1}(j)\wedge\displaystyle\sum_{m=1}^da_m(j)v_m(j)}\\
		&=-\frac{a_{i_1+1}(j)}{a_{i_1}(j)}.
	\end{align*}
	Similarly, if $s_j$ is a left exit of $M$, then
	\begin{align*}
		\exp\left(\alpha^{\bf i}({\bf T}_j)\right)
		&=-\frac{v_{d}(j)\wedge\dots\wedge v_{i_2+1}(j)\wedge v_{1}(j)\wedge\dots\wedge v_{i_2-1}(j)\wedge\displaystyle\sum_{m=1}^da_m(j)v_m(j)}{v_{d}(j)\wedge\dots\wedge v_{i_2+1}(j)\wedge v_{1}(j)\wedge\dots\wedge v_{i_2-1}(j)\wedge\displaystyle\sum_{m=1}^dv_m(j)}\\
		&\quad\quad  \cdot\frac{v_{d}(j)\wedge\dots\wedge v_{i_2+2}(j)\wedge v_{1}(j)\wedge\dots\wedge v_{i_2}(j)\wedge\displaystyle\sum_{m=1}^dv_m(j)}{v_{d}(j)\wedge\dots\wedge v_{i_2+2}(j)\wedge v_{1}(j)\wedge\dots\wedge v_{i_2}(j)\wedge\displaystyle\sum_{m=1}^da_m(j)v_m(j)}\\
		&=-\frac{a_{i_2}(j)}{a_{i_2+1}(j)}.
	\end{align*}
	Taking products of the above equations over all $i_1\in\{1,\dots,m-1\}$ (respectively, $i_2\in\{1,\dots,m-1\}$) when $s_j$ is a right exit (respectively, left exit) of $M$ proves (iii).
\end{proof}

\subsection{Proof of Lemma \ref{lem: amj on a rectangle}} To prove Lemma \ref{lem: amj on a rectangle}, we will use the following.

\begin{lemma}\label{lem: blah}
	If $j,j'\in\{1,\dots,\ell\}$ are distinct integers such that $\ksf_j$ and $\ksf_{j'}$ are both of rectangle type and lie in a common rectangle $R$, then for all $m\in\{1,\dots,d\}$,
	\[a_m(j)=\left\{\begin{array}{ll}
		1/a_m(j')&\text{if }R\text{ is orientable},\\
		1/a_{d-m+1}(j')&\text{if }R\text{ is unorientable}.
	\end{array}\right.\] 
\end{lemma}

\begin{proof}
	Recall that for all $j\in\{1,\dots,\ell\}$, $T_j$ denotes the plaque of $\widetilde\lambda$ whose vertices are ${\bf x}(j)$. Also, recall that $(\gamma_1,\dots,\gamma_{4g})$ denotes the relation sequence associated to $(\mathcal G,\mathcal G',\mathsf b)$, and $(\mathcal C_0,\dots,\mathcal C_{4g})$ denotes the sequence of connected component of $\pi_S^{-1}(\mathcal G')$ that we constructed from $\widetilde{\mathsf c}$ in Section \ref{sec: families}. If we denote by $\omega_0$ the identity element in $\Gamma$ and $\omega_i:=\gamma_1\ldots\gamma_i$ for each $i\in\{1,\dots,4g\}$, then we observed in Section \ref{sec: families} that $\mathcal C_i=\omega_i\,\mathcal C_0$. By definition, the sequence of bases ${\bf v}(0),\dots,{\bf v}(\ell)$ for $\widetilde{\mathsf c}$ associated to $\rho$ comes from a choice of linear maps $A_1,\dots,A_{4g}\in\SL_d(\Cb)$ such that $\rho(\gamma_i)$ is the projectivization of $A_i$ for all $i$, and $A_i=A_j^{-1}$ whenever $\gamma_i=\gamma_j^{-1}$. Set $B_0$ to be the identity in $\SL_d(\Cb)$ and $B_i:=A_1\dots A_i$ for each $i\in\{1,\dots,4g\}$.
	
	Let $i,i'\in\{0,\dots,4g\}$ be the integers such that $T_j$ contains a vertex of $\mathcal C_i$ and $T_{j'}$ contains a vertex of $\mathcal C_{i'}$. If both $T_{j-1}$ and $T_j$ contain a vertex of the same connected component of $\pi_S^{-1}(\mathcal G')$, then $\mathcal C_{i'}=\gamma\,\mathcal C_i$, where $\gamma=\omega_{i'}\omega_i^{-1}$. In this case, set
	\[B_\gamma:=B_{i'}B_i^{-1}\in\SL_d(\Cb).\]
	On the other hand, if $T_{j-1}$ and $T_j$ contain vertices of different connected components of $\pi_S^{-1}(\mathcal G')$, we have that $\mathcal C_{i'-1}=\gamma\,\mathcal C_i$ and $\mathcal C_{i'}=\gamma\,\mathcal C_{i-1}$, and so $\gamma=\omega_{i'-1}\omega_i^{-1}=\omega_{i'}\omega_{i-1}^{-1}$. Also, since $\ksf_j$ and $\ksf_{j'}$ are both of rectangle type and lie in a common rectangle, we have that $\gamma_i^{-1}=\gamma_{i'}$, and hence $A_i^{-1}=A_{i'}$, so we may set
	\[B_\gamma:=B_{i'-1}B_i^{-1}=B_{i'}B_{i-1}^{-1}\in\SL_d(\Cb).\]
	Since $B_k$ is a linear representative of $\rho(\omega_k)$ for all $k$ and $\gamma=\omega_{i'-1}\omega_i^{-1}=\omega_{i'}\omega_{i-1}^{-1}$, in both cases, we have that $B_\gamma$ is a linear representative of $\rho(\gamma)\in\PGL_d(\Cb)$.
	
	\begin{figure}[h!]
		\hspace*{-1.5cm}\resizebox{1.2\textwidth}{!}{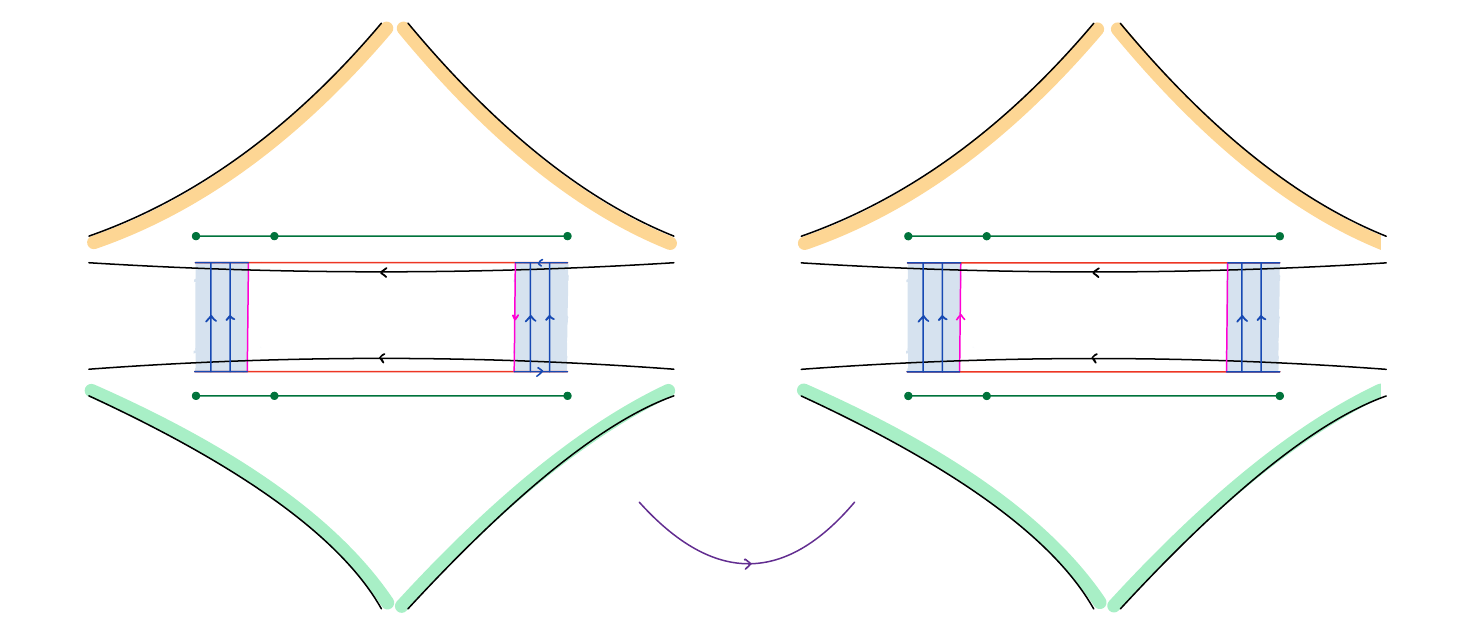}
		\vspace{.3cm}
		
		\hspace*{-1.5cm}\resizebox{1.2\textwidth}{!}{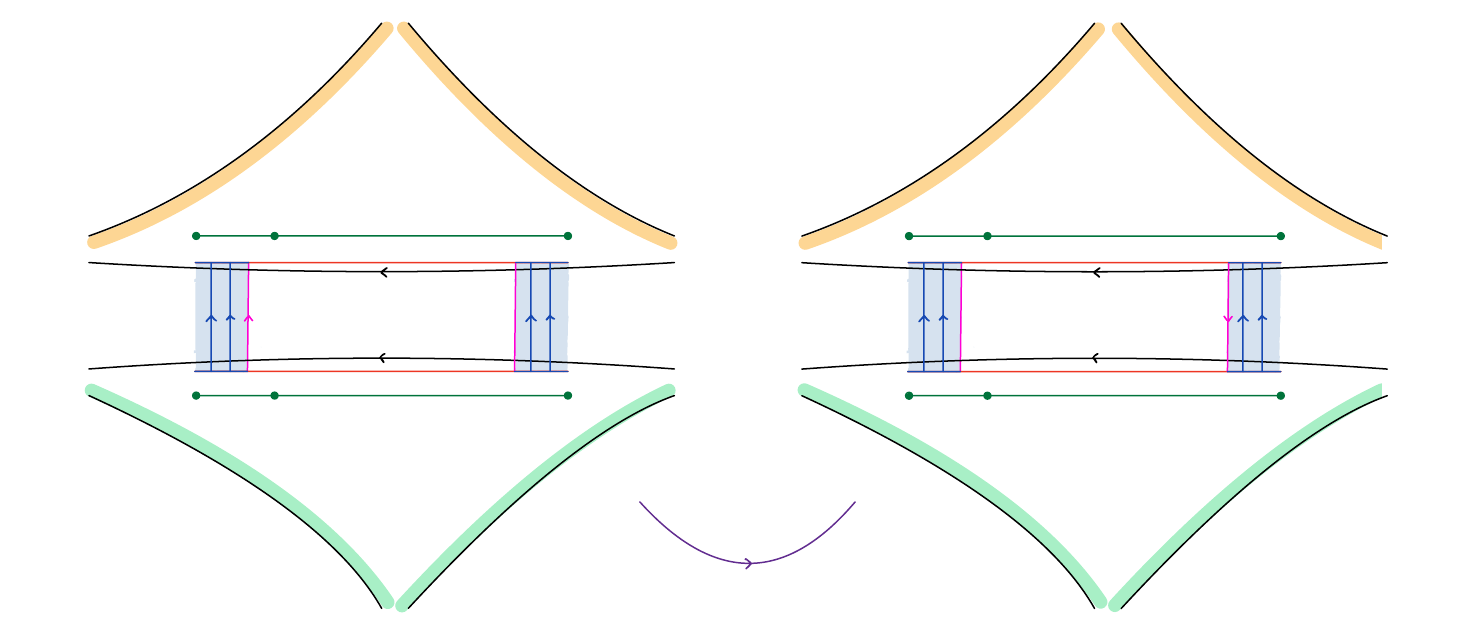}
		\caption{\label{fig: final cases} \small The rectangle $R$ is orientable and the maximal tree $M$ is shaded in blue. In the top (respectively, bottom) picture, $T_h$, $\gsf_h$, and ${\bf x}(h)$ are drawn for all $h\in\{j-1,j,j'-1,j'\}$ when $\ksf_j$ is a left (respectively, right) exit and $\ksf_{j'}$ is a right (respectively, left) exit of $M$.}
	\end{figure}

	For all $h\in\{j-1,j,j'-1,j'\}$, let $\mathsf g_h$ be the oriented geodesic in $\widetilde S$ whose forward and backward endpoints are $x_1(h)$ and $x_3(h)$ respectively. Recall that $\pi_S:\widetilde S\to S$ is the universal cover, and $\widetilde M$ is the connected component of $\pi_S^{-1}(M)$ bounded by $\widetilde{\mathsf c}$. By the definition of the enhanced cutting sequence ${\bf x}(0),\dots,{\bf x}(\ell)$ of $\widetilde{\mathsf c}$, the ties of $\widetilde M$, equipped with the orientation induced by the chosen orientation on the ties of $M$, pass from the left to the right of $\mathsf g_h$. Thus, 
	\begin{itemize}
		\item if $R$ is orientable, then (see Figure \ref{fig: final cases})
		\[\gamma\,{\bf x}(j-1)={\bf x}(j')\quad\text{and}\quad\gamma\, {\bf x}(j)={\bf x}(j'-1),\] 
		\item if $R$ is unorientable, then (see Figure \ref{fig: final cases2})
		\[\gamma\,{\bf x}(j-1)={\bf x}(j')^{\rm opp}\quad\text{and}\quad\gamma\, {\bf x}(j)={\bf x}(j'-1)^{\rm op},\] 
		where for any ${\bf x}=(x_1,x_2,x_3)\in\widetilde\Delta^o$, we denote ${\bf x}^{\rm op}:=(x_3,x_2,x_1)$. 
	\end{itemize}

	\begin{figure}[h!]
		\hspace*{-1.5cm}\resizebox{1.2\textwidth}{!}{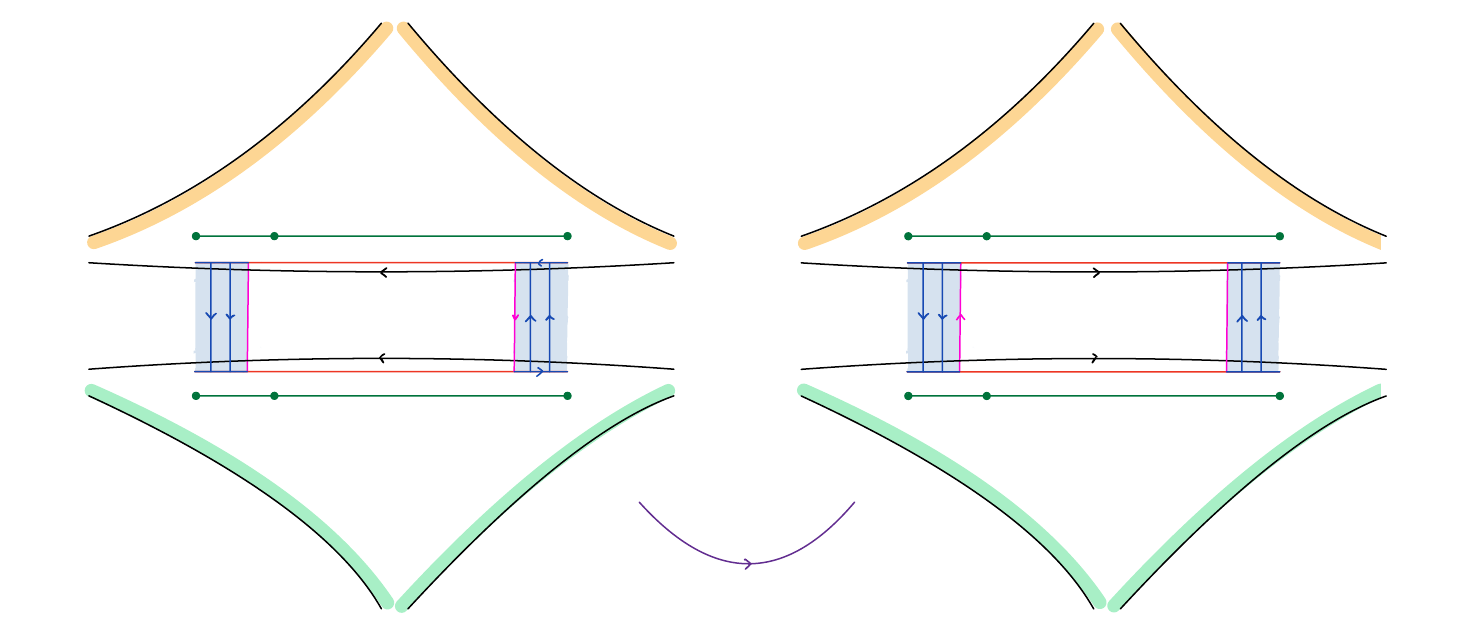}
		\vspace{.3cm}
		
		\hspace*{-1.5cm}\resizebox{1.2\textwidth}{!}{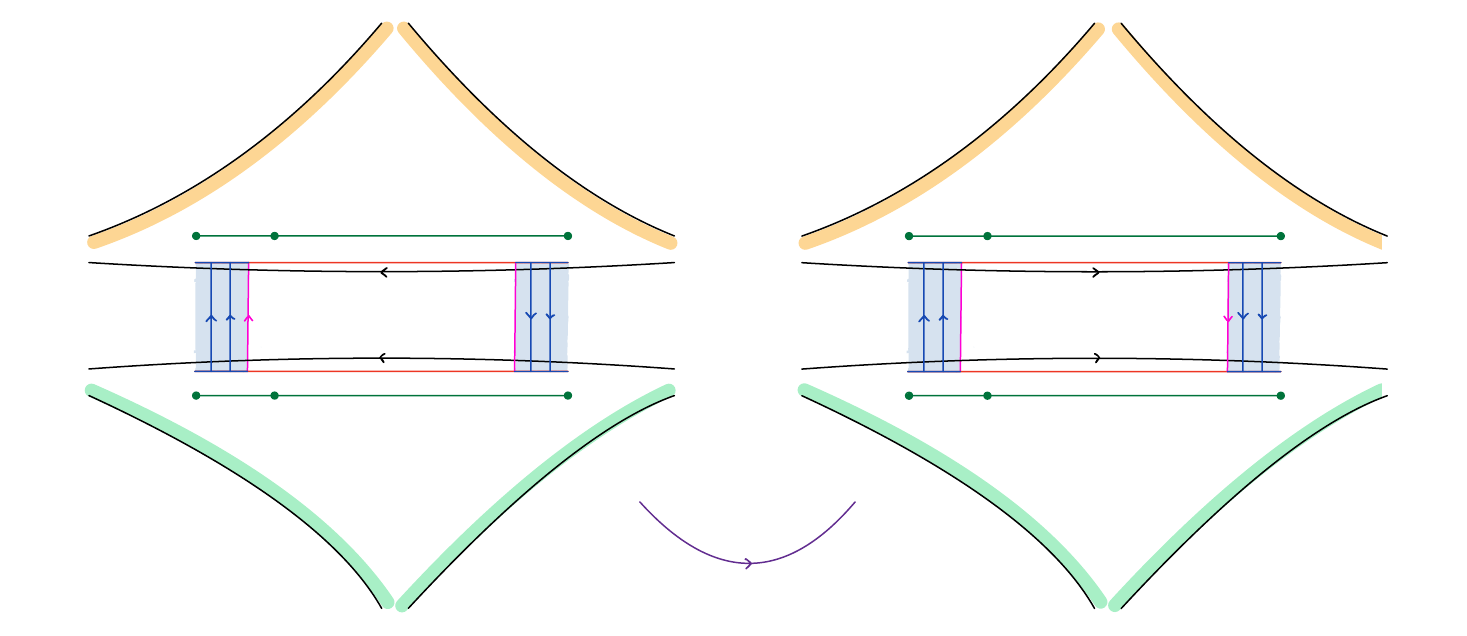}		
		\caption{\label{fig: final cases2} \small The rectangle $R$ is unorientable and the maximal tree $M$ is shaded in blue. In the top (respectively, bottom) picture, $T_h$, $\gsf_h$, and ${\bf x}(h)$ are drawn for all $h\in\{j-1,j,j'-1,j'\}$ when $R$ is a right (respectively, left) unorientable rectangle of $N$ that does not lie in $M$.}
	\end{figure}
	
	One can now verify from the definition of ${\bf v}$ that for all $m\in\{1,\dots,d\}$
	\[B_\gamma\, v_m(j-1)=\left\{\begin{array}{ll}
		v_m(j')&\text{if }R\text{ is orientable},\\
		v_{d-m+1}(j')&\text{if }R\text{ is unorientable},
	\end{array}\right.\]
	and
	\[B_\gamma\, v_m(j)=\left\{\begin{array}{ll}
		v_m(j'-1)&\text{if }R\text{ is orientable},\\
		v_{d-m+1}(j'-1)&\text{if }R\text{ is unorientable}
	\end{array}\right.\]
	We will verify the first equality; the argument for the second is similar.
	
	Recall that for each $i\in\{0,\dots,4g\}$, ${\bf v}_i=(v_{i,1},\dots,v_{i,d})$ is the map defined in Section \ref{sec: families} that assigns, to a triple in $\widetilde\Delta^o(\mathcal C_i)$ (which is the set of orderings of the vertices of the plaques of $\widetilde\lambda$ that contain a vertex of $\mathcal C_i$), a basis of $\Cb^d$. By definition, ${\bf v}_i({\bf x})=B_i\,{\bf v}_0(\omega_i\,{\bf x})$ for all ${\bf x}\in \widetilde\Delta^o(\mathcal C_i)$. Thus, by the definition of ${\bf v}(j-1)=(v_1(j-1),\dots,v_d(j-1))$ and ${\bf v}(j)=(v_1(j),\dots,v_d(j))$, we have that for all $m\in\{1,\dots,d\}$,  
	\begin{align*}
		B_\gamma\, v_m(j-1)&=\left\{\begin{array}{ll}
			B_\gamma\, v_{i,m}({\bf x}(j-1))&\text{if }T_{j-1}\text{ contains a vertex of }\mathcal C_i\\
			B_\gamma\, v_{i-1,m}({\bf x}(j-1))&\text{if }T_{j-1}\text{ contains a vertex of }\mathcal C_{i-1}\\
		\end{array}\right.\\
		&=v_{i',m}(\gamma\,{\bf x}(j-1))\\
		&=\left\{\begin{array}{ll}
			v_{i',m}({\bf x}(j'))&\text{if }R\text{ is orientable},\\
			v_{i',m}({\bf x}(j')^{\rm op})&\text{if }R\text{ is unorientable},
		\end{array}\right.\\
		&=\left\{\begin{array}{ll}
			v_m(j')&\text{if }R\text{ is orientable},\\
			v_{d-m+1}(j')&\text{if }R\text{ is unorientable}.
		\end{array}\right.
	\end{align*}
	
	Notice that $\gamma\, g_{j-1}=g_{j'}$ and $\gamma\, g_j=g_{j'-1}$, so by the $\rho$--equivariance of the slithering map,
	\[\Sigma(g_j,g_{j-1})
	=\rho(\gamma)^{-1}\circ\Sigma(g_{j'-1},g_{j'})\circ\rho(\gamma)=B_\gamma^{-1}\Sigma(g_{j'},g_{j'-1})^{-1} B_\gamma.\]
	Thus, by the definition of the slithering coefficients,
	\begin{align*}
		a_m(j)v_m(j)&=\Sigma(g_j,g_{j-1})\,v_m(j-1)\\
		&=\left\{\begin{array}{ll}
			B_\gamma^{-1}\Sigma(g_{j'},g_{j'-1})^{-1}\,v_m(j')&\text{if }R\text{ is orientable},\\
			B_\gamma^{-1}\Sigma(g_{j'},g_{j'-1})^{-1}\,v_{d-m+1}(j')&\text{if }R\text{ is unorientable},
		\end{array}\right.\\
		&=\left\{\begin{array}{ll}
			B_\gamma^{-1}\,v_m(j'-1)/a_m(j')&\text{if }R\text{ is orientable},\\
			B_\gamma^{-1}\,v_{d-m+1}(j'-1)/a_{d-m+1}(j')&\text{if }R\text{ is unorientable},
		\end{array}\right.\\
		&=\left\{\begin{array}{ll}
			v_m(j)/a_m(j')&\text{if }R\text{ is orientable},\\
			v_m(j)/a_{d-m+1}(j')&\text{if }R\text{ is unorientable}.
		\end{array}\right.
	\end{align*}
	This proves the lemma.
\end{proof}

We will now finish the proof of Lemma \ref{lem: amj on a rectangle}. 

\begin{proof}[Proof of Lemma \ref{lem: amj on a rectangle}] If $R$ is orientable, then Lemma \ref{lem: blah} implies that $a_m(j)a_m(j')=1$. Thus, we now only need to focus on the case when $R$ is unorientable. 
	
	If the rectangle $R$ is right unorientable, then the two exits of $M$ that lie in $R$ are both right exits. Thus, by Lemma \ref{lem: blah} and (iii) of Lemma \ref{lem: final switch}, we have that for all $m\in\{1,\dots, d\}$
	\begin{align*}
		a_m(j)a_m(j')&=\frac{a_m(j)}{a_{d-m+1}(j)}\\
		&=\frac{(-1)^{m-1}a_1(j)\displaystyle\prod_{{\bf i}\in\mathcal A\colon i_1\le m-1}\exp(\alpha_\rho^{\bf i}({\bf T}_j))}{(-1)^{d-m}a_1(j)\displaystyle\prod_{{\bf i}\in\mathcal A\colon i_1\le d-m}\exp(\alpha_\rho^{\bf i}({\bf T}_j))}\\
		&=(-1)^{d-1}\prod_{{\bf i}\in\mathcal A\colon i_1\le m-1}\exp(\alpha_\rho^{\bf i}({\bf T}_j))\prod_{{\bf i}\in\mathcal A\colon i_1\le d-m}\exp(-\alpha_\rho^{\bf i}({\bf T}_j))\\
		&=\begin{cases}
			(-1)^{d-1}\displaystyle\prod_{{\bf i}\in\mathcal A\colon m\leq i_1\le d-m}\exp\left(-\alpha_\rho^{\bf i}({\bf T}_j)\right)&\text{ if }m\le \lfloor\frac{d+1}{2}\rfloor,\\
			(-1)^{d-1}\displaystyle\prod_{{\bf i}\in\mathcal A\colon d-m+1\leq i_1\le m-1}\exp\left(\alpha_\rho^{\bf i}({\bf T}_j)\right)&\text{ if }m\ge \lceil\frac{d+1}{2}\rceil.
		\end{cases}
	\end{align*}
	Similarly, if $R$ is left unorientable, then the two exits of $M$ that lie in $R$ are both left exits, and a similar computation gives that for all $m\in\{1,\dots,d\}$, 
	\begin{align*}
		a_m(j)a_m(j')
		&=\begin{cases}
			(-1)^{d-1}\displaystyle\prod_{{\bf i}\in\mathcal A\colon m\leq i_1\leq d-m}\exp\left(\alpha_\rho^{\bf i}({\bf T}_{j})\right)&\text{ if }m \le \lfloor\frac{d+1}{2}\rfloor,\\
			(-1)^{d-1}\displaystyle\prod_{{\bf i}\in\mathcal A\colon d-m+1\leq i_1\leq m-1}\exp\left(-\alpha_\rho^{\bf i}({\bf T}_{j})\right)&\text{ if }m \ge \lceil\frac{d+1}{2}\rceil.
		\end{cases}
	\end{align*}
\end{proof}

\subsection{Proof of Lemma \ref{prop: count}} Recall that if $L\subset N$ is a tree with a continuous orientation on its ties, then $\mathcal S^r(L)$ and $\mathcal S^\ell(L)$ respectively denote the set of right and left vertical boundary components of $N$ that lie in $L$, and $\mathcal E^r(L)$ and $\mathcal E^\ell(L)$ respectively denote the set of right and left exits of $L$. To prove Lemma \ref{prop: count}, we will use the following.

\begin{lemma}\label{lem: count} Let $L\subset N$ be a tree equipped with a continuous orientation on the ties of $L$. Then $\abs{\mathcal E^r(L)}=1+\abs{\mathcal S^r(L)}$. In particular, $\abs{\mathcal E^r}=1+\abs{\mathcal S^r}$.
\end{lemma} 

\begin{proof} 
	The proof proceeds by induction on the number of truncated rectangles in $L$. First, consider the base case, when there are no truncated rectangles in $L$, i.e. $L$ is a stumpy switch. In this case, $L$ contains a unique vertical boundary component of $N$. If it is a left vertical boundary component, then $\abs{\mathcal E^r(L)}=1$ and $\abs{\mathcal S^r(L)}=0$. If it is a right vertical boundary component, then $\abs{\mathcal E^r(L)}=2$ and $\abs{\mathcal S^r(L)}=1$. In either case, the required equality holds.
	
	For the inductive step, we suppose that $L$ contains at least one truncated rectangle $\check R$. Let $L_1$ and $L_2$ be the disjoint trees in $N$ such that $L=L_1\cup L_2\cup \check R$. Notice that the orientation on the ties of $L$ restricts to continuous orientations on the ties of both $L_1$ and $L_2$. For both $i=1,2$, $L_i$ has fewer rectangles than $L$, so we may apply the inductive hypothesis to deduce that 
	\begin{equation}\label{eqn:1}
		\abs{\mathcal E^r(L_i)}=1+\abs{\mathcal S^r(L_i)}\quad\text{for both}\quad i=1,2.
	\end{equation}
	
	Notice that for both $i=1,2$, $\check R\cap L_i$ is an exit of $L_i$. Furthermore, $\check R\cap L_1$ is a left (respectively, right) exit of $L_1$ if and only $\check R\cap L_2$ is a right (respectively, left) exit of $L_2$. This implies that
	\begin{equation}\label{eqn:2}
		\abs{\mathcal E^r(L)}=\abs{\mathcal E^r(L_1)}+\abs{\mathcal E^r(L_2)}-1.
	\end{equation}
	On the other hand, notice that
	\begin{equation}\label{eqn:3}
		\abs{\mathcal S^r(L)}=\abs{\mathcal S^r(L_1)}+\abs{\mathcal S^r(L_2)}.
	\end{equation}
	Combining equations \eqref{eqn:1}, \eqref{eqn:2}, and \eqref{eqn:3} proves the inductive step.
\end{proof}

We will now prove Lemma \ref{prop: count}.

\begin{proof}[Proof of Lemma \ref{prop: count}] 
	Recall that $\mathcal N$ denotes the set of rectangles in $N$, $\mathcal M$ denotes the set of rectangles in $M$, and $\mathcal O$ and $\mathcal U$ denote the sets of orientable and unorientable rectangles in $N$ but not in $M$. In Section \ref{ssub:tt_neighborhood}, we observed that $\abs{\mathcal N}=18g-18$ and $\abs{\mathcal S}=12g-12$. Since $M$ is a maximal tree, the fact that $\abs{\mathcal S}=12g-12$ also implies that $\abs{\mathcal M}=12g-13$. Hence, 
	\begin{equation}\label{eqn:4}
		\abs{\mathcal O}+\abs{\mathcal U}=\abs{\mathcal N}-\abs{\mathcal M}=6g-5.
	\end{equation} 
	
	If $R$ is an orientable rectangle in $N\setminus M$, then note that $\check R\cap M$ is the union of one left exit and one right exit of $M$. On the other hand, if $R$ is a right (respectively, left) unorientable rectangle in $N\setminus M$, then $\check R\cap M$ is the union of two right (respectively, left) exits of $M$. From these observations, we deduce that 
	\begin{equation}\label{eqn:5}
		\abs{\mathcal E^r} = \abs{\mathcal O}+2\abs{\mathcal U^r}.
	\end{equation}
	Together, equations \eqref{eqn:4}, \eqref{eqn:5}, and Lemma \ref{lem: count} give
	\[\abs{\mathcal U}+\abs{\mathcal S^r}=\abs{\mathcal U}+\abs{\mathcal E^r}-1=\abs{\mathcal U}+\abs{\mathcal O}+2\abs{\mathcal U^r}-1=6g-6+2\abs{\mathcal U^r}\]
	which is even.
\end{proof}

\bibliographystyle{amsalpha}
\bibliography{biblio}
\end{document}